\documentclass[11pt]{amsart}
\pdfoutput=1

\newif\ifarxiv
\arxivtrue

\newif\ifsubmission
\submissionfalse

\newcommand{\arxivorsubmission}[2]{\ifarxiv #1\else #2\fi}
\newcommand{\arxivonly}[1]{\arxivorsubmission{#1}{}}

\usepackage[utf8]{inputenc}  

\ifarxiv
  \calclayout
\fi

\usepackage{amsmath}
\usepackage{amssymb}
\usepackage{amsthm}
\usepackage[mathscr]{eucal}

\usepackage{xcolor}
\definecolor{darkgreen}{rgb}{0,0.45,0}
\definecolor{darkred}{rgb}{0.75,0,0}
\definecolor{darkblue}{rgb}{0,0,0.6}
\usepackage[linktoc=all,colorlinks,citecolor=darkgreen,linkcolor=darkred,urlcolor=darkblue]{hyperref}
\usepackage{enumerate} 

\usepackage{mathpartir}

\usepackage{graphicx}
\usepackage{tikz}
\usepackage{tikz-cd}
\usepackage{mathtools} 
\usepackage{adjustbox} 

\usepackage{xifthen} 
\usepackage{ifmtarg} 
\usepackage{comment} 
\usepackage{filecontents}
\usepackage{refcount}

\theoremstyle{plain}
\newtheorem{theorem}{Theorem}[section]
\newtheorem{proposition}[theorem]{Proposition}
\newtheorem{lemma}[theorem]{Lemma}

\newtheorem{corollary}[theorem]{Corollary}

\newtheorem*{theorem*}{Theorem}

\theoremstyle{definition}
\newtheorem{definition}[theorem]{Definition}
\newtheorem{remark}[theorem]{Remark}
\newtheorem{notation}[theorem]{Notation}
\newtheorem{example}[theorem]{Example}

\newtheorem{construction}[theorem]{Construction}

\newenvironment{varnotation}[1] 
  {\edef\thetheorem{#1}%
   \addtocounter{theorem}{-1}
   \notation}
  {\endnotation}

\newcommand{\todo}[1]{} 
\newcommand{\optionaltodo}[1]{} 

\newcommand{\defemph}[1]{\textbf{#1}} 

\newenvironment{maxwidth}{\begin{adjustbox}{max width=\linewidth}}{\end{adjustbox}}

\newenvironment{maxwidthtikzpicture}[1][]
  {\begin{maxwidth}\begin{tikzpicture}[#1]}
  {\end{tikzpicture}\end{maxwidth}}

\newenvironment{maxwidthtikzcd}[1][]
  {\begin{maxwidth}\begin{tikzcd}[#1]}
  {\end{tikzcd}\end{maxwidth}}

\usetikzlibrary{calc}
\usetikzlibrary{fit}
\usetikzlibrary{shapes}
\usetikzlibrary{arrows}
\usetikzlibrary{decorations.markings}
\usetikzlibrary{decorations.pathmorphing}
\usetikzlibrary{positioning}
\usetikzlibrary{intersections}

\tikzset{
  commutative diagrams/arrow style=tikz,
  commutative diagrams/diagrams={row sep=large},
}

\tikzset{cd-style/.style={commutative diagrams/every diagram}}
\tikzset{cd-arrow-style/.style={commutative diagrams/.cd, every arrow, every label}}
\newcommand{\arr}[1][]{\draw[cd-arrow-style,#1]}

\makeatletter
\tikzset{
  label/.style n args={2}{
    edge node={node [
      execute at begin node=\iftikzcd@mathmode$\fi,
      execute at end node=\iftikzcd@mathmode$\fi,
      /tikz/commutative diagrams/.cd,every label,
      #2
      ] {#1}}
  }
}
\makeatother

\tikzset{fibtip/.tip={Triangle[open,angle=60:4.5pt]}}
\tikzset{tfibtip/.tip={Bar[sep]Triangle[open,angle=45:4pt]}} 

\pgfarrowsdeclaredouble{fibbtip}{fibbtip}{fibtip}{.fibtip}

\pgfarrowsdeclarealias{c}{c}{left hook}{right hook}
\pgfarrowsdeclarealias{c'}{c'}{right hook}{left hook}

\tikzset{drpb/.style={commutative diagrams/.cd, dr, phantom, "\lrcorner", very near start}}

\tikzset{fib/.code={\pgfsetarrowsend{fibtip}}}
\tikzset{fibb/.code={\pgfsetarrowsend{fibbtip}}}
\tikzset{tfib/.code={\pgfsetarrowsend{tfibtip}}}
\tikzset{inj/.code={\pgfsetarrowsstart{c}}}
\tikzset{tcof/.style={tail}}
\tikzset{zigzag/.style={commutative diagrams/rightsquigarrow}}
\tikzset{lw/.style={"lw"{#1}},lw/.default={very near end}}
\tikzset{lw'/.style={"lw"'{#1}},lw'/.default={very near end}}
\tikzset{weq/.style={"\sim"}}
\tikzset{weq'/.style={"\sim"'}}
\tikzset{sup/.style={label={#1}{auto=left,pos=1}}}  
\tikzset{sub/.style={label={#1}{auto=right,pos=1}}}

\newcommand{\generalto}[2]{ \mathrel{\mkern-1mu
  \tikz[baseline={([yshift=-0.55ex]a.south)}]{%
    \node[minimum width=1.5em,align=center,inner xsep=0.5ex,inner ysep=0.15ex] (a) {$\scriptstyle #2$};
    \draw[#1] (a.south west) -- (a.south east);}
 \mkern-1mu}}

\newcommand{\generalfrom}[2]{ \mathrel{\mkern-1mu
  \tikz[baseline={([yshift=-0.55ex]a.south)}]{%
    \node[minimum width=1.5em,align=center,inner xsep=0.5ex,inner ysep=0.15ex] (a) {$\scriptstyle #2$};
    \draw[#1] (a.south east) -- (a.south west);}
  \mkern-1mu}}

\renewcommand{\to}[1][]{ \generalto{->}{#1} }
\newcommand{\from}[1][]{ \generalfrom{->}{#1} }
\newcommand{\fibto}[1][]{ \generalto{fib}{#1} }
\newcommand{\fibbto}[1][]{ \generalto{fibb}{#1} }
\newcommand{\extto}[1][]{ \generalto{>->}{#1} }
\newcommand{\injto}[1][]{ \generalto{->,inj}{#1} }
\newcommand{\lwto}[1][]{ \generalto{->}{#1}_{\lw} }

\makeatletter
\newcommand{\weqto}[1][]{\@ifmtarg{#1}{\generalto{->}{\sim}}{\NotYetDefined}}
\newcommand{\weqfrom}[1][]{\@ifmtarg{#1}{\generalfrom{->}{\sim}}{\NotYetDefined}}
\makeatother
\newcommand{\zigzagto}[1][]{ \generalto{zigzag}{#1} }

\newenvironment{verticalhack}
  {\begin{array}[b]{@{}c@{}}\displaystyle}
  {\\\noalign{\hrule height0pt}\end{array}}

\newcommand{\Abar}{\bar{A}}
\renewcommand{\AA}{{\vec A}}

\newcommand{\BB}{{\vec B}}
\newcommand{\Bbar}{\bar{B}}

\newcommand{\C}{\mathbf{C}}

\newcommand{\CC}{{\vec C}}
\newcommand{\D}{\mathbf{D}}

\newcommand{\EE}{{\vec E}}

\newcommand{\I}{\mathcal{I}}
\newcommand{\J}{\mathcal{J}}

\newcommand{\N}{\mathbb{N}}
\newcommand{\T}{\mathbf{T}}
\newcommand{\W}{\mathcal{W}}

\newcommand{\yon}{\mathbf{y}}
\newcommand{\del}{\partial}
\newcommand{\bdry}[1]{{\del #1}}

\newcommand{\CwA}{\mathrm{CwA}}

\newcommand{\Set}{\mathrm{Set}}
\newcommand{\FinSet}{\mathrm{FinSet}}

\newcommand{\SpanCat}{\mathrm{Span}}
\newcommand{\WkMapCat}{\mathrm{WkMap}}
\newcommand{\EqvCat}{\mathrm{Eqv}}

\newcommand{\supfunctor}[2]{\ifthenelse{\equal{#2}{}}%
  {(-)^\mathrm{#1}}%
  {{#2}^\mathrm{#1}}%
}
\newcommand{\Span}[1][]{\supfunctor{Span}{#1}}

\newcommand{\elem}[2][]{\int_{#1} #2} 

\newcommand{\slice}[2]{( #2 \downarrow #1 ) }
\newcommand{\strictslice}[2]{\partial \slice{#1}{#2} }

\newcommand{\cotensor}[2]{#1 \mathbin{\hat{\pitchfork}} #2}

\DeclareMathOperator{\Cyl}{Cyl}
\newcommand{\op}{\mathrm{op}}
\DeclareMathOperator{\ob}{Ob}

\newcommand{\id}{\mathrm{id}}
\newcommand{\ev}{\mathsf{ev}}
\newcommand{\lw}{\mathrm{lw}} 
\newcommand{\str}{\mathrm{str}} 

\DeclareMathOperator{\cod}{cod}

\newcommand{\wkreedy}{\mathrm{wR}}
\newcommand{\reedy}{\mathrm{R}}
\DeclareMathOperator{\im}{im}

\newcommand{\iso}{\cong}
\renewcommand{\equiv}{\simeq}

\newcommand{\homot}{\sim}

\newcommand{\Id}{\mathsf{Id}}
\newcommand{\refl}{\mathsf{r}}

\newcommand{\Piext}{\mathsf{\Pi}_{\mathsf{ext}}}
\newcommand{\synsplit}{\mathsf{split}}
\newcommand{\synJ}{\mathsf{J}}
\newcommand{\pair}{\mathsf{pair}}
\newcommand{\rec}{\mathsf{rec}} 
\newcommand{\funext}{\mathsf{funext}}
\newcommand{\funextcomp}{\mathsf{funext\textsf{-}comp\textsf{-}prop}}

\newcommand{\ext}{\mathrm{ext}}

\newcommand{\Ty}{\mathrm{Ty}}

\newcommand{\barbar}[1]{\overline{\overline{#1}}}

\newcommand{\restr}[1]{|_{#1}} 

\newcommand{\descend}[2]{#1\! \urcorner \, #2} 

\newcommand{\disc}[1]{{#1}^{\circ}}

\newcommand{\copsh}[1]{\Set^{#1}}

\makeatletter
\newcommand{\jpaacite}[2][]{\@ifmtarg{#1}%
  {\cite{\arxivorsubmission{#2}{#2-for-jpaa}}}%
  {\cite[#1]{\arxivorsubmission{#2}{#2-for-jpaa}}}}
\makeatother

\begin{document}

\ifsubmission \begin{frontmatter} \fi
  
\title{Homotopical inverse diagrams in categories with attributes}

\ifarxiv
  \author[K.~Kapulkin]{Krzysztof Kapulkin}
  \address{Dept.\ of Mathematics\\The University of Western Ontario\\London, Ontario}
  
  \author[P.~LeF.~Lumsdaine]{Peter LeFanu Lumsdaine}
  \address{Dept.\ of Mathematics\\Stockholm University\\Stockholm, Sweden}
\else
  \author[1]{Krzysztof Kapulkin}
  \author[2]{Peter LeFanu Lumsdaine}
\fi
  
 
\begin{abstract}
 We define and develop the infrastructure of \emph{homotopical inverse diagrams} in categories with attributes.

  Specifically, given a category with attributes $\C$ and an ordered homotopical inverse category $\I$, we construct the category with attributes $\C^\I$ of \emph{homotopical diagrams} of shape $\I$ in $\C$ and \emph{Reedy types} over these;
  and we show how various logical structure ($\Pi$-types, identity types, and so on) lifts from $\C$ to $\C^\I$.
  This may be seen as providing a general class of diagram models of type theory.
   
  In a companion paper ``The homotopy theory of type theories'' (\arxivorsubmission{\href{http://arxiv.org/abs/1610.00037}{\path{arXiv:1610.00037}}}{Kapulkin, Lumsdaine, \emph{Advances in Mathematics}, 2018}), we apply the present results to construct semi-model structures on categories of contextual categories.
\end{abstract}

\ifarxiv \maketitle \fi

\setcounter{tocdepth}{1}
\setcounter{secnumdepth}{2}

\ifarxiv \tableofcontents \fi

\ifsubmission \end{frontmatter} \fi

\section{Introduction}

Diagram models are a well-established tool in both categorical logic and homotopical algebra.
For semantics of type theory (particularly intensional and homotopical type theory), however, diagram models are comparatively under-developed (with the exception of presheaf models).
The complication is that logical constructors in this setting are typically not strictly functorial, making it harder to lift them to diagram models.

Specific cases, however, such as spans \jpaacite{tonelli}, spreads \cite{martin-lof:spreads-seminars}, and the various categories considered in \jpaacite[\textsection 5]{kapulkin-lumsdaine:homotopy-theory-of-type-theories}, have shown that at least on certain domain categories, diagram models for intensional type theory should be viable.

Homotopy theory provides a well-established setup to deal with such non-functorial constructions and unify these special cases: the language of \emph{Reedy diagrams} on \emph{inverse categories}.

The first main contribution of the present paper is constructing Reedy diagram models on inverse categories in \emph{categories with attributes} (CwA’s), an algebraic formulation of type theories, and showing that these inherit various logical constructors from the original CwA.
Specifically, we show:

\begin{theorem*}
 Let $\C$ be a category with attributes, and $\I$ an inverse category.%
 \footnote{In fact the diagram model will depend on some extra structure on $\I$ --- certain \emph{orderings} --- but every inverse category admits such structure.}
 Then the diagram category $\C^\I$, together with the presheaf of \emph{Reedy types} over $\I$ in $\C$, is again a category with attributes.
 If in addition $\C$ carries identity, $\Sigma$-, unit-, or $\Pi$-types, possibly with functional extensionality, then so does $\C^\I$.
\end{theorem*}

Further mileage can be obtained by restricting to \emph{homotopical} diagrams, another tool borrowed from homotopy theory.

Specifically, given a category with a class of morphisms distinguished as “equivalences”, a diagram on it is \emph{homotopical} if it sends these equivalences to equivalences in the type-theoretic sense.
We show:

\begin{theorem*}
 Let $\C$ be a CwA with identity types, and $\I$ an inverse category equipped with a class of equivalences.
 Then the \emph{homotopical} diagrams and Reedy types forms a sub-CwA $\C^\I$ of the non-homotopical diagram CwA $\C^{\disc{\I}}$.
 If in addition $\C$ carries $\Sigma$-types or unit-types, then $\C^\I$ is closed under these in $\C^{\disc{\I}}$; and similarly for $\Pi$-types with functional extensionality, provided all maps in $\I$ are equivalences.
\end{theorem*}

Constructions along these lines are familiar from abstract homotopy theory, presented in terms of \emph{fibration categories} \arxivorsubmission{\cite{brown:abstract-homotopy-theory,radulescu-banu,szumilo:two-models}}{\cite{brown:abstract-homotopy-theory,radulescu-banu-for-jpaa,szumilo:two-models-for-jpaa}} or comparable settings.
An application of such constructions to type theory has previously been given by Shulman \jpaacite{shulman:inverse-diagrams}, using \emph{type-theoretic fibration categories}; see Remark~\ref{rmk:related-work}(\ref{item:shulman}) for comparison with the present work.

This paper originated as a spin-off of another paper of ours \jpaacite{kapulkin-lumsdaine:homotopy-theory-of-type-theories}, developing a left semi-model structure on the category of contextual categories.
Section 5 of that paper requires four specific diagram models, which we originally planned to give individually as they appeared.
However, the details became sufficiently lengthy and repetitive to write out (and to read!) that it seemed more worthwhile to break them out into a separate paper, and give the construction in generality.

We therefore make use in the present paper of results from Sections 1--4 of \jpaacite{kapulkin-lumsdaine:homotopy-theory-of-type-theories}, while Sections 5--6 in that paper make use of the constructions presented here.

\subsection*{Organization.}

We begin in Section~\ref{sec:background} by setting up some general background material on categories with attributes and logical structure on them.
There are no substantively novel ideas, but some aspects of the presentation are new --- for instance, the systematic development of elimination structures.

In Section~\ref{sec:inverse-diagrams}, we introduce inverse categories, and set up the CwA’s of Reedy types, along with an auxiliary infrastructure of \emph{Reedy limits}, for constructing matching objects and the like.
Having done this, we show in Section~\ref{sec:logical-structure} how logical structure on the base CwA $\C$ (identity types, $\Pi$-types, and so on) induces similar structure on the diagram CwA’s $\C^\I$.
Along the way, we also characterise elimination structures and equivalences in $\C^\I$.

In Section~\ref{sec:homotopical-diagrams}, we consider the restriction to CwA’s of \emph{homotopical} diagrams, and show when the logical structure restricts from plain to homotopical diagram CwA’s.
Lastly, in Section~\ref{sec:fibrations-and-shit}, we give conditions on a functor $u \colon \J \to \I$ for the induced map $\C^\I \to \C^\J$ to be a \emph{local fibration} or \emph{local equivalence} in the sense of \jpaacite[\textsection 4]{kapulkin-lumsdaine:homotopy-theory-of-type-theories}.

We conclude in Section~\ref{sec:conclusion} by summarising the main constructions of the earlier sections for quick reference, briefly surveying the connections with related work, and noting various possible generalisations that we have not covered in this paper.


\section{Type-theoretic background} \label{sec:background}

We recall in this section some background on categories with attributes, and logical structure on them.
Most of the material is standard, but some aspects of the presentation --- for instance, the systematic use of elimination structures --- are novel.
\optionaltodo{Reword this so it doesn’t echo what we said in the overview so much!}

\subsection{Categories with Attributes}

\begin{definition}
 A \defemph{category with attributes} (CwA) $\C$ consists of:
 \begin{enumerate}
  \item a category $\C$, with a chosen terminal object $1$;%
    \footnote{We include the terminal object for the sake of following the established definition.  However, it is irrelevant to the constructions of the present paper, which work without alteration if it is omitted.}
  \item a functor $\Ty \colon \C^\op \to \Set$;
  \item an assignment to each $A \in \Ty(\Gamma)$, an object $\Gamma.A \in \C$ (the \defemph{extension} of $\Gamma$ by $A$) and a map $p_A \colon \Gamma.A \to \Gamma$ (which we distinguish graphically as $\fibto$);
  \item for each $A \in \Ty(\Gamma)$ and $f \colon \Delta \to \Gamma$, a map $f.A \colon \Delta.f^*A \to \Gamma.A$ such that the following square is a pullback:
  \[\begin{tikzcd}
    \Delta.f^*A \ar[r, "f.A"] \ar[d, fib, "p_{f^*A}"'] \arrow[drpb] & \Gamma.A \ar[d, fib, "p_A"] \\
    \Delta \ar[r, "f"] & \Gamma
  \end{tikzcd}\]  
 \end{enumerate}
\end{definition}

As defined, categories with attributes are models for an evident essentially algebraic theory.
A map of categories with attributes is a homomorphism of such models: explicitly, a functor $F \colon \C \to \C'$ and transformation $F_\Ty \colon \Ty_\C \to \Ty_{\C'}\cdot F$, strictly preserving all the structure (chosen terminal object, context extension, and so on).
Write $\CwA$ for the category of categories with attributes.%
\footnote{It is very arguably more natural to consider the \emph{2-category} of CwA’s, with \emph{pseudo-maps} preserving context extension just up to isomorphism; but that is beyond the scope or needs of the present work.}

\begin{definition}
 A \defemph{comprehension category} consists of a category $\C$ together with a Grothendieck fibration $p \colon \T \to \C$ and functor $\chi \colon \T \to \C^\rightarrow$, such that $\mathrm{cod} \cdot \chi = p$, and sending $p$-cartesian arrows to pullback squares.
\end{definition}

  It is easy to see that  CwA's correspond to discrete comprehension categories (i.e., ones in which the fibration $p$ is discrete), via the correspondence between presheaves and discrete fibrations.%
\footnote{This does not contradict the perhaps more well-known fact that CwA’s correspond to \emph{full split} comprehension categories; compare how sets may be seen as corresponding to either discrete or codiscrete categories.}

  Elements of the sets $\Ty(\Gamma)$ in a CwA, or of the fibers $\T_\Gamma$ in a comprehension category, will be called \defemph{types over $\Gamma$}.
It is often also useful to consider more general \emph{context extensions}:

\begin{definition}
  A \defemph{context extension} $\AA$ of an object $\Gamma$ of a CwA $\C$ is a sequence $(A_1, \ldots, A_n)$, where $A_i \in \Ty(\Gamma.A_1.\ldots.A_{i-1})$.
\end{definition}

We will often write just \defemph{extensions}, when this is not ambiguous.
Context extensions form an evident presheaf $\Ty^*$ over $\C$, and indeed an alternative CwA structure on $\C$ \cite[1.3.1]{lumsdaine:thesis}, with $p_\AA$ given by the composite $p_{A_1}p_{A_2} \ldots p_{A_n} \colon \Gamma . \AA \to \Gamma$. %
We distinguish maps of the form $p_\AA$ diagrammatically as $\fibbto$.

\begin{definition} \label{def:fibrant-slice}
  Let $\C$ be a CwA.
  Then for each $\Gamma$ in $\C$, there is a CwA $\C(\Gamma)$, the \defemph{fibrant slice} over $\Gamma$, in which objects are extensions $\AA$ of $\Gamma$, maps $f \colon \AA \to \BB$ are maps $\Gamma.\AA \to \Gamma.\BB$ over $\Gamma$ in $\C$, and $\Ty_{\C(\Gamma)}(\AA) \coloneqq \Ty_\C(\Gamma.\AA)$.
  
  Moreover, a map $f \colon \Gamma' \to \Gamma$ induces an evident CwA map $f^* \colon \C(\Gamma) \to \C(\Gamma')$, and this forms a (strict) functor $\C(-) \colon \C^\op \to \CwA$.
\end{definition}

This justifies very liberal use of the notation for reindexing: for instance, for  $A \in \Ty(\Gamma)$, $B \in \Ty(\Gamma . A)$, and $f \colon \Gamma' \to \Gamma$, we will write just $f^*B$ for $(f.A)^*B \in \Ty(\Gamma.f^*A)$, and so on.
We will also often suppress “weakening” (that is, reindexing along a dependent projection $p_A$); so for instance, for $A, B \in \Ty(\Gamma)$, we will write just $\Gamma.A.B$ for $\Gamma.A.p_A^*B$.

By a \defemph{section} of a type $A \in \Ty(\Gamma)$ (sometimes for emphasis a section of $A$ \defemph{over $\Gamma$}), we mean a section of the dependent projection $p_A \colon \Gamma.A \to \Gamma$.

Categories with attributes correspond to the structural rules of dependent type theory; for the logical rules, one adds extra structure.

\begin{definition}
A \defemph{$\Pi$-type structure} on a CwA $\C$ consists of:
\begin{enumerate}
  \item for each $\Gamma \in \C$, $A \in \Ty(\Gamma)$, and $B \in \Ty(\Gamma.A)$, a type $\Pi[A,B] \in \Ty (\Gamma)$;
  \item for each such $\Gamma$, $A$, $B$, and section $b$ of $B$ over $\Gamma . A $, a section $\lambda(b)$ of $\Pi[A, B]$ over $\Gamma$;
  \item for each such $\Gamma$, $A$, $B$, a map $\ev_{A,B} \colon \Gamma.\Pi[A,B].A \to \Gamma . A . B$ over $\Gamma . A$;
 \item such that for $\Gamma$, $A$, $B$ as above and sections $a$ of $A$ and $b$ of $B$, we have $\ev_{A,B}(\lambda(b),a) = b a \colon \Gamma \to \Gamma.A.B$;
 \item all stable under reindexing: i.e.,\ for any $f \colon \Gamma' \to \Gamma$ and appropriate arguments as above,
   \[
     f^*(\Pi [A, B]) = \Pi [f^*A, f^*B], \hskip 0.5em plus 1em
     f^*\lambda(b) = \lambda({f^*b}), \hskip 0.5em plus 1em
     f^*(\ev_{A,B}) = \ev_{f^*A, f^*B}.
   \]
\end{enumerate}

  A $\Pi$-type structure satisfies the \defemph{$\Pi$-$\eta$ rule} if for any $\Gamma$, $A$, $B$ as above, the ``$\eta$-expansion'' map
\[ (p_{\Pi(A,B)} . \Pi(A,B)) \cdot \lambda (\id_{p_{\Pi(A,B)}^*A}, \ev_{A,B})  \colon  \Gamma . \Pi(A,B) \to \Gamma . \Pi(A,B) \]
is equal to the identity map $\id_{\Gamma . \Pi(A,B)}$.
  For brevity, we call a $\Pi$-structure satisfying this condition a \defemph{$\Pi_\eta$-structure}.
\end{definition}

\begin{definition}
A \defemph{$\Sigma$-type structure} on a CwA $\C$ consists of:
\begin{enumerate}
  \item for each $\Gamma \in \C$, $A \in \Ty(\Gamma)$, and $B \in \Ty(\Gamma.A)$, a type $\Sigma[A,B] \in \Ty (\Gamma)$;
  
  \item for each such $\Gamma$, $A$, $B$, a map $\pair_{A,B} \colon \Gamma . A . B \to \Gamma . \Sigma[A,B]$ over $\Gamma$;
  
  \item for each $\Gamma$, $A$, $B$ as above,
  extension $\Delta \in \Ty^*(\Gamma . \Sigma[A,B])$,
  type $C \in \Ty(\Gamma . \Sigma[A,B] . \Delta)$,
  and map $d \colon \Gamma . A . B . (\pair_{A,B})^* \Delta \to \Gamma . \Sigma[A,B] . \Delta . C$ 
  over $\pair_{A,B} . \Delta \colon \Gamma . A . B . \allowbreak (\pair_{A,B})^* \Delta \allowbreak \to \Gamma . \Sigma[A,B] . \Delta$,
  a section $\synsplit_{\Delta,C,d}$ of $C$, 
  such that  $\synsplit_{\Delta,C,d} (\pair_{A,B}.\Delta) = d$;

  \item all stable under reindexing: for $f \colon \Gamma' \to \Gamma$ and arguments as above,
   \begin{gather*}
     f^*(\Sigma[A, B]) = \Sigma [f^*A, f^*B], \\
     f^*\pair_{A,B} = \pair_{f^*A,f^*B}, \qquad
     f^*\synsplit_{\Delta,C,d} = \synsplit_{f^*\Delta,f^*C,f^*d}.
   \end{gather*}
 \end{enumerate}
\end{definition}

\begin{definition}
An \defemph{$\Id$-type structure} on a CwA $\C$ consists of:
\begin{enumerate}
  \item for each $\Gamma \in \C$ and $A \in \Ty(\Gamma)$, an element $\Id_A \in \Ty (\Gamma . A . p_A^*A)$;
  
  \item for each such $\Gamma$ and $A$, a map $\refl_A \colon \Gamma . A \to  \Gamma . A . p_A^*A . \Id_A$, over $(\id_A, \id_A) \colon \allowbreak \Gamma . A  \to \Gamma . A . p_A^*A$;

  \item for each $\Gamma$ and $A$ as above,
  \arxivonly{and each} extension $\Delta \in \Ty^*(\Gamma. A . p_A^*A . \Id_A)$,
  type $C \in \Ty(\Gamma. A . p_A^*A . \Id_A. \Delta)$,
  and map $d \colon \Gamma . A . \refl_A^*\Delta \to \Gamma . A . p_A^*A . \Id_A . \Delta .  C$
    over $p_C  d = \refl_A . \Delta$, 
  a section $\synJ_{\Delta, C,d}$ of $C$, 
  such that $\synJ_{\Delta, C,d} (\refl_A . \Delta) = d$;

  \item all stable under reindexing: for $f \colon \Gamma' \to \Gamma$ and appropriate arguments as above,
   \begin{gather*}
     f^*(\Id_A) = \Id_{f^*A}, \qquad
     f^*\refl_A = \refl_{f^*A}, \qquad 
     f^*\synJ_{\Delta, C,d} = \synJ_{f^*\Delta, f^*C,f^*d}.
   \end{gather*}
 \end{enumerate}
\end{definition}

\begin{definition}
A \defemph{unit-type structure} on a CwA $\C$ consists of:
\begin{enumerate}
  \item for each $\Gamma \in \C$, a type $1_\Gamma \in \Ty(\Gamma)$;
  \item for each $\Gamma \in \C$, a section $\star_\Gamma$ of $1_\Gamma$;
  \item for each $\Gamma \in \C$, 
  extension $\Delta \in \Ty^*(\Gamma . 1_\Gamma)$, 
  type $C \in \Ty(\Gamma . 1_\Gamma . \Delta)$,
  and \arxivonly{each} map $d \colon \Gamma . \star_\Gamma^* \Delta \to \Gamma . 1_\Gamma . \Delta . C$
    over $\star_\Gamma . \Delta$,
  a section $\rec_{\Delta, C, d}$ of $C$ over $ \Gamma . 1_\Gamma . \Delta $, 
    such that $\rec_{\Delta, C, d} (\star_\Gamma . \Delta) = d$;
  \item all stable under reindexing: for $f \colon \Gamma' \to \Gamma$ and appropriate arguments as above,
   \begin{gather*}
     f^*(1_\Gamma) = 1_{\Gamma'}, \qquad
     f^*\star_\Gamma = \star_{\Gamma'}, \qquad
     f^*\rec_{\Delta, C, d} = \rec_{f^*\Delta, f^*C,f^*d}.
   \end{gather*} 
 \end{enumerate}
\end{definition}

\begin{remark}
  The context extension argument $\Delta$ in the “eliminator” operations above is usually omitted.
  In the presence of $\Pi$-type structure, it is redundant, since these general eliminators can be constructed from the special case where $\Delta$ is empty.
  When considering the logical constructors individually, though, the more general form seems desirable, as noted in \jpaacite{gambino-garner} for the case of identity types.

  For disambiguation, these could be called \defemph{Frobenius} $\Id$-type structure, and so on, since as noted in \jpaacite{garner-berg:types-are-weak}, the argument $\Delta$ corresponds to a categorical Frobenius condition.

  Except for this difference, the present definitions are unchanged from \jpaacite[App.~A,~B]{kapulkin-lumsdaine:simplicial-model}, which in turn are direct algebraic translations of the original rules of \cite{martin-lof:bibliopolis}.
\end{remark}

\subsection{Elimination structures}

It is often profitable to encapsulate type-theoretic induction principles in the notion of an \emph{elimination structure}, defined in \cite{lumsdaine:thesis}.
We work for this subsection in a fixed ambient CwA $\C$.

\begin{definition}[{Cf.~\cite[Def.~1.2.8]{lumsdaine:thesis}}]
  A \defemph{pre-elimination structure} $e$ on a map $j \colon \Delta' \to \Delta$ is an operation providing, for each type $C \in \Ty(\Delta)$ and map $d \colon \Delta' \to \Delta.C$ over $\Delta$, a section $e_{C,d}$ of $C$ such that $e_{C,d}j = d$.
  \[ \begin{tikzcd}[column sep=small]
       \Delta' \ar[rr,"d"] \ar[dr,"j"'] & & \Delta.C \ar[dl,fib] \\
      & \Delta \ar[ur,dashed,bend left=15,shift left=0.5ex,"e_{C,d}"]
  \end{tikzcd} \]

  A \defemph{Frobenius pre-elimination structure} on $j \colon \Delta' \to \Delta$ is a family of pre-elimination structures $e_{\EE}$ on $j.\EE \colon \Delta'.j^*\EE \to \Delta.\EE$, for each context extension $\EE$ of $\Delta$.
  (We impose no compatibility condition between these elimination structures.)

  Given context extensions $\AA$, $\BB$ of a common base $\Gamma$, a \defemph{stable Frobenius elimination structure} on a map $j \colon \Gamma.\AA \to \Gamma.\BB$ over $\Gamma$ is:
  \begin{itemize}
  \item a family of Frobenius pre-elimination structures $e^f$ on $f^* j \colon \Gamma'.f^*\AA \to \Gamma'.f^*\BB$, for each $f \colon \Gamma' \to \Gamma$,
  \item commuting with reindexing, in that given any $g \colon \Gamma'' \to \Gamma'$, $f \colon \Gamma' \to \Gamma$, $C \in \Ty(\Gamma'.f^*\BB)$, and suitable inputs $\EE$, $C$, $d$ for $e^f$, we have $g^*e^f_{\EE,C,d} = e^{fg}_{g^*\EE,g^*C,g^*d}$.
  \end{itemize}
  By an \defemph{elimination structure}, we mean in the present paper a \emph{stable Frobenius elimination structure}.
\end{definition}

\begin{proposition}
  Suppose $j \colon \Gamma.\AA \to \Gamma.\BB$ is equipped with an elimination structure $e$ over $\Gamma$.
  Then for any $f \colon \Gamma' \to \Gamma$, the pullback $f^*j \colon \Gamma'.f^*\AA \to \Gamma'.f^*\BB$ carries an elimination structure $f^*e$; and this operation is functorial in $f$.
  Similarly, any context extension $j.\EE$ of $j$ carries an elimination structure $e.\EE$; and this commutes with the preceding operation, i.e.\ $f^*(e.\EE) = f^*e.f^*\EE$. \qed
\end{proposition}

\begin{definition}
  Suppose $\C$ is equipped with (Frobenius) $\Id$-type structure.
  Then for any $A \in \Ty(\Gamma)$, the map $\refl_A \colon \Gamma.A \to \Gamma.A.p^*A.\Id_A$ carries an elimination structure, which we call the \emph{canonical} such structure, induced by the elimination operation and computation axiom of the $\Id$-type structure.
  Moreover, this is stable under reindexing in $\Gamma$, and functorial in $\C$.

  Similarly, if $\C$ is equipped with (Frobenius) $\Sigma$-type structure, then for all suitable $\Gamma$, $A$, $B$, the pairing map $\pair_{A, B} \colon \Gamma.A.B \to \Gamma. \Sigma_A B$ carries a \emph{canonical} elimination structure, stably in $\Gamma$ and functorially in $\C$.
\end{definition}

For inductive types/families with multiple constructors, e.g., binary sums, one could generalise the definitions above to elimination structures not only on single maps but on families of maps; but we do not treat such type-formers in the present work.

Elimination structures are a type-theoretic analogue of left lifting properties, and are frequently used as such:
\begin{proposition}
  Let $e \colon \Gamma.\AA \to \Gamma.\BB$ be equipped with an elimination structure over $\Gamma$, and suppose $k \colon \Gamma.\BB \to \Delta$ and $h \colon \Gamma.\AA \to \Delta.C$ form a commutative square from $e$ to $\pi_C$.
  Then there is an induced diagonal filler $j$, making both triangles commute:
  \[ \begin{tikzcd}
       \Gamma.\AA \ar[d,"e"'] \ar[r,"h"] & \Delta.C \ar[d,fib] \\
      \Gamma.\BB \ar[ur,dashed,"j"] \ar[r,"k"] & \Delta
  \end{tikzcd} \]
  Moreover, this construction commutes with pullback in the base.
\end{proposition}

\begin{proof}
Apply the elimination structure to the induced map $\Gamma.\AA \to \Gamma.\BB.k^*C$.
\end{proof}

Under that analogy, the following propositions are familiar from the theory of weak factorisation systems, and their proofs adapt directly.
\begin{proposition}
  Every isomorphism $\Gamma.\AA \to \Gamma.\BB$ over $\Gamma$ canonically carries an elimination structure over $\Gamma$.
  Given elimination structures on composable maps $f$, $g$ over a base $\Gamma$, the composite $gf$ carries an induced elimination structure.
  Moreover, these constructions are stable in $\Gamma$, and functorial in $\C$.
  \qed
\end{proposition}

\begin{proposition}
  Let $f \colon \Gamma.\AA \to \Gamma.\BB$ be a map over $\Gamma$.
  Then an elimination structure on $f$ in $\C$ induces an elimination structure on $f$ in $(\C,\Ty^*)$;
  concretely, one can eliminate into arbitrary context extensions, not just single types.
  Moreover, this construction is stable in $\Gamma$, and functorial in $\C$.
  \qed
\end{proposition}

\begin{lemma} \label{lem:descent-from-elim-structure}
  Suppose $e \colon \Gamma.\AA \to \Gamma.\BB$ carries an elimination structure over $\Gamma$.

  \begin{enumerate}
  \item There is an induced retraction $r$ for $e$.
  \item For any $C \in \Ty(\Gamma.\AA)$, there is a type $\descend{e}{C} \in \Ty(\Gamma.\BB)$ such that $e^*(\descend{e}{C}) = C$;
    we call $\descend{e}{C}$ the \defemph{descent} of $C$ along $e$. 
  \[\begin{tikzcd}
    \Gamma.\AA.C \ar[d, fib] \ar[r,dashed] \arrow[drpb] & \Gamma.\BB.\descend{e}{C} \ar[d, fib, dashed] \\
    \Gamma.\AA \ar[r, "e"] & \Gamma.\BB
  \end{tikzcd}\]  
  \end{enumerate}
  Moreover, both these constructions are stable in the base $\Gamma$, and functorial in $\C$.
\end{lemma}

\begin{proof}
  The elimination structure, applied to $(e,\id) \colon \Gamma.\AA \to \Gamma.\BB.\pi_{\BB} \AA$, gives a section of $\pi_\BB^*\AA$, whose composite with $\pi_\BB.\AA$ yields the desired retraction $r$.
  \[ \begin{tikzcd}[column sep=small]
       \Gamma.\AA \ar[rr,"{(e,\id_\AA)}"] \ar[dr,"e"'] & & \Gamma.\BB.\pi_\BB^*\AA \ar[dl,fib] \ar[dr,"\pi_\BB.\AA"] \\
      & \Gamma.\BB \ar[ur,dashed,bend left=15,shift left=0.5ex] \ar[rr,dashed,"r"] & & \Gamma.\AA
  \end{tikzcd} \]
  Now, given $C \in \Ty(\Gamma.\AA)$, take $\descend{e}{\CC}$ to be $r^*\CC$.
\end{proof}

\subsection{Logical structure on context extensions}

Many of the logical structures on types extend, by iteration, to analogous constructions on context extensions in a CwA.

\begin{construction}[{\jpaacite[Prop.~3.3.1]{garner:2-d-models}}] \label{def:identity-context}
  If a CwA $\C$ carries identity type structure, then so does $(\C,\Ty^*)$, functorially in $\C$.
  
  Concretely, any context extension $\AA \in \Ty^*(\Gamma)$ has an \defemph{identity context} $\Id_\AA \in \Ty^*(\Gamma.\AA.\AA)$, along with a reflexivity map $\refl_\AA \colon \Gamma.\AA \to \Gamma.\AA.\AA.\Id_\AA$ over $(\id,\id) \colon \Gamma.\AA \to \Gamma.\AA.\AA$, equipped with an elimination structure over $\Gamma$.
 
  $\Id_\AA$ is given by induction on the length of $\AA$.
  When $\AA$ is empty, so is $\Id_\AA$, and $\refl_\AA$ is the identity map.
  When $\AA = (\AA_{<n},A_n)$, then assuming that we have constructed $\Id_{\AA_{<n}}$, $\refl_{\AA_{<n}}$, and its elimination structure, we take $\Id_\AA$ to be $(\Id_{\AA_{<n}}, (\Id_\AA)_n )$, where the type $(\Id_\AA)_n$ is given as a descent of $\Id_{A_n}$ as in the following diagram:
  \[ \begin{maxwidthtikzcd}[row sep=scriptsize]
       \Gamma.\AA_{<n}.A_n \ar[r,"\refl_{A_n}"] \ar[dr] & \Gamma.\AA_{<n}.A_n.A_n.\Id_{A_n} \ar[d,fib] \ar[rr,dashed] \ar[drpb] & &[-1em] \Gamma.\AA.\AA.\Id_{\AA_{<n}}. (\Id_\AA)_n \ar[d,dashed,fib] \\
       & \Gamma.\AA_{<n}.A_n.A_n \ar[r,"\refl_{\AA_{<n}}.A_n.A_n"] \ar[dr] & \Gamma.\AA_{<n}. \AA_{<n}. \Id_{\AA_{<n}}.A_n.A_n \ar[r,"\iso"] \ar[d] & \Gamma.\AA.\AA.\Id_{\AA_{<n}} \ar[d,fibb] \\
       & & \Gamma.\AA_{<n}. \AA_{<n}.A_n.A_n \ar[r,"\iso"] & \Gamma.\AA.\AA
  \end{maxwidthtikzcd} \]
  Then $\refl_\AA$ is the composite the upper edge of the diagram; that is, of $r_{A_n}$, a context extension of $\refl_{\AA_{<n}}$, and an isomorphism, and so as such, carries an elimination structure.

  We note for future use that $\refl_\AA$ is always a composite of context extensions of the individual reflexivity maps $\refl_{A_i}$ and isomorphisms. \qed
\end{construction}

Similarly, the $\Pi$-type structure can be lifted to $(\C, \Ty^*)$ via the following construction.
(Although this construction is well-known in practice, we are unaware of it being documented in the literature.)

\begin{construction} \label{prop:iterated-pi}
  If a CwA $\C$ carries $\Pi$-type structure, then so does $(\C,\Ty^*)$, functorially in $\C$.

  Let $\AA = A_1.\cdots.A_n$ be an extension of $\Gamma \in \C$, and $\BB = B_1.\cdots.B_m$ a further extension of $\Gamma.\AA$.
  We define the context extension $\Pi[\AA,\BB]$ and map $\ev_{\AA,\BB} \colon \Gamma.\Pi[\AA,\BB].\AA \to \Gamma.\AA.\BB$ over $\Gamma.\AA$, as follows.
  
  For a single type $B$, $\Pi[\AA,B]$ is the type $\Pi[A_1,\Pi[A_2, \ldots \Pi[A_n,B] \ldots ]]$, and $\ev_{\AA,B}$ is the composite
   \[ \begin{tikzcd}[row sep=scriptsize]
       \Gamma . \Pi[A_1,\Pi[A_2, \ldots \Pi[A_n,B] \ldots ]] . A_1 . \cdots .A_n
       \ar[d, "{\ev_{A_1, \Pi[A_2, \ldots \Pi[A_n,B] \ldots ]}.A_2. \cdots .A_n}"]
       \\ \Gamma . A_1 . \Pi[A_2, \ldots \Pi[A_n,B] \ldots ] . A_2 . \cdots .A_n
       \ar[d, "{\ev_{A_2, \Pi[A_3, \ldots \Pi[A_n,B] \ldots ]}.A_3. \cdots .A_n}"]
       \\ \vdots
       \ar[d, "{\ev_{A_n,B}}"]
       \\ \Gamma. \AA . B.
     \end{tikzcd} \]
   
   For general $\BB$, work by induction on its length $m$.
   When $m=0$, $\Pi[\AA,\BB]$ is the trivial 0-step context extension (as must be $\BB$), and $\ev_{\AA,\BB}$ is $\id_{\Gamma.\AA}$.
   When $m > 0$, write $\BB = \BB_{<m}.B_m$.
   By induction, we have $\Pi[\AA,\BB_{<m}]$ and
   \[ \ev_{\AA,\BB_{<m}} \colon \Gamma.\Pi[\AA,\BB_{<m}].\AA \to \Gamma.\AA.\BB_{<m}. \]
   Now we take $\Pi[\AA,\BB]$ to be $\Pi[\AA,\BB_{<m}] . \Pi[\AA, \ev_{\AA,\BB_{<m}}^*B_m]$ (where the new last component is given by the single-step case above), and $\ev_{\AA,\BB}$ to be the composite
   \[ \begin{tikzcd}[row sep=scriptsize]
       \Gamma . \Pi[\AA,\BB_{<m}] . \Pi[\AA, \ev_{\AA,\BB_{<m}}^*B_m] . \AA
       \ar[d, "{\ev_{\AA,\ev_{\AA,\BB_{<m}}^*B_m}}"]
       \\ \Gamma . \Pi[\AA,\BB_{<m}] . \AA . \ev_{\AA,\BB_{<m}}^*B_m
       \ar[d, "{\ev_{\AA,\BB_{<m}}.B_m}"]
       \\ \Gamma . \AA . \BB.
     \end{tikzcd} \]

  The construction of the $\lambda$ operation is straightforward along similar lines, again by induction on the length of $\BB$. \qed
\end{construction}

We will generally consider both identity contexts and iterated $\Pi$-types not from the point of view of $(\C,\Ty^*)$ as a CwA in its own right, but as auxiliary operations for working in $\C$.
A similar construction is possible for $\Sigma$-types, but we do not recall it here as we will not require it.

\subsection{Equivalences}

We recall here the well-known definition of \emph{equivalences} of types or contexts, along with a generalisation from \jpaacite{kapulkin-lumsdaine:homotopy-theory-of-type-theories} to arbitrary maps in CwA’s, and some key facts about these notions.

Work for the remainder of this subsection in a fixed ambient CwA $\C$ equipped with $\Id$-types.

\begin{definition}
  Let $f, g \colon \Gamma.\AA \to \Gamma.\BB$ be a map between context extensions over some base $\Gamma \in \C$.
  A \defemph{homotopy} $h \colon f \homot h$ over $\Gamma$ is a map $h \colon \Gamma.\AA \to \Gamma.\BB.\BB.\Id_{\BB}$ (where $\Id_{\BB}$ is the iterated identity type as constructed above), such that $p_{\Id_{\BB}}h = (f,g)$.
  We write $f \homot g$ to mean that some such $h$ exists.
\end{definition}

\begin{definition} \label{def:equivalence-in-fib-slice}
  Let $\AA, \BB$ be context extensions of some base $\Gamma \in \C$.
  A map $f \colon \Gamma.\AA \to \Gamma.\BB$ over $\Gamma$ is an \defemph{equivalence} over $\Gamma$ if there exist maps $g, g' \colon \Gamma.\BB \to \Gamma.\AA$ and homotopies $\eta \colon fg \homot \id_{\Gamma.\BB}$ and $\varepsilon \colon g'f \homot \id_{\Gamma.\AA}$ over $\Gamma$.
  \defemph{Equivalence data} for $f$ consists of some choice of such $g$, $\eta$, $g'$, $\varepsilon$.
\end{definition}

This approach works only within “fibrant slices” of a CwA, since iterated identity types are available just for context extensions, not for arbitrary objects.
In Section~\ref{sec:homotopical-diagrams}, we will need a more general notion:

\begin{definition}[{\jpaacite[Def.~4.5]{kapulkin-lumsdaine:homotopy-theory-of-type-theories}}] \label{def:equivalence-in-cwa}
  A map $f \colon \Gamma' \to \Gamma$ in $\C$ is an \defemph{equivalence} if it satisfies the following properties:
  \begin{enumerate}
  \item \defemph{weak type lifting}: given any context extension $\AA$ of $\Gamma$ and type $B \in \Ty(\Gamma'. f^*\AA)$, there exists some type $B' \in \Ty(\Gamma.\AA)$ and equivalence $w \colon \allowbreak \Gamma'. f^*\AA.B \weqto \Gamma'.f^*\AA.f^*B'$ over $\Gamma' .f^*\AA$ (in the sense of Definition~\ref{def:equivalence-in-fib-slice});
  \item \defemph{weak term lifting}: given any context extension $\AA$ of $\Gamma$, type $B \in \Ty(\Gamma'.\AA)$, and section $b$ of $f^*B$, there exists some section $b'$ of $B$, such that $f^*b'$ is propositionally equal to $b$.
  \end{enumerate}
\end{definition}
From the point of view of $f$, these might be more naturally called \emph{descent} properties; they are \emph{lifting} when viewed as properties of the induced functor $f^* \colon \C(\Gamma) \to \C(\Gamma')$ (cf.~Definition~\ref{def:fibrations-and-equivs-of-cwas} below).

We recall without proof some key facts about these notions:

\begin{lemma}[{\jpaacite[Prop.~4.6]{kapulkin-lumsdaine:homotopy-theory-of-type-theories}}] \label{lem:equivalences-equivalent}
  A map $f \colon \Gamma.\AA \to \Gamma.\BB$ over $\Gamma$ is an equivalence in the sense of Definition~\ref{def:equivalence-in-cwa} if and only if it is one in the sense of Definition~\ref{def:equivalence-in-fib-slice}. \qed
\end{lemma}

This justifies making no terminological distinction between the general and the more restricted notion.

\begin{lemma}
  Equivalences are stable under reindexing between fibrant slices:
  if $w \colon \Gamma.\AA \to \Gamma.\BB$ is an equivalence over $\Gamma$, then for any $f \colon \Gamma' \to \Gamma$, the map $f^*w \colon \allowbreak \Gamma'.f^*\AA \to \Gamma'.f^*\BB$ is again an equivalence.
  
  Moreover, this is canonically witnessed by operations on equivalence data, functorially in $\Gamma$ and $\C$.
  That is, if $w \colon \Gamma.\AA \to \Gamma.\BB$ carries equivalence data $E$, then $\id_\Gamma^* E = E$, and for any $\Gamma'' \to[g] \Gamma' \to[f] \Gamma$, $g^*f^*E = (fg)^*E$;
  and if $F \colon \C \to \C'$ is a map of CwA’s with identity types, $w \colon \Gamma.\AA \to \Gamma.\BB$ a structured equivalence over $\Gamma$ with equivalence data $E$, and $f \colon \Gamma' \to \Gamma$ a map, then $f^*(FE)$ and $(Ff)^*(FE)$ are equal as equivalence data on $F(f^*w)$ over $F \Gamma'$.
\end{lemma}

\begin{proof}
The ``data'' version is straightforward; the ``property'' version follows directly.
\end{proof}

\begin{lemma}[{\jpaacite[Prop.~4.7]{kapulkin-lumsdaine:homotopy-theory-of-type-theories}, \jpaacite[Lemma~3.2.6]{avigad-kapulkin-lumsdaine}}]
  Equivalences in $\C$ satisfy the 2 out of 3 property.
  Moreover, within fibrant slices $\C(\Gamma)$, the 2 out of 3 property is witnessed by operations on equivalence data, stably in $\Gamma$ and functorially in $\C$. \qed
\end{lemma}

\begin{lemma}[{\jpaacite[Prop.~4.9]{kapulkin-lumsdaine:homotopy-theory-of-type-theories}}]
  Equivalences satisfy “right properness”: if $w \colon \Gamma' \to \Gamma$ is an equivalence, and $\AA$ is any context extension of $\Gamma$, then $w.\AA \colon \Gamma'.w^*\AA \to \Gamma.\AA$ is again an equivalence.
  
  Moreover, within a fibrant slice $\C(\Gamma)$, this is witnessed by operations on equivalence data, stably in $\Gamma$ and functorially in $\C$. \qed
\end{lemma}

\subsection{Function extensionality}

\begin{definition}[{\jpaacite[Def.~B.3.1]{kapulkin-lumsdaine:simplicial-model}, \jpaacite[\textsection 5.1]{garner:on-the-strength}}] \label{def:funext-structure}
  Suppose $\C$ is equipped with $\Pi_\eta$- and $\Id$-type structure.
  Then \defemph{function extensionality structure} on $\C$ consists of operations giving:
  \begin{enumerate}
  \item for each $\Gamma$, $A$, $B$ as in the $\Pi$ operation, a map
    \[ \funext_{A,B} \colon \Gamma. \Pi[A,B]. \Pi[A,B] . \Pi[A,e_{A,B}^*\Id_B] \to \Gamma. \Pi[A,B]. \Pi[A,B] . \Id_{\Pi[A,B]} \]
    over $\Gamma. \Pi[A,B]. \Pi[A,B] $, where $\Id_B$ in the domain is pulled back along the map
    $e_{A,B} \colon \Gamma.\Pi[A,B]. \Pi[A,B]. A \to \Gamma.A.B.B$
    that evaluates both its function arguments;
  \item and for each such $\Gamma$, $A$, $B$, a homotopy
   \begin{multline*}
     \funextcomp_{A,B} \colon \funext_{A,B} \lambda(\refl_B \ev_{A,B}) \homot \refl_{\Pi[A,B]} \\
     \colon \Gamma.\Pi[A,B] \to \Gamma. \Pi[A,B]. \Pi[A,B] . \Id_{\Pi[A,B]},
   \end{multline*}
  \item all stable under pullback in the base.
  \end{enumerate}

  We say a $\Pi_\eta$-type structure is \defemph{extensional} if it is additionally equipped with function extensionality structure,
  and write \defemph{$\Pi_\ext$-structure} to refer to the combined structure.
\end{definition}

For the remainder of this subsection, fix a CwA $\C$ equipped with $\Id$- and $\Pi$-type structure.

The following alternative form of function extensionality is less type-theo\-retically economical, but more categorically succinct:

\begin{lemma}[due to Voevodsky \cite{UniMath}; see \cite{lumsdaine:funext-blog}] \label{lem:funext-iff-r-equiv}
  A given triple $\Gamma$, $A$, $B$ admits maps $\funext_{A,B}$, $\funextcomp_{A,B}$ as in the arguments of function extensionality structure if and only if the map
  \[  \lambda_{A,B.B.\Id_B}(\refl_B \ev_{A,B}) \colon \Gamma . \Pi[A,B] \to \Gamma . \Pi[A,B.B.\Id_B] \] 
  admits equivalence data.
  Moreover, the operations converting between these two types of data on $\Gamma$, $A$, $B$ are stable under reindexing in $\Gamma$, and functorial in $\C$ under maps preserving $\Id$- and $\Pi_\eta$-structure.

  Overall, $\C$ therefore admits function extensionality structure if and only if it admits a choice of equivalence data on each map $\lambda_{A,B.B.\Id_B}(\refl_B \ev_{A,B})$, stably in $\Gamma$;
  moreover, a map of CwA’s preserving the extensionality structure preserves the resulting choice of equivalence data, and a map preserving such chosen equivalence data preserves the resulting derived extensionality structure.  \qed
\end{lemma}

It is straightforward to check that just like $\Id$- and $\Pi$-types, extensionality also lifts to context extensions, and hence to general homotopies:
\begin{lemma} \label{lem:iterated-pi-extensionality}
  Suppose $\C$ is equipped with $\Pi$-types, $\Id$-types, and function extensionality structure.
  Then $(\C,\Ty^*)$ carries functional extensionality structure, for the $\Pi$-type structure and $\Id$-type structure given above.
  
  In particular, given two sections $b, b'$ of a context extension $\BB \in \Ty^*(\Gamma.\AA)$, and a homotopy $h \colon b \homot b'$ between them over $\AA$, there is an induced homotopy between their abstractions $\lambda(b), \lambda(b') \colon \Pi[\AA,\BB]$, stably in $\Gamma$.

  Moreover, these constructions are functorial in $\C$. \qed
\end{lemma}

As a corollary, we derive the covariant action of iterated $\Pi$-types on equivalences between their codomains (though without making its functoriality precise, as we do not need it):
\begin{lemma} \label{lem:iterated-pi-covariant}
  Suppose $w \colon \Gamma.\AA.\BB \to \Gamma.\AA.\BB'$ is an equivalence over $\Gamma.\AA$.
  Then the induced map $\lambda(w) \colon \Gamma.\Pi[\AA,\BB] \to \Gamma.\Pi[\AA,\BB']$ is an equivalence over $\Gamma$. \qed
\end{lemma}


\section{Inverse diagrams and Reedy types} \label{sec:inverse-diagrams}

In this section, we recall the definition of inverse categories, and define \emph{Reedy types} over diagrams on inverse categories valued in a CwA.
These are analogous to \emph{Reedy fibrations}, a well developed tool from the homotopy theory of (co-)fibration categories \arxivorsubmission{\cite{radulescu-banu,szumilo:two-models}}{\cite{radulescu-banu-for-jpaa,szumilo:two-models-for-jpaa}} and similar settings.

Compared to such settings, Reedy types in CwA’s incur several extra complications.
Due to these, we split the definition into two stages.
We first define \emph{weak Reedy types}, which are unsatisfactorily flabby, but suffice for setting up the machinery of \emph{Reedy limits}.
Armed with these we can then define \emph{(strict) Reedy types}, which are what we really want.

\subsection{Inverse categories and weak Reedy types}

\begin{definition}
  Let $\I$ be a category.
  The \defemph{precedence} ordering $<$ on $\ob \I$ is defined by taking $i < j$ just if there exists some non-identity arrow $\alpha \colon j \to i$.
 
  An \defemph{inverse category} $\I$ is a category in which the precedence ordering is well-founded, and each object has finitely many predecessors.%
  \footnote{The finiteness condition is sometimes omitted in this definition.
  In that case, it should instead be added in Definition~\ref{def:ordering} (\emph{ordered} inverse categories) and assumed in Propositions~\ref{prop:orderable} and~\ref{prop:weak-reedy-comp-cat}.}

  For $i \in \I$, we write $\yon{i} \in \copsh{\I}$ for the Yoneda co-presheaf $\I(i,-)$, and $\bdry{i} \subseteq \yon{i}$ for the sub-co-presheaf of non-identity maps out of $i$.
\end{definition}

\begin{example}
  A useful running example throughout this paper is the “walking span”: the posetal inverse category $(0 \from 01 \to 1)$, which we denote $\SpanCat$.
\end{example}

\begin{definition}
  Given categories $\C$, $\I$, diagrams $F \in \copsh{\I}$, $D \in \C^\I$, and an object $C \in \C$, an \defemph{$F$-cylinder} $\lambda$ from $C$ to $D$ consists of maps $\lambda_x \colon C \to D(i)$ for each $i \in \I$ and $x \in F(i)$, such that $D(\alpha) \lambda_x = \lambda_{\alpha x}$ for each $\alpha \colon i \to j$, $x \in F(i)$.%
\footnote{These are a special case of the cylinders of enriched category theory \cite[\textsection 3.1]{kelly:basic-concepts}.  As we require only a few facts about them, we give a self-contained treatment rather than assuming familiarity with that theory.}
  
  We denote such cylinders by $\lambda\colon C \to_F D$.
  They are easily seen to be functorial in all parameters: contravariantly in $C$, covariantly in $D$, contravariantly in $F$, covariantly 2-functorial in $\C$, and contravariantly 2-functorial in $\I$.
\end{definition}

\begin{definition}
  Let $\C$, $\I$ be categories, and $p \colon Y \to X$ a map in $\C^\I$.
  Then for $i \in \I$, a \defemph{(relative) matching object for $p$ at $i$} is
  \begin{itemize}
  \item an object $M_i p \in \C$ equipped with a $\yon{i}$-cylinder $\lambda \colon M_i p \to_{\yon{i}} X$ and a $\bdry{i}$-cylinder $\mu \colon M_i p \to_{\bdry{i}} Y$, such that $\lambda \restr{\bdry{i}} = p \mu$,
  \item that is moveover terminal among objects equipped with such cylinders.
  \end{itemize}

  A functor $F \colon \C \to \D$ \defemph{preserves} some relative matching object $(M_i p, \lambda, \mu)$ in $\C$ if its image $(F M_i p, F \lambda, F \mu)$ is a matching object for $F p$ at $i$ in $\D$.
\end{definition}

In classical presentations of Reedy fibrations, relative matching objects are not often explicitly defined, appearing instead as pullbacks of absolute matching objects: $M_i p \coloneqq M_i Y \times_{M_i X} X_i$.
In our setting of CwA’s, however, these absolute matching objects may often fail to exist; so we define the relative ones directly.

Relative matching objects may also be seen as limits indexed over the oplax pushout category $\displaystyle \strictslice{\I}{i} \underset{\strictslice{\I}{i}}{+^{\rightarrow}} \slice{\I}{i}$.

\begin{proposition} \label{prop:discrete-opfib-matching-object}
  Let $u \colon \J \to \I$ be a discrete opfibration of inverse categories.
  Then for any category $\C$, map $p$ in $\C^\I$, and $j \in \J$, if $(M_{uj} p, \mu, \lambda)$ is some matching object for $p$ at $uj$, then $(M_i p, \strictslice{u}{j}^* \mu, \slice{u}{j}^* \lambda)$ is a matching object for $u^*p$ at $j$.
\end{proposition}

\begin{proof}
  Since $u$ is a discrete opfibration, its induced maps $\slice{u}{j}$ on co-slices and $\strictslice{u}{j}$ on strict co-slices 
  are isomorphisms; so the required limit property of matching objects is unchanged by $u^*$. 
\end{proof}

Before defining the “Reedy types” that we actually want, we need to start with a more general “weak” notion, from which we will afterwards carve out the “strict” ones.
\optionaltodo{This is repeating the intro a bit…}

Fix, for the next few definitions, a CwA $\C$ and an inverse category $\I$.
We will define the fibration of \defemph{weak Reedy $\I$-types in $\C$} over the diagram category $\C^\I$ (or just \defemph{weak Reedy types}, when $\I$ and $\C$ are clear).

\begin{definition} \label{def:weak-reedy-type}
  A \defemph{weak Reedy $\I$-type} over a base $X \in \C^\I$ consists of:
  \begin{enumerate}
    \item an object $Y$ and map $p \colon Y \to X$ in $\C^\I$, together with
    \item chosen matching objects $M_ip$ for $p$, for each $i \in \I$, and
    \item for each $i \in \I$, a type $A_i \in \Ty(M_ip)$, and an isomorphism $\varphi_i \colon Y_i \to M_i p.A_i$ over $M_i p$.
  \end{enumerate}

  A map of weak Reedy types $(Y',p',M',A',\varphi') \to (Y,p,M,A,\varphi)$, over a map $f \colon X' \to X$ in $\C^\I$, consists of:
  \begin{enumerate}
    \item a map $\bar{f} \colon Y' \to Y$ over $f$;
    \item maps $\barbar{f}_i \colon A'_i \to A_i$ in $\elem{\Ty}$ over the induced maps $M_i \bar{f}_i \colon M'_ip' \to M_ip$,
    \item commuting with the isomorphisms $\varphi$, $\varphi'$:
      \[ \begin{tikzcd}
        M'_ip'.A'_i \ar[r,equals] & M'_ip'.(M_i \bar{f})^*A_i \ar[r,"\barbar{f}_i"] & M_i p . A_i \\
        Y'_i \ar[u,"\varphi'_i","\iso"'] \ar[rr,"\bar{f}_i"] & & Y_i \ar[u,"\varphi_i","\iso"']
      \end{tikzcd} \]
  \end{enumerate}

  Write $\T_{\wkreedy(\C,\I)}$ for the total category of weak Reedy $\I$-types in $\C$; this has an evident forgetful functor to $\C^\I$.

  We will often write $A$ for a weak Reedy $\I$-type $(Y,p,M,A,\varphi)$ over $\Gamma$, and in this case write $\Gamma.A$ for its total diagram $Y$, $\pi_A \colon \Gamma.A \to \Gamma$ for the projection $p$, and $M_i A$ for the given matching objects.
  This notation is motivated by the special case of strict Reedy types, Definition~\ref{def:strict-reedy-types}, in which the types $A$ fully determine the rest of the data, and which form the types of a CwA structure on $\C^\I$.
\end{definition}

Weak Reedy types are reasonably transparently an analogue of Reedy fibrations, as ordinarily defined in a fibration category or similar settings.
However, in our setting they are a little unsatisfactorily flabby, due to the essentially redundant data of chosen matching objects: for instance, in the case $\I = 1$, a weak Reedy type corresponds to “a map in $\C$ isomorphic to a dependent projection”.

%
For these reasons, we will rein in the redundancy by cutting down to \emph{strict} Reedy types, whose matching objects are constructed canonically rather than forming extra data.
These matching objects are also needed to show that weak Reedy types form a comprehension category; for now we note as much of that as is obvious.

\begin{proposition} \label{prop:weak-reedy-almost-comp-cat}
  There is a “comprehension” functor $\chi \colon \T_{\wkreedy(\C,\I)} \to (\C^\I)^{\rightarrow}$ over $\C$, sending $(Y,p,M_i,A,\varphi)$ to $(Y,p)$, forgetting the given matching objects and types.
  \[ \begin{tikzcd}[column sep=tiny]  \T_{\wkreedy(\C,\I)} \ar[rr,"\chi"] \ar[dr] & & (\C^\I)^{\rightarrow} \ar[dl,"\cod"] \\
      & \C &
    \end{tikzcd} \]
  We write $\C^\I_{\wkreedy}$ to denote this whole triangle of data, considered as a comprehension category minus the existence and preservation of cartesian lifts.

  Moreover, any discrete opfibration $u \colon \J \to \I$ induces a functor $u^* \colon \allowbreak \T_{\wkreedy(\C,\I)} \allowbreak \to \T_{\wkreedy(\C,\J)}$ (by Proposition~\ref{prop:discrete-opfib-matching-object}), and any CwA map $F \colon \C \to \D$ induces a ``partial reindexing functor'', i.e.\ a functor $F^\I \colon \T_{\wkreedy(\C,\I)}^F \to \T_{\wkreedy(\D,\I)}$, where $\T_{\wkreedy(\C,\I)}^F$ is the full subcategory of $\T_{\wkreedy(\C,\I)}$ on types whose matching objects are preserved by $F$. \qed
\end{proposition}

In Proposition~\ref{prop:weak-reedy-comp-cat} below we will fill in the gaps: we will show that $\T_{\wkreedy(\C,\I)}$ has and $\chi$ preserves cartesian lifts, and CwA maps preserve matching objects for Reedy types; so $\C^\I_{\wkreedy}$ forms a comprehension category, functorially in $\C$ and $\I$.

\subsection{Reedy limits}

It is familiar from homotopy theory that one does not really need general finite completeness to work with Reedy diagrams, since the limits used---matching objects of Reedy fibrant diagrams, and related constructions---can always be constructed from just pullbacks of fibrations.
Analogously, in our setting, matching objects and other limits we use can be constructed as context extensions, using just pullbacks of types.

Here we set up machinery for constructing and manipulating these limits rather precisely,
%
%
since we often care about \emph{strict} preservation of types/context extensions under reindexing and CwA maps.

Fix, for this subsection, an inverse category $\I$.

\begin{definition} \label{def:relative-pullback}
  Let $i \colon F \to G$ be a map in $\copsh{\I}$, $p \colon Y \to X$ a map in $\C^\I$ (for some category $\C$); suppose moreover we are given $C \in \C$, and a $G$-cylinder $\lambda \colon C \to_G X$ and $F$-cylinder $\mu \colon C \to_F Y$, such that $p \mu = \lambda i \colon C \to_F X$.
  Then a \defemph{$i$-pullback of $p$ along $(\lambda,\mu)$} is an object $A \in \C$ together with a map $q \colon A \to C$ and $G$-cylinder $\gamma \colon A \to_G Y$, such that :  
  \begin{enumerate}
  \item $p \gamma = \lambda q \colon A \to_G X$;
  \item $\gamma i = \mu q \colon A \to_F Y$;
  \item $(A,q,\gamma)$ is universal among objects over $\Delta$ equipped with such a cylinder.
  \end{enumerate}
  \[ \begin{tikzcd}[column sep=huge]
    A \ar[d,dashed,"q"'] \ar[r,dashed,shift left=0.3ex,"G"' very near end,"\gamma"] & Y \ar[d,"p"] \\
    C \ar[r,"\lambda", "G"' very near end] \ar[ur,end anchor={[yshift=-0.3ex]},"\mu","F"' very near end] & X
  \end{tikzcd}\]
  \optionaltodo{Can we think of a less bad terminology for these than the current “$i$-pullback”, “relative pullback”?}
\end{definition}

\begin{remark} \label{rem:matching-object-as-relative-pullback}
  When enough limits exist, this universal property says just that $(A,q,\gamma)$ is a pullback of $\cotensor{i}{p}$ along $(\mu,\lambda)$, where $\cotensor{i}{p}$ is the Leibniz cotensor of $i$ with $p$:
  \[ \begin{tikzcd}
    A \ar[d,dashed] \ar[r,dashed] \ar[drpb] & Y^G \ar[d,"{\cotensor{i}{p}}"] \\
    C \ar[r,"{(\mu,\lambda)}"] &  Y^F \times_{X^F} X^G 
  \end{tikzcd} \]
  %
  %
  Relative matching objects are a special case, with $F = 0$, $G = \delta{i}$, and $C = X_i$.
\end{remark}

To construct such relative pullbacks, we will assume additional structure on $i \colon F \to G$.

\begin{definition}
  A \defemph{finite extension} in $\copsh{\I}$ is a monomorphism $i \colon F \injto G$, together with a linear ordering on its total complement $\coprod_i G_i \setminus F_i$, such that:
  \begin{enumerate}
  \item the total complement is finite, and
  \item for any $\alpha \colon i \to j$ and $x \in G_i \setminus F_i$, if $\alpha x \in G_j \setminus F_j$ then $\alpha x \leq x$.
  \end{enumerate}
  
  A map of finite extensions is a pushout square between them that preserves the given orders on the total complements.
\end{definition}

Equivalently, a finite extension is a map exhibited as a finite cell complex of the boundary inclusions $\bdry{i} \injto \yon{i}$.

The orderings given in finite extensions are exactly the extra data needed to construct relative pullbacks of Reedy types.

\begin{lemma}[Master lemma for Reedy limits] \label{lem:master-lemma}
  Suppose we are given a finite extension $i \colon F \extto G$ in $\copsh{\I}$, a weak Reedy type $A$ over $\Gamma$ in $\C^\I$, an object $\Delta \in \C$, a $G$-cylinder $\lambda \colon \Delta \to_G \Gamma$, and an $F$-cylinder $\mu \colon \Delta \to_F \Gamma.A$ over $\lambda \restr{F}$. 

  Then we obtain a context extension $(\lambda,\mu)^*A \in \Ty^*{\Delta}$ (with length equal to the length of the extension $F \extto G$) and a $G$-cylinder $(\lambda,\mu).A \colon \Delta.(\lambda,\mu)^*A \to_G \Gamma.A$, such that $\Delta. (\lambda,\mu)^*A$ together with $\pi_{(\lambda,\mu)^*A}$ and $(\lambda,\mu).A$ forms an $i$-relative pullback of $\pi_A$ along $(\lambda,\mu)$:
  \[ \begin{tikzcd}[column sep=huge]
    \Delta.(\lambda,\mu)^*A \ar[d,fibb,dashed] \ar[r,dashed,shift left=0.3ex,"G"' very near end,"{(\lambda,\mu).A}"] & \Gamma.A \ar[d,fib] \\
    \Delta \ar[r,"\lambda", "G"' very near end] \ar[ur,end anchor={[yshift=-0.3ex]},"\mu","F"' very near end] & \Gamma
  \end{tikzcd} \]
\end{lemma}

\begin{proof}
  First consider the case where $F \extto G$ is a single-step extension $F \extto F +_b \yon{i}$, for some $b \colon \bdry{i} \to F$.
  In this case, $\lambda \restr{\yon{i}}$ and $\mu \restr{\bdry{i}}$ induce a map $m \colon \Delta \to M_iA$.
  We take $(\lambda,\mu)^*A$ to be $m^*A_i \in \Ty{\Delta} \subseteq \Ty^*{\Delta}$, and $(\lambda,\mu).A$ to be the $(F +_b \yon{i})$-cylinder consisting of $\mu$ together with $\varphi_i^{-1} (m.A_i) \colon \Delta. m^*A_i \to (\Gamma.A)_i$, where $\varphi_i \colon (\Gamma.A)_i \to M_iA . A_i$ is supplied by the weak Reedy data of $A$.

  For the general case, suppose $F \extto G$ is given as a composite of single-step extensions:
  \[ F = F_0 \extto F_1 \extto \cdots \extto F_n = G. \]
  Then we define by induction a sequence of objects $\Delta_k \in \C$, along with $G$-cylinders $\lambda_k \colon \Delta_k \allowbreak \to_G \Gamma$ and $F_k$-cylinders $\mu_k \colon \Delta_k \to \Gamma.A$ over $\lambda_k$, for $k = 0, \ldots, n$, as follows:
  \begin{enumerate}
  \item $\Delta_0 \coloneqq \Delta$; $\lambda_0 \coloneqq \lambda$; $\mu_0 \coloneqq \mu$;
  \item $\Delta_{k+1} \coloneqq \Delta_k . (\lambda_k \restr{F_{k+1}},\mu_k)^* A$;
    $\lambda_{k+1} \coloneqq \lambda \pi_{(\lambda_k \restr{F_{k+1}},\mu_k)^* A}$;
    and
    $\mu_{k+1} \coloneqq \allowbreak (\lambda_k \restr{F_{k+1}},\mu_k).A$.
  \end{enumerate}
  Now we take $(\lambda,\mu)^*A$ to be the resulting context extension $\Delta_n \fibbto \Delta$, and $(\lambda,\mu).A$ to be $\mu_n$.
  \optionaltodo{TODO: add diagrams?}
\end{proof}

By \defemph{Reedy limits}, we will mean the objects and cylinders constructed according to the preceding lemma.

\begin{lemma} \label{lem:first-properties-of-reedy-limits}
  Reedy limits enjoy the following properties:

  \begin{enumerate}
  \item Functoriality in $\Delta$: Given $F \extto G$, $\Gamma$, $A$, $\Delta$, $\lambda$, $\mu$ as above, and a map $f \colon \Delta' \to \Delta$, we have $f^*(\lambda,\mu)^*A = (\lambda f,\mu f)^*A$, and $((\lambda,\mu).A) (f.((\lambda,\mu)^*A)) \allowbreak = (\lambda f,\mu f).A$.
  \item Functoriality in $\Gamma$, $A$: Given data as above together with $\Gamma'$, $A'$, $f \colon \Gamma \to \Gamma'$, and $\bar{f} \colon \Gamma.A \to \Gamma'.A'$ over $f$, we get a map $\Delta.(\lambda,\mu)^*A \to \Delta.(f\lambda,\bar{f}\mu)^*A'$ over $\Delta$, functorially in $(f,\bar{f})$; and in the case that $A = f^*A'$ and $\bar{f} = f.A$, then $(\lambda,\mu)^*A' = (\lambda,\mu)^*A'$ and this map is the identity.
  \item \label{item:functoriality-of-reedy-limits-in-extension}
    Functoriality in $i \colon F \extto G$: Given data as above together with a finite extensions $j \colon F' \to G'$ and map of extensions $(h,k) \colon j \to i$, we have $(\lambda k,\mu h)^*A = (\lambda,\mu)^*A $, and $((\lambda,\mu).A) k = (\lambda k,\mu h).A$.
  \item Functoriality in $\C$: Given $F \extto G$, $\Gamma$, $A$, $\Delta$, $\lambda$, $\mu$ as above, and a CwA map $H \colon \C \to \D$ preserving the matching objects of $A$ (and hence sending $A$ to a weak Reedy type $HA$ in $\D$), we have $H((\lambda,\mu)^*A) = (H\lambda,H\mu)^*HA$ and $H((\lambda,\mu).A) = (H\lambda,H\mu).HA$.
    (Note that the first equality here is an on-the-nose equality of context extensions, not merely an isomorphism of objects of $\D$.)
  \item \label{item:cwa-maps-preserve-reedy-limits} 
    Preservation under CwA maps: Given $F \extto G$, $\Gamma$, $A$, $\Delta$, $\lambda$, $\mu$ as above, and an arbitrary CwA map $H \colon \C \to \D$ (not assumed to preserve the matching objects of $A$), the universal property of $(\lambda,\mu)^*A$ and $(\lambda,\mu).A$ as an $i$-pullback is preserved by $H$.
  \end{enumerate}
\end{lemma}

\begin{proof}
  The functoriality statements are routine to verify, immediate in the single-step case and extending by induction to the general case.
  The preservation statement comes in the single-step case from the fact that CwA maps preserve canonical pullbacks of types, and again extends to the general case straightforwardly by induction.
\end{proof}

Reedy limits are moreover functorial in the inverse category $\I$, but the statement of this functoriality is slightly less straightforward.

For a functor $F \colon \C \to \D$, write $F_!$ and $F^*$ for its pushforward/restriction functors:
\[ \begin{tikzcd}[column sep=huge]
  \copsh{\C} \ar[r,bend left=15,"F_!"] \ar[r,phantom,"\bot"] & \copsh{\D} \ar[l,bend left=15,"F^*"]
\end{tikzcd} \]

\begin{lemma}
  Given a discrete opfibration $u \colon \J \to \I$, we have $u_! ( \bdry{j} \extto \yon{j} ) \iso (\bdry{uj} \extto \yon{uj})$ for each $j \in \J$.
  More generally, $u_!$ sends finite extensions in $\copsh{\J}$ to finite extensions in $\copsh{\I}$.
\end{lemma}

\begin{proof}
  Since $u$ is a discrete opfibration, $u_!$ can be computed by $u_!(F)(i) = \coprod_{j \in \J_i} F(j)$.
  %
  %
  The desired isomorphism is immediate.
  As a left adjoint, $u_!$ also preserves pushouts, so the action on finite extensions follows, viewing them as cell complexes of the boundary inclusions.
\end{proof}

\begin{lemma}
  For any category $\C$, object $X$ in $\C$, diagram $Y \colon \I \to \C$, and co-presheaf $F \in \copsh{\I}$, there is an isomorphism (natural in $F$, $X$, and $Y$) between $u_!F$-cylinders $\lambda \colon X \to_{u_!F} Y$ and $F$-cylinders $\lambda' \colon X \to_F u^*Y$. 
\end{lemma}

\begin{proof}
  \hfill $\begin{aligned}[t]
    \Cyl_{u_!F}(X,Y) & \iso \copsh{\I}(u_!F, \C(X,Y_\bullet)) \\
                      & \iso \copsh{\I}(F, \C(X,Y_{u\bullet})) \\
                      & \iso \Cyl_F(X,u^*Y) & \hspace{10em} \qedhere
  \end{aligned}$
\end{proof}

\begin{lemma}[Functoriality of Reedy limits in the domain] \label{lem:functoriality-of-reedy-limits-in-domain}
  Suppose given a discrete opfibration of inverse categories $u \colon \J \to \I$, along with $\Gamma \in \C^\I$, a weak Reedy $\I$-type $A$ over $\Gamma$, a finite extension $F \extto G$ in $\copsh{\I}$, an object $\Delta \in \C$, and cylinders $\lambda \colon \Delta \to_{u_!F} \Gamma$, $\mu \colon \Delta \to_{u_!F} \Gamma.A$ (or equivalently $\lambda' \colon \Delta \to_F u^* \Gamma $, $\mu' \colon \Delta \to u^* \Gamma . u^* A$).

  Then we have $u^*((\lambda,\mu)^*A) = (\lambda',\mu')^*(u^*A)$; and moreover the cylinders $(\lambda,\mu).A$ and $(\lambda',\mu').(u^*A)$ correspond in the sense of the preceding lemma.
\end{lemma}

\begin{proof}
  Direct from the construction of $(\lambda,\mu)^* A$, together with the action of $u_!$ on finite extensions.
\end{proof}

\begin{lemma}[Reedy limits for context extensions] \label{lem:reedy-limits-for-context-extensions}
  Let $\AA$ be a context extension of weak Reedy types over $\Gamma$ in $\C^\I$, and let $i \colon F \extto G$, $\lambda \colon \Delta \to_G \Gamma$, and $\mu \colon \Delta \to_F \Gamma.A$ be as in the master lemma.

  Then we obtain a context extension $(\lambda,\mu)^*A \in \Ty^*{\Delta}$ (with length $mn$, where $m$ is the length of $\AA$ and $n$ the length of $i \colon F \extto G$), along with $(\lambda,\mu).A \colon \Delta.(\lambda,\mu)^*A \allowbreak \to_G \Gamma.A$, forming an $i$-relative pullback of $\pi_{\AA}$ along $(\lambda,\mu)$ as in the master lemma.

  Moreover, this construction enjoys all the functoriality properties of Lemmas \ref{lem:first-properties-of-reedy-limits} and \ref{lem:functoriality-of-reedy-limits-in-domain}.
\end{lemma}

\optionaltodo{Was there any reason (either expository or mathematical) not to just give this as part of the master lemma itself??}

\begin{proof}
  Built from Reedy limits given by the master lemma, just like the generalisation there from a single-step finite extension to the multi-step case.
\end{proof}

\subsection{Strict Reedy types}

Each weak Reedy type comes supplied, by definition, with matching objects.
However, given a little extra data on the domain category $\I$, the master lemma provides us with canonical matching objects, as context extensions of the objects in the base.

\begin{definition} \label{def:ordering}
  By an \defemph{ordered inverse category} we mean an inverse category $\I$ equipped with, for each $i$, a total order on the set of morphisms out of $i$ extending the precedence ordering of $i/\I$, i.e.\ such that for any composable $f, g$, with $f$ non-identity, we have $fg < g$.

  A discrete fibration $u \colon \J \to \I$ between ordered inverse categories is \defemph{ordered} if it respects the given orderings.
\end{definition}

  Choosing such structure on an inverse category $\I$ amounts precisely to giving each $0 \to \yon{i}$ the structure of a finite extension; or equivalently, each $0 \to \bdry{i}$, since the ordering on $\yon{i}$ must have $\id_i$ as its top element.
  This extra data is not burdensome to supply:
\begin{proposition} \label{prop:orderable}
  Any inverse category $\I$ admits some ordering. 
\end{proposition} 

\begin{proof}
  Any finite partial order can be extended to a total order.
\end{proof}

\begin{example}
  We will consider $\SpanCat$, the category $(0 \from 01 \to 1)$, as ordered by $0 <_{01} 1$.
  (The rest of the ordering is determined.)
\end{example}

\begin{definition} \label{def:canonical-matching-object}
  Once $\I$ is ordered, then given a weak Reedy type $A = (Y,p,M,A,\varphi)$ over a diagram $\Gamma \in \C^\I$, we can for any $i \in \I$ apply the master lemma with the finite extension $0 \to \bdry{i}$ to construct a relative matching object for $p_A$ as observed in Remark~\ref{rem:matching-object-as-relative-pullback}:
  \[ \begin{tikzcd}
    \Gamma_i.A_{\bdry{i}} \ar[d,fibb] \ar[r,"\bdry{i}"' very near end] & \Gamma.A \ar[d,fib] \\
    \Gamma_i \ar[r,"\bdry{i}"' very near end] & \Gamma.
    \end{tikzcd}
  \]

  We call this the \defemph{canonical matching object} for $A$ at $i$, and denote it (as a context extension) by $A_{\bdry{i}} \in \Ty^* \Gamma_i$.
  It has a canonical isomorphism to the given matching objects of $A$, $\psi_i \colon \Gamma_i.A_{\bdry{i}} \iso M_iA$.
  We write $A_{\yon{i}}$ for the further context extension $A_{\bdry{i}}.\psi^*A_i \in \Ty^* \Gamma_i$ (which may be constructed directly as a Reedy limit along the finite extension $0 \to \yon{i}$).
\end{definition}

\begin{varnotation}{{\getrefnumber{def:canonical-matching-object}a}} \label{notation:reedy-limits}
  %
  Generally, for a context extension $\AA$ of weak Reedy types over a diagram $\Gamma$, we write $\AA_{\bdry{i}}$ and $\AA_{\yon{i}}$ for the corresponding Reedy limits supplied by Lemma~\ref{lem:reedy-limits-for-context-extensions},
  and $M_i\AA$ for the resulting matching objects $\Gamma_i.\AA_{\bdry{i}}$.
  
  Correspondingly, for a map $f \colon \Gamma . \AA \to \Gamma. \BB$ over $\Gamma$, we write $M_i f$ for the induced map $M_i \AA \to M_i \BB$.
  An important special case is when $\AA$ is the empty extension, so $f$ is a section of $\BB$.
  Then $M_i \AA$ is just $\Gamma$, and $M_i f$ is a section of $\BB_{\bdry{i}}$.
\end{varnotation}

We can now tie up the dangling thread from above:

\begin{proposition} \label{prop:weak-reedy-comp-cat}
  Let $\I$ is an inverse category.
  Then for any CwA $\C$, weak Reedy $\I$-types in $\C$ form a comprehension category $\C^\I_{\wkreedy}$, and this is functorial in $\C$ and $\I$: a discrete opfibration $u \colon \J \to \I$ induces a functor $u^* \colon \C^\I \to \C^\J$, and a CwA map $F \colon \C \to \D$ induces a functor $F^\I \colon \C^\I \to \D^\I$, each commuting with composition and identities.%
  \footnote{A 2-functoriality statement would of course be nicer here; unfortunately we cannot give it, since we have followed the literature in considering CwAs just as a 1-category.}
\end{proposition}

\begin{proof}
  Fix some ordering on $\I$, as in Proposition~\ref{prop:orderable}.
  As noted in Proposition~\ref{prop:weak-reedy-almost-comp-cat}, we just need to give cartesian lifts, and show they are sent to pullbacks by comprehension.
  
  For construction of the lifts, suppose $f \colon \Gamma' \to \Gamma$ is a map in $\C^\I$, and $A = (Y,p,M_i,A,\varphi)$ a weak Reedy type over $\Gamma$.
  We construct $f^*A$, along with the map $f.A \colon f^*A \to A$ over $f$, by induction on $i \in \I$ (under the precedence ordering).
  Suppose $f^*A$ and $f.A$ have been constructed in all levels $j < i$.
  Then this suffices to construct a matching object $M_i(f^*A)$ as a Reedy limit (using the given ordering on $\I$), and also the induced map $M_i(f.A) \colon M_i(f^*A) \to M_iA$ over $f_i \colon \Gamma'_i \to \Gamma_i$ (following Notation~\ref{notation:reedy-limits}).
  Now pulling back $A_i \in \Ty(M_iA)$ along $M_i(f.A)$ gives us the type $(f^*A)_i$; taking $(\Gamma'.f^*A)_i$ as exactly $M_i(f^*A).f^*A_i$ completes the construction of $f^*A$ and $f.A$ in level $i$:
  \[ \begin{tikzcd}
     (\Gamma'.f^*A)_i \ar[dr] \ar[r,"\id"'] \ar[rrr,bend left=10,shift left=0.6ex,"(f.A)_i" description] & M_i(f^*A).(f^*A)_i \ar[r] \ar[d,fib] \ar[drpb] & M_iA.A_i \ar[d,fib] \ar[r,"\iso"'] & (\Gamma.A)_i \ar[dl] \\
     & M_i(f^*A) \ar[r,"M_i(f.A)"] \ar[drpb] \ar[d,fibb] & M_i A \ar[d] \\
     & \Gamma'_i \ar[r,"f_i"] & \Gamma_i
  \end{tikzcd} \]
  Cartesianness of this lift is immediate.

  To see that the comprehension of $f.A$ is a pullback in $\C^\I$, we show by induction that it is a levelwise pullback.
  Its $i$-th component is as in the diagram above.
  The upper square is a pullback by construction, while the lower is a pullback by the universal property of matching objects together with the assumption that $\chi(f.A)_j$ is a pullback for all $j < i$.

  Functoriality in $\I$ was already noted in Proposition~\ref{prop:weak-reedy-almost-comp-cat}.
  For functoriality in $\C$, note again that matching objects of a weak Reedy type are up to isomorphism Reedy limits, and so preserved by the action of CwA maps.
\end{proof}

\begin{definition} \label{def:strict-reedy-types}
  A \defemph{(strict) Reedy type} over $\Gamma$ is a weak Reedy type $A = (Y,p,M,A,\varphi)$ over $\Gamma$ such that for each $i \in \I$,
  \begin{enumerate}
  \item the given matching objects $M_i p$ are precisely the matching objects $\Gamma_i.A_{\bdry{i}}$ supplied by the master lemma (including their associated cylinders);
  \item $(\Gamma.A)_i = \Gamma. A_{\bdry{i}} . A_i$, and $\varphi_i = \id_{\Gamma_i}$.
  \end{enumerate}

  Write $\T_{\reedy(\C,\I)}$ for the full subcategory of $\T_{\wkreedy(\C,\I)}$ on strict Reedy types.
  We will generally write just “Reedy types” to mean the strict ones, except when explicitly contrasting them with weak Reedy types.
\end{definition}


\begin{proposition} \label{prop:strict-reedy-types} \leavevmode
  \begin{enumerate}
  \item For any $f \colon \Gamma' \to \Gamma$ in $\C^\I$ and weak Reedy type $A$ over $\Gamma$, the reindexing $f^*A$ constructed in Proposition~\ref{prop:weak-reedy-comp-cat} is in fact a strict Reedy type.
  \item For any $f \colon \Gamma' \to \Gamma$ in $\C^\I$ and weak Reedy type $A$ over $\Gamma$, if $B$, $B'$ are strict Reedy types over $\Gamma'$ with maps $\bar{f} \colon B \to A$ and $\bar{f}' \colon B' \to A$ over $f$, then $B = B'$ and $f = f'$. 
  \item $\T_{\reedy(\C,\I)}$ forms a discrete subfibration of $\T_{\wkreedy(\C,\I)}$ over $\C^\I$.
  \item For an ordered discrete opfibration $u \colon \J \to \I$, the contravariant action $u^*$ on weak Reedy types restricts to an action on strict Reedy types.
  \item Sending each weak Reedy type $A$ over $\Gamma$ to its trivial reindexing $\id_\Gamma^* A$ gives a \defemph{strictification} equivalence $\str \colon \T_{\wkreedy(\C,\I)} \to \T_{\reedy(\C,\I)}$ over $\C$. \optionaltodo{Consider notation for $\str$.}
  \item For a CwA map $F \colon \C \to \D$, composing the action $F^\I \colon \T_{\wkreedy(\C,\I)} \to \T_{\wkreedy(\C,\J)}$ with strictification gives an action on strict Reedy types, strictly functorial in $F$. \optionaltodo{If we don’t generalise to mention pseudo maps of CwA’s, then we can remove mention of strictification here.}
  \end{enumerate}
\end{proposition}

\begin{proof}
  Parts not specifically noted here follow directly from earlier ones.
  (2) is straightforward by induction on levels, using the fact that $\int \Ty$ is a discrete fibration.
  For (4), note that when $u \colon \J \to \I$ is ordered, the isomorphism $u_! \bdry{j} \iso \bdry{uj}$ is a map of finite extensions; naturality of Reedy limits (Lemma~\ref{lem:first-properties-of-reedy-limits}(\ref{item:functoriality-of-reedy-limits-in-extension})) then implies that the canonical matching objects are preserved by $u^*$.
  (6) similarly uses naturality of Reedy limits under CwA maps (Lemma~\ref{lem:first-properties-of-reedy-limits}(\ref{item:cwa-maps-preserve-reedy-limits})).
\end{proof}

These observations justify the following definition:

\begin{definition} \label{def:cwa-of-strict-reedy-types}
  We write just $\C^\I$ for \defemph{the CwA of strict Reedy types}, with base category $\C^\I$ and presheaf $\Ty$ corresponding to the discrete fibration $\T_{\reedy(\C,\I)}$.
  (Occasionally for emphasis or disambiguation we may write $\C^\I_{\reedy}$ or $\Ty_{\reedy(\C,\I)}$.)
  This construction is covariantly functorial in CwA maps, and contravariantly in ordered discrete opfibrations.
  
  When we speak of types over diagrams $\Gamma \in \C^\I$, we will always mean strict Reedy types unless specified otherwise, and similarly for context extensions, etc.
\end{definition}

\begin{example}
  A direct description of the CwA $\C^\SpanCat$ is spelled out in detail in \jpaacite[Definition~5.5 et seq.]{kapulkin-lumsdaine:homotopy-theory-of-type-theories}. 
\end{example}

\begin{remark}
  We will often construct strict Reedy types and maps between them by induction, as with the reindexings in Proposition~\ref{prop:weak-reedy-comp-cat}.
  Such constructions often involve applying lemmas about Reedy types to the type under construction, before it is fully defined.
  
  For instance, suppose one has earlier proved a lemma that “Reedy limits of levelwise nice Reedy types are nice”, and is now constructing by induction a levelwise nice Reedy type $A$.
  In the induction step, one must construct $A_i$ and prove it nice, assuming that $A_j$ is given and nice in all levels $j < i$.
  Then, in the proof of niceness of $A_i$, one may invoke the lemma to conclude that $M_i A$ is nice.

  Formally, this may be justified by noting that the data given constitute a levelwise nice Reedy type over the subcategory $\I_{<i}$ (or the strict slice $\strictslice{\I}{i}$), and that $M_iA$ can be computed as a Reedy limit of this type, to which the lemma then applies.
  
  For the sake of readability, we will generally avoid belabouring such points.
\end{remark}

\subsection{Levelwise extensions}

Besides Reedy types, we will in places make use of \emph{levelwise} context extensions of diagrams.

\begin{definition}
  Let $\I$ be an arbitrary category, and $\Gamma \in \C^\I$ a diagram.
  A \defemph{levelwise extension} $\AA$ over $\Gamma$ is a family of context extensions $\AA_i \in \Ty^*(\Gamma_i)$, for each $i \in \I$, along with maps in $\C$ making the objects $\Gamma_i.A_i$ and their projection maps into a diagram over $\Gamma$ (which we denote $\Gamma.\AA$).

  Write the set of levelwise extensions over $\Gamma$ as $\Ty_\lw(\Gamma)$; these constitute an evident CwA structure on the category $\C^\I$, which we write as $\C^\I_\lw$.
  Moreover, the construction of $\C^\I_\lw$ is functorial in CwA maps $\C \to \D$ and arbitrary functors $\J \to \I$.
  
  In case $\I$ is an ordered inverse category, there is an evident natural map from Reedy types to levelwise extensions, commuting with the context extension operation; that is, a strict CwA map $\C^\I \to \C^\I_\lw$, acting as the identity as on the base category.
\end{definition}

The CwA $\C^\I_\lw$ is in general of less intrinsic interest than $\C^\I$, since it will not typically inherit as much logical structure from $\C$.
However, levelwise extensions appear naturally as intermediate stages in several constructions/proofs on Reedy types.

In diagrams, we denote projections from levelwise extensions by $\Gamma.\AA \lwto \Gamma$.


\section{Logical structure on inverse diagrams} \label{sec:logical-structure}

The goal of this section is to show that if $\C$ carries $\Id$-types (resp.\ $\Sigma$, unit, $\Pi$, funext), then so does $\C^\I$, for any ordered inverse category $\I$. 

As the statement suggests, each piece of logical structure lifts from $\C$ to $\C^\I$ individually (except of course for functional extensionality, which relies on $\Id$ and $\Pi$);
we will tackle them one at a time in the propositions of this section.
In the course of this, we will also require a few technical lemmas on the interaction of Reedy limits with elimination structures and equivalences.

Fix the CwA $\C$ and ordered inverse category $\I$ throughout.
When we say that constructions are functorial in $\C$ and $\I$, we mean with respect to CwA maps $\C \to \D$ preserving whatever logical structure is under consideration, and ordered discrete opfibrations $\J \to \I$.

\subsection{Elimination structures in inverse diagrams}

\begin{lemma} \label{lem:elim-structure-from-levelwise}
  Suppose $f \colon \Gamma.\AA \to \Gamma.\BB$ is a map over $\Gamma$ in $\C^\I$ equipped with a levelwise elimination structure, i.e.\ an elimination structure on each $f_i$ over $\Gamma_i$.
  Then $f$ carries an elimination structure over $\Gamma$ in $\C^\I$.
  Moreover, this is stable under pullback in $\Gamma$, and functorial in $\C$ and $\I$.
\end{lemma}

\begin{proof}
  First we give the pre-elimination structure.
  Suppose we have a type $C$ over $\Gamma.\BB$, and map $d \colon \Gamma.\AA \to \Gamma.\BB.C$ over $\Gamma.\BB$.
  We wish to construct a section $e$ of $C$ such that $ef = d$.
    \[ \begin{tikzcd}[column sep=small]
       \Gamma.\AA\ar[rr,"d"] \ar[dr,"f"'] & & \Gamma.\BB.C \ar[dl,fib] \\
      & \Gamma.\BB \ar[ur,dashed,bend left=15,shift left=0.5ex,"e"]
  \end{tikzcd} \]

  Work by Reedy induction: suppose that $e$ has already been constructed in degrees below $i$, satisfying the desired equations.
  Then these components assemble to give a section $e'$ of $M_i C \fibbto (\Gamma.\BB)_i$, forming a commutative square from $f_i$ to $p_{C_i}$; we take $e_i$ as the diagonal filler for this square supplied by the elimination structure of $f_i$:
  \[ \begin{tikzcd}
    (\Gamma.\AA)_i \ar[r,"d_i"] \ar[d,"f_i"] & (\Gamma.\BB.C)_i \ar[d,fib] \\
    (\Gamma.\BB)_i \ar[r,bend left=15,"e'"] \ar[ur,dashed,"e_i"] & M_i C \ar[l,fibb,shift left=0.2ex] 
  \end{tikzcd} \]

  Next, for the Frobenius condition, note that the components of any pullback $f.\CC$ of $f$ along a context extension $\pi_\CC \colon \Gamma.\BB.\CC \fibbto \Gamma.\BB$ are just pullbacks of the components of $f$ along the context extensions $\CC_{\yon{i}}$.
  So any such pullback $f.\CC$ again carries a levelwise elimination structure, and hence a pre-elimination structure by the previous paragraph.

  Finally, note that the above constructions are all stable under pullback in $\Gamma$, so they assemble to produce an elimination structure stably in $\Gamma$, as required; and similarly, are functorial in $\C$ and $\I$.
\end{proof}

So levelwise elimination structures suffice to give elimination structures in $\C^\I$.
However, for some purposes one needs a slightly stronger notion.

\begin{definition} \label{def:comparison-component} \label{def:reedy-elim-structure}
  Let $f \colon \Gamma.\AA \to \Gamma.\BB$ be a map over $\Gamma$ in $\C^\I$.
  For each $i \in \I$, let $m_i f$ denote the \defemph{comparison component} $m_i f \colon (\Gamma.\AA)_i \to M_i \AA . (M_i f)^* B_i$.
  \[ \begin{tikzcd}
    (\Gamma.\AA)_i \ar[r,dashed,"m_i f"']\ \ar[rrr,bend left=14,shift left=0.8ex,"f_i" description] \ar[dr] & M_i \AA . (M_if)^*\BB_i \ar[r,"M_i f . \BB_i"] \ar[d,fibb] \ar[drpb] & M_i \BB . \BB_i \ar[d,fibb] \ar[r,"{\varphi_{\BB,i}}"',"\iso"] & (\Gamma.\BB)_i \ar[dl] \\
    & M_i \AA \ar[r,"M_i f"] & M_i \BB
  \end{tikzcd} \]
  (where $\varphi_{\BB,i}$ is an evident context-reordering isomorphism).
  By a \defemph{Reedy elimination structure on $f$ over $\Gamma$}, we mean a choice of elimination structure on each $m_i f$ over $M_i \AA$.
\end{definition}

\begin{proposition} \label{prop:closure-for-reedy-elim-structures} \leavevmode
  \begin{enumerate}
    \item Suppose $f, g$ are composable maps over $\Gamma$, both equipped with Reedy elimination structures over $\Gamma$.
    Then the composite $gf$ carries an induced Reedy elimination structure.
  
  \item Suppose $f \colon \Gamma.\AA \to \Gamma.\BB$ is equipped with a Reedy elimination structure.
    Then for any Reedy type $C$ over $\Gamma.\BB$, $f.C$ carries an induced Reedy elimination structure.
  \end{enumerate}
  Moreover, these constructions are functorial in $\I$ and $\C$.
\end{proposition}

\begin{proof}
  Immediate by the same closure properties for elimination structures in $\C$.
\end{proof}

\begin{lemma} \label{lem:elim-structure-on-reedy-limit}
  Suppose $f \colon \Gamma.\AA \to \Gamma.\BB$ is equipped with a Reedy elimination structure over $\Gamma$.
  Then:
  \begin{enumerate}
  \item For any finite extension $F \extto G$ and $\Delta$, $\lambda$, $\mu$ as in the master lemma (Lemma~\ref{lem:master-lemma}), the induced “pullback” map $(\lambda,\mu)^*f \colon \Delta.(\lambda,\mu)^* \AA \to \Delta.(\lambda, f \mu)^* \BB$ carries an induced elimination structure;
    \[ \begin{maxwidthtikzpicture}[cd-style,x={(4cm,0cm)},y={(0cm,2.2cm)},z={(3cm,-1cm)}]
      \node (D) at (0,0,0) {$\Delta$};
      \node (DA) at (-0.1,1,-0.5) {$\Delta.(\lambda,\mu)^*\AA$};
      \node (DB) at (0.1,1,0.5) {$\Delta.(\lambda,f\mu)^*\BB$};
      \node (G) at (2,0,0) {$\Gamma$};
      \node (GA) at (2,1,-0.5) {$\Gamma.\AA$};
      \node (GB) at (2,1,0.5) {$\Gamma.\BB$};
      \arr[fibb] (DA) to (D);
      \arr[fibb] (DB) to (D);
      \arr (DA) to node[commutative diagrams/.cd,description,pos=0.49] {$(\lambda,\mu)^*f$} (DB);
      \arr[fibb] (GA) to (G);
      \arr[fibb] (GB) to (G);
      \arr (GA) to node {$f$} (GB);
      \arr (D) to[sub={G}] node {$\lambda$} (G);
      \arr[shorten >={0.2cm}] (D) to[label={F}{auto=right,pos=0.95}] node[auto=right] {$\mu$} (GA);
      \arr (DA) to[sub={G}] (GA);
      \arr[crossing over] (DB) to[sub={G}] (GB);
    \end{maxwidthtikzpicture} \]

  \item and moreover, $f$ carries a levelwise elimination structure over $\Gamma$ in $\C^\I$.
  \end{enumerate}

  Moreover, these constructions are stable under pullback in $\Gamma$, and functorial in $\C$ and $\I$.
\end{lemma}

\begin{proof}
  For the first statement, we exhibit $(\lambda,\mu)^*f$ as a composite of pullbacks along context extensions of pullbacks of the maps $m_i f$ along maps into the matching objects $M_i A$.

  Specifically, in the case of a single-step extension $F \extto F +_h \yon{i}$, $(\lambda,\mu)^* f$ is a genuine pullback of $m_i f$ along the map $\Delta \to M_i A$ induced by $(\lambda \restr{\yon{i}}, \mu \restr{\bdry{i}})$
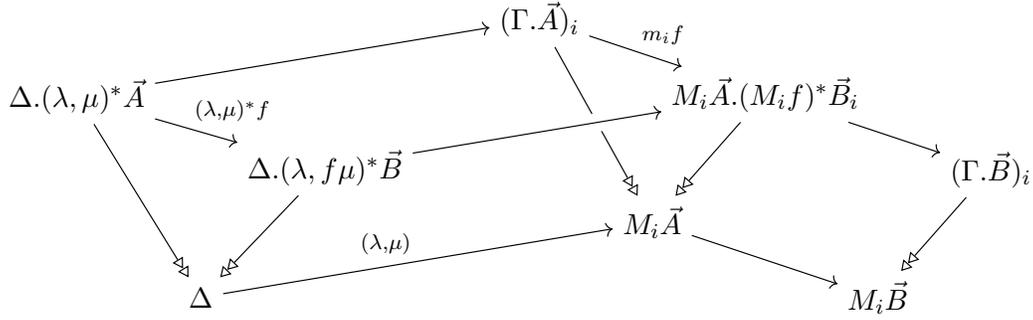
  \[ \begin{maxwidthtikzpicture}[cd-style,x={(3cm,0.5cm)},y={(0cm,2.2cm)},z={(3cm,-1cm)}]
    \node (D) at (0,0,0) {$\Delta$};
    \node (DA) at (-0.05,1,-0.5) {$\Delta.(\lambda,\mu)^*\AA$};
    \node (DB) at (0.05,1,0.5) {$\Delta.(\lambda,f\mu)^*\BB$};
    \node (MiA) at (2,0,0) {$M_i \AA$};
    \node (GAi) at (2,1,-0.5) {$(\Gamma.\AA)_i$};
    \node (MiAB) at (2,1,0.5) {$M_i \AA . (M_if)^*\BB_i $};
    \node (MiB) at (2,0,1) {$M_i \BB$};
    \node (GBi) at (2,1,1.5) {$(\Gamma.\BB)_i$};
    \arr[fibb] (DA) to (D);
    \arr[fibb] (DB) to (D);
    \arr (DA) to node[pos=0.4] {$(\lambda,\mu)^*f$} (DB);
    \arr[fibb] (GAi) to (MiA);
    \arr[fibb] (MiAB) to (MiA);
    \arr (GAi) to node {$m_if$} (MiAB);
    \arr (D) to node {$(\lambda,\mu)$} (MiA);
    \arr (DA) to (GAi);
    \arr[crossing over] (DB) to (MiAB);
    \arr[fibb] (GBi) to (MiB);
    \arr (MiA) to (MiB);
    \arr (MiAB) to (GBi);
  \end{maxwidthtikzpicture} \]
  and, as such, carries an induced elimination structure by stability.

  In the case of a multi-step extension $F \extto G$, write it as a composite of single-step extensions $F = F_0 \extto F_1 \extto \cdots \extto F_n = G$.
  Then, similarly to in the proof of the master lemma, we form a tower:
  
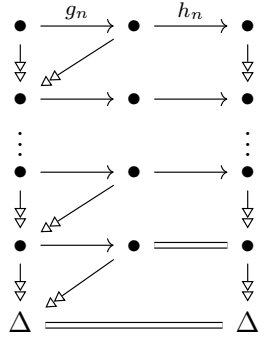
\[ \begin{tikzcd}[row sep=scriptsize]
    \bullet \ar[d,fibb] \ar[r,"g_n"] & \bullet \ar[dl,fibb] \ar[r,"h_n"] & \bullet \ar[d,fibb] \\
    \bullet \ar[r] \ar[d,phantom,description,"\vdots"] & \bullet \ar[r] & \bullet \ar[d,phantom,description,"\vdots"] \\
    \bullet \ar[d,fibb] \ar[r] & \bullet \ar[dl,fibb] \ar[r] & \bullet \ar[d,fibb] \\
    \bullet \ar[d,fibb] \ar[r] & \bullet \ar[dl,fibb] \ar[r,equal] & \bullet \ar[d,fibb] \\
    \Delta \ar[rr,equal] & & \Delta
  \end{tikzcd} \]
  \optionaltodo{Label/explain this diagram more?}
Each of the comparison maps $g_k$ carries an elimination structure by the single-step case, so by induction up the tower using the Frobenius property and composition of elimination structures, each composite $h_k g_k$ carries an elimination structure; but the final such map $h_n g_n$ is isomorphic over $\Delta$ to our desired map $(\lambda,\mu)^*f$.

  For the levelwise elimination structure, note that $f_i$ is canonically isomorphic to the Reedy limit $f_{\yon i} \colon  M_i \AA . \AA_{\yon i} \to M_i \BB . \BB_{\yon i}$, which carries an elimination structure by the first statement.

  Finally, as ever, all constructions used are stable in $\Gamma$, and functorial in $\C$ and $\I$.
\end{proof}

Combining the preceding two lemmas gives:
\begin{corollary} \label{cor:elim-structure-from-reedy}
  A Reedy elimination structure on $f \colon \Gamma.\AA \to \Gamma.\BB$ over $\Gamma$ induces an elimination structure on $f$ over $\Gamma$ in $\C^\I$, stably in $\Gamma$ and functorially in $\C$, $\I$. \qed
\end{corollary}

\subsection{Unit types in inverse diagrams}

We begin the logical structure with the lowest-hanging fruit.

\begin{proposition} \label{prop:reedy-unit-types}
  Suppose $\C$ is equipped with unit-type structure.
  Then so is $\C^\I$.

  Moreover, this construction is functorial in $\C$ (with respect to CwA maps preserving unit types) and in $\I$.
\end{proposition}

\begin{proof}
  Given $\Gamma \in \C^\I$, define the unit Reedy type $1_\Gamma$ over it by induction, taking $(1_\Gamma)_i$ to be $1_{M_i 1_\Gamma} \in \Ty(M_i 1_\Gamma)$.
  Similarly, take $(\star_\Gamma)_i$ to be $\star_{M_i 1_\Gamma} \cdot (M_i \star_\Gamma)$.
  The $i$th comparison component of $\star_\Gamma$ is then exactly $\star_{\Gamma_i}$; so the canonical elimination structures on these maps assemble to form a Reedy elimination structure, and hence by Corollary~\ref{cor:elim-structure-from-reedy} an elimination structure in $\C^\I$.
  \[ \begin{tikzcd}[column sep=huge]
      \Gamma_i . 1_{\Gamma_i} \ar[r] \ar[drpb] \ar[d,fib,bend left]
      & (\Gamma.1_\Gamma)_i \ar[d,fib,bend left]
      \\ \Gamma_i \ar[u,bend left,dashed,"\star_{\Gamma_i}" description] \ar[r,"M_i \star_\Gamma"'] \ar[ur,dashed,"(\star_\Gamma)_i" description]
      & M_i 1_\Gamma \ar[u,bend left,"\star_{M_i 1_\Gamma}" description]
  \end{tikzcd} \]

  These constructions are all stable under reindexing, and so constitute unit-type structure on $\C^\I$.
  Functoriality in $\C$ and $\I$ is routine to check, since all constructions involved (matching objects over $\I$, and the unit-type structure on $\C$ applied at each level) were so functorial. 
\end{proof}

\subsection{Identity types in inverse diagrams}

\begin{proposition} \label{prop:reedy-id-types}
  Suppose $\C$ is equipped with $\Id$-type structure.
  Then so is $\C^\I$.
  Moreover, this construction is functorial in $\C$ and $\I$
\end{proposition}

\begin{proof}
  Let $\Gamma.A \fibto \Gamma$ be a type over a diagram in $\C^\I$.
  We will construct, by the usual Reedy induction over $i \in \I$, the type $\Id_A \in \Ty(\Gamma.A.A)$, together with a map $\refl_A \colon \Gamma.A \to \Gamma.A.A.\Id_A$ over $\Gamma.A.A$, equipped with a Reedy elimination structure over $\Gamma$.

  Assuming we have done this for all $j < i$, we have the map $M_i \refl_A \colon M_i A \to M_i (A. A.\Id_A)$; and by Lemma~\ref{lem:elim-structure-on-reedy-limit}, $M_i \refl_A$ is equipped with an elimination structure over $\Gamma_i$.
  So taking the pullback diagram
  \[\begin{tikzcd}
    M_i A . A_i . A_i \ar[r] \ar[d, fibb] \ar[drpb] & M_i(A.A.\Id_A).A_i.A_i \ar[d,fibb]  \ar[r,"\iso"] & M_i \Id_A \ar[dl] \\
    M_i A \ar[r, "M_i \refl_A"] & M_i(A.A.\Id_A)
  \end{tikzcd}\]
  the top composite $r' \colon M_i.A_i.A_i \to M_i \Id_A$ carries an elimination structure over $\Gamma_i$.
  
  Now take $(\Id_A)_i$ as the descent of $\Id_{A_i}$ along $r'$ as given by Lemma~\ref{lem:descent-from-elim-structure}, and $(\refl_A)_i$ to be the composite of $r'.(\Id_A)_i$ with $\refl_{A_i}$:
  \[\begin{tikzcd}
    M_i A . A_i \ar[ddr,fib] \ar[dr,"\Delta_{A_i}"] \ar[r,"\refl_{A_i}"] & M_i A . A_i . A_i . \Id_{A_i} \ar[d,fib] \ar[rr,dashed,"r'.(\Id_A)_i"] \ar[drpb] & & M_i \Id_A . (\Id_A)_i \ar[d,fib,dashed] \\
    & M_i A . A_i . A_i \ar[r] \ar[rr,bend left=10,"r'"] \ar[d, fibb] \ar[drpb] & M_i(A.A.\Id_A).A_i.A_i \ar[d,fibb]  \ar[r,"\iso"'] & M_i \Id_A \ar[dl] \\
    & M_i A \ar[r, "M_i \refl_A"] & M_i(A.A.\Id_A)
  \end{tikzcd}\]
The new comparison component $m_ir$ is by construction $\refl_{A_i}$; as required, this lies over the comparison component $m_i \Delta_A$ (since that is exactly $\Delta_{A_i}$), and carries an elimination structure over $M_i A$.

  This completes the inductive construction of $\Id_A$, $\refl_A$, and the Reedy elimination structure on $\refl_A$.
  Corollary~\ref{cor:elim-structure-from-reedy} shows that this constitutes an elimination structure on $\refl_A$ in $\C^\I$;
  and all these components are by construction stable under pullback in $\Gamma$, so together, they constitute $\Id$-type structure on $\C^\I$.
  Functoriality in $\C$ and $\I$ is once again routine to check.
\end{proof}

\subsection{Equivalences in inverse diagrams}

With the identity types available, we are now able to speak of homotopies and equivalences in $\C^\I$, and compare them with these notions in $\C$.
The first lemma is a direct analogue of Lemma~\ref{lem:elim-structure-on-reedy-limit} on Reedy limits of elimination structures.

\begin{definition} \label{def:reedy-equivalence}
  Let $w \colon \Gamma.\AA \to \Gamma.\BB$ be a map over $\Gamma$ in $\C^\I$.
  As in Definition~\ref{def:comparison-component}, write $m_i w$ for the comparison component $(\Gamma.\AA)_i \to \Gamma_i.(M_iw)^*\BB_i$.
  We say $w$ is a \defemph{Reedy equivalence} if $m_i w$ is an equivalence for all $i$; similarly, by \defemph{Reedy equivalence data} on $w$ we mean a choice of equivalence data on each $m_i w$.
\end{definition}

\begin{lemma} \label{lem:equivalences-and-reedy-limits}
  Suppose $w \colon \Gamma.\AA \to \Gamma.\BB$ is a Reedy equivalence (resp.\ equipped with Reedy equivalence data) over $\Gamma$ in $\C^\I$.
  Then:
  \begin{enumerate}
  \item for any finite extension $F \extto G$, and any $\Delta$, $\lambda$, $\mu$ as before, $(\lambda,\mu)^*w$ is an equivalence (resp.\ carries induced equivalence data); \label{item:reedy-limit-of-equiv}
  \item and $w$ is a levelwise equivalence (resp.\ each $w_i$ carries equivalence data). \label{item:levelwise-from-reedy-equiv}
  \end{enumerate}
  Moreover, for the “data” version, these constructions are stable under pullback in the base, and functorial in $\C$ and $\I$.
\end{lemma}

\begin{proof}
  Entirely analogous to Lemma~\ref{lem:elim-structure-on-reedy-limit}.
  When $F \extto G$ is a single-step extension by $\bdry{i} \extto \yon{i}$, then $(\lambda,\mu)^*w$ is a pullback of $m_i w$ along a map into the $\Gamma_i$.
  In the multi-step case, $(\lambda,\mu)^*w$ has a canonical decomposition as a composite of context extensions of such single-step pullbacks.
  Since equivalences (resp.\ equivalence data) are preserved by pullback in the base, pullback along context extensions, and composition, the Reedy limit statement follows.

  The levelwise statement follows as a special case, since $w_i$ is isomorphic to $w_{\yon i}$.

  Finally, stability/functoriality of the “data” version follows from stability/functoriality of all steps in the construction.
\end{proof}

\begin{lemma} \label{lem:levelwise-equivalence-iff-reedy}
  Suppose $w \colon \Gamma.\AA \to \Gamma.\BB$ is a map over $\Gamma$ in $\C^\I$.
  Then $w$ is a levelwise equivalence (resp.\ carries levelwise equivalence data) if and only if it is a Reedy equivalence (resp.\ carries Reedy equivalence data).
  Moreover, the maps between levelwise and Reedy equivalence data are stable under pullback in $\Gamma$, and functorial in $\C$ and $\I$.
\end{lemma}

\begin{proof}
  The “if” direction is just part (\ref{item:levelwise-from-reedy-equiv}) of the previous lemma.
  
  For the converse, suppose $w$ is a levelwise equivalence, and work by induction: assume we have shown $m_j w$ is an equivalence for each $j < i$.
  Then $M_i w$ is an equivalence, by part (\ref{item:reedy-limit-of-equiv}) of the previous lemma, applied to $w \restr{i/\I}$ in $\C^{i/\I}$.
  But now $w_i = (M_i w . \BB_i) . (m_i w)$, so by 2 out of 3, $m_i w$ is an equivalence as required.

  Moreover, all the above reasoning extends directly to a construction on equivalence data, stably in the base and functorially in $\C$, $\I$.
\end{proof}

\begin{lemma} \label{lem:reedy-limits-preserve-homotopy}
  Suppose $f, g \colon \Gamma.\AA \to \Gamma.\BB$ are homotopic over $\Gamma$ in $\C^\I$.
  Then they are levelwise homotopic in $\C$, i.e.\ $f_i, g_i \colon (\Gamma.\AA)_i \to (\Gamma.\BB)_i$ are homotopic over $\Gamma_i$.
  Moreover, this map from homotopies to levelwise homotopies is stable in $\Gamma$ and functorial in $\C$, $\I$.
\end{lemma}

\begin{proof}
  Take $h \colon \Gamma.\AA \to \Gamma.\BB.\BB.\Id_{\BB}$ to be some homotopy.
  Now for each $i$, we have $h_i \colon \Gamma_i.\AA_{\yon i} \allowbreak \to \Gamma_i.\BB_{\yon i}.\BB_{\yon i}.(\Id_{\BB})_{\yon i}$, over $(f_i,g_i) \colon \Gamma_i.\AA_{\yon i} \to \Gamma_i.\BB_{\yon i}.\BB_{\yon i}$.
  To convert these into the desired levelwise homotopies, we need to construct maps $\Gamma_i.\BB_{\yon i}.\BB_{\yon i}.(\Id_{\BB})_{\yon i} \to \Gamma_i.\BB_{\yon i}.\BB_{\yon i}.\Id_{\BB_{\yon i}}$, over $\Gamma_i.\BB_{\yon i}.\BB_{\yon i}$.

  For this, it suffices to give an elimination structure on each map
  \[ (r_\BB)_i \colon \Gamma_i.\BB_{\yon i} \to \Gamma_i.\BB_{\yon i}.\BB_{\yon i}.(\Id_{\BB})_{\yon i}, \]
  or in other words, to show that $(\Id_{\BB})_{\yon i}$ is itself an alternative identity context for $\BB_{\yon i}$.
  By Lemma~\ref{lem:elim-structure-on-reedy-limit}, it suffices to give a Reedy elimination structure on $r_\BB$ over $\Gamma$.
  But this follows by Proposition~\ref{prop:closure-for-reedy-elim-structures} and the remark in Definition~\ref{def:identity-context}: $r_\BB$ is, up to isomorphism, a composite of isomorphisms and context extensions of the reflexivity maps of individual Reedy types, which by construction carry Reedy elimination structures.
\end{proof}

We are now equipped to completely characterise the equivalences of $\C^\I$ within fibrant slices, and give a recognition condition for them in general.

\begin{lemma} \label{lem:equivalence-iff-levelwise} 
  A map $f \colon \Gamma.\AA \to \Gamma.\BB$ over $\Gamma$ is an equivalence (resp.\ carries equivalence data) in $\C^\I(\Gamma)$, and hence $\C^\I$, if and only if it is a levelwise equivalence (resp.\ carries levelwise equivalence data).
    Moreover, these constructions on equivalence data are stable in $\Gamma$ and functorial in $\C$, $\I$.
\end{lemma}

\begin{proof}
  First, suppose $f$ is an equivalence in $\C^\I(\Gamma)$.
  Pick some homotopy section $(g,\alpha)$ and homotopy retraction $(g',\beta)$.
  By Lemma~\ref{lem:reedy-limits-preserve-homotopy}, these give levelwise homotopy sections and retractions of $f$; so $f$ is a levelwise equivalence.

  Conversely, suppose $f$ is a levelwise equivalence.
  Without loss of generality, we may assume $f$ is a dependent projection $\Gamma.\AA \to \Gamma$.
  
  (For the general case, first factor $f$ as an equivalence $f' \colon \Gamma.\AA \to \Gamma.\BB.\AA'$ in $\C^\I(\Gamma)$ followed by a dependent projection $\pi_{\AA'}$, according to \jpaacite[Lem.~11]{gambino-garner}.
  By the first direction of the present lemma, $f'$ is a levelwise equivalence; so by 2 out of 3, so is $\pi_{\AA'}$.
  So by the special case, $\pi_{\AA'}$ is an equivalence in $\C(\Gamma.\BB)$, hence in $\C(\Gamma)$; and by 2 out of 3, so is $f$.)
  
  For the case of a dependent projection, first recall that a type is contractible just if it is inhabited and a proposition \cite[Lemma~3.3.2]{\arxivorsubmission{hott:book}{hott:book-for-jpaa}}.
  That is, a dependent projection $\pi_B \colon \Delta.B \to \Delta$ (in any CwA with $\Id$-types) carries equivalence data if and only if it can be equipped with a section $b$ of $B$, together with  “proposition data”, i.e.\ a section $h$ of $\Id_B$ over $\Delta.B.B$ (and the operations witnessing this are stable in the base context).
  
  With this in mind, suppose $\pi_A \colon \Gamma.A \to \Gamma$ is a levelwise equivalence.
  By Lemma~\ref{lem:levelwise-equivalence-iff-reedy}, it is then also a Reedy equivalence.
  But its comparison maps $m_i \pi_A$ are exactly the projections $\pi_{A_i} \colon (\Gamma.A)_i \fibto M_i A$; so these are equivalences, and we may choose for each $i$ sections $s_i$ of $A_i$, and $g_i$ of $\Id_{A_i}$.
   
  From these, we will construct a sections $a$ of $A$ and $h$ of $\Id_A$, in $\C^\I$.
  The former is direct by Reedy induction, taking each $a_i$ to be the composite $\Gamma_i \to[m_iA] M_i A \to[s_i] (\Gamma.A)_i$.
  
  For the latter, first take for each $i$ a section $g'_i$ of $(\Id_A)_i$ as in the following diagram:
  \[\begin{tikzcd}
    M_i A . A_i . A_i . \Id_{A_i} \ar[d,fib,bend left,"\pi_{\Id_{A_i}}"] \ar[rr] \ar[drpb] & & (\Gamma.A.A.\Id_A)_i \ar[d,fib,bend left,"\pi_{(\Id_A)_i}"] \\
    M_i A . A_i . A_i \ar[r,"r'"] \ar[u,bend left,"g_i"] & M_i(A.A.\Id_A).A_i.A_i \ar[r,"\iso"] & M_i \Id_A \ar[u,bend left,dashed,"g_i'"]
  \end{tikzcd}\]
  where $r'$ and the isomorphism are as in the construction of identity types (Proposition~\ref{prop:reedy-id-types}), and as noted there $r'$ carries an elimination structure, which we apply to $g_i$ to yield $g_i'$.
  Now we get the desired section $h$ of $\Id_A$ by Reedy induction, setting $h_i = g_i' M_ih$:
  \[\begin{tikzcd}
    & (\Gamma.A.A.\Id_A)_i \ar[d,bend left, fib] \ar[dr,fibb,"(\pi_{\Id_A})_i"] & \\
    (\Gamma.A.A)_i \ar[r,"M_i h"] \ar[ur,dashed,"h_i"] & M_i(\Id_A) \ar[u,bend left,"g_i'"] \ar[r,fibb] & (\Gamma.A.A)_i
  \end{tikzcd}\]
  (where in $M_i h$, we consider $h$ as a map between extensions of $\Gamma.A.A$, not just of $\Gamma$).
\end{proof}

\begin{lemma} \label{lem:equivalence-if-levelwise} 
  If $f \colon \Gamma \to \Delta$ is a levelwise equivalence, then it is an equivalence in $\C^\I$.
\end{lemma}

\begin{proof}
  First, for weak type lifting, suppose we are given $A \in \Ty(\Gamma)$.
  We will construct a type $B$ over $\Delta$, together with a Reedy equivalence $e \colon \Gamma.A \to \Gamma.f^*B$ over $\Gamma$.
  Once this is done, it follows by Lemmas~\ref{lem:equivalences-and-reedy-limits} and~\ref{lem:equivalence-iff-levelwise} that $e$ is an equivalence in $\C^\I$, and so exhibits $B$ as a weak lift of $A$ as required.

  To construct $B$ and $e$, we work by induction.
  Suppose they have been constructed in degrees $<i$.
  We now have maps $M_i e \colon M_i A \to M_i (f^*B)$, an equivalence by Lemma~\ref{lem:equivalences-and-reedy-limits} and our inductive assumption on $e$, and $M_i(f.B) \colon M_i (f^*B) \to M_i B$, an equivalence by right properness.
  So we can take a weak lifting of $A_i$ along the composite equivalence $(f_i.B_\bdry{i})(M_i e) \colon M_i A \to M_i B$, to obtain $B_i \in \Ty(M_i B)$ and an equivalence $w_i \colon M_i A . A_i \to M_i A . ((f_i.B_\bdry{i}) M_i e)^* B_i$.
  
  \[ \begin{tikzcd}
  M_A.A_i \ar[d,fib] \ar[r,dashed,"w_i",weq'] & M_i A .((f_i.B_\bdry{i})M_ie)^* B_i \ar[dl,fib] \ar[r,dashed,weq'] \ar[drpb] & M_i B . B_i \ar[d,fib,dashed] \\
  M_i A \ar[d,fibb] \ar[r,"M_i e",weq'] & M_i (f^*B) \ar[dl,fibb] \ar[drpb] \ar[r,"M_i (f.B)",weq'] & M_i B \ar[d,fibb] \\
  \Gamma_i \ar[rr,weq',"f_i"] & & \Delta_i
  \end{tikzcd} \]  

  Now taking $e_i$ as $((M_i e).(f_i.B_\bdry{i})^* B_i)w_i \colon (\Gamma.A)_i \to M_i A . (f_i.B_\bdry{i})^* B_i = (\Gamma.f^*B)_i$, its comparison component at $i$ is exactly $w_i$; so $e$ is a Reedy equivalence in degree $i$ as required.

  For weak term lifting, the argument is similarly straightforward.
  The lifting term and homotopy are built by induction, using at each stage first transport along the homotopy in lower degrees, and then weak term lifting along a right-proper pullback of a component of $f$.
  %
  %
  %
  \end{proof}

Summarising the two previous lemmas, we have:

\begin{lemma} \label{lem:equivalences-are-levelwise} 
  \leavevmode \begin{enumerate}
  \item A map $f \colon \Gamma.\AA \to \Gamma.\BB$ over $\Gamma$ is an equivalence (resp.\ carries equivalence data) in $\C^\I(\Gamma)$, and hence $\C^\I$, if and only if it is a levelwise equivalence (resp.\ carries levelwise equivalence data).
    Moreover, these constructions on equivalence data are stable in $\Gamma$ and functorial in $\C$, $\I$.
    
  \item A map $f \colon \Gamma \to \Delta$ is an equivalence in $\C^\I$ if it is a levelwise equivalence.
    \qed
  \end{enumerate}
\end{lemma}

\begin{remark}
  Note that the slightly careful phrasing above is necessary: in general, an equivalence $f \colon \Gamma \to \Delta$ in $\C^\I$ may fail to be levelwise.
  For instance, consider the case where:
  \begin{itemize}
  \item $\I$ is the arrow category $(1 \rightarrow 0)$; 
  \item $\C$ is constructed by freely adjoining a terminal object $\bot$ to some non-trivial CwA, e.g.\ $\Set$, and setting $\Ty(\bot) \coloneqq 0$;
  \item $e_0 \colon \Gamma_0 \to \Delta_0$ is $\id_{\bot}$, and $e_1 \colon \Gamma_1 \to \Delta_1$ is any non-equivalence in $\Set$:
    \[ \begin{tikzcd}[sep = small]
      \Gamma_1 \ar[r,"e_1"] \ar[d] & \Delta_1 \ar[d] \\
      \bot \ar[r,"e_0"] & \bot
    \end{tikzcd} \]
  \end{itemize}
  Then $e$ is not a levelwise equivalence, by choice of $e_1$; but it is an equivalence in $\C^\rightarrow$, since there are no Reedy types over $\Gamma$ (as there are no types over $\bot$), so both weak lifting properties hold vacuously.
\end{remark}

\subsection{$\Sigma$-types in inverse diagrams}

\begin{proposition}
  Suppose $\C$ is equipped with $\Sigma$-type structure.
  Then so is $\C^\I$.
  Moreover, this construction is functorial in $\C$ and $\I$.
\end{proposition}

\begin{proof}
  The construction is analogous to the one for identity types in Proposition~\ref{prop:reedy-id-types}.
 
  Let $\Gamma.A.B \fibto \Gamma.A \fibto \Gamma$ be a pair of types over a diagram in $\C^\I$.
  We will construct by Reedy induction the type $\Sigma_A B$ over $\Gamma$, along with its pairing morphism $\pair \colon \Gamma.A.B \to \Gamma.\Sigma_A B$ and a Reedy elimination structure on $\pair$ over $\Gamma$.
  
  As usual, we assume by induction that these are defined on $\I_{<i}$, and wish to define them at $i$.
  We will construct $(\Sigma_A B)_i$ as a $\Sigma$-type of a pair of types $\Abar_i$, $\Bbar_i$ over $M_i\Sigma_A B$.
  
  By the inductive Reedy elimination structure of $\pair$, and Lemma~\ref{lem:elim-structure-on-reedy-limit}, we know that $M_i \pair \colon M_i(A.B) \to M_i(\Sigma_A B)$ carries an elimination structure; so we can take $\Abar_i$, $\Bbar_i$ to be the descent of $A_i$, $B_i$ along $M_i \pair$, according to Lemma~\ref{lem:descent-from-elim-structure}.
  We then set $(\Sigma_A B)_i \coloneqq \Sigma_{\Abar_i} \Bbar_i$, and take ${\pair}_i$ to be the composite of the top edge of the following diagram:
  \[
  \begin{maxwidthtikzpicture}[cd-style,x={(2.8cm,0cm)},y={(0cm,1.8cm)},z={(3cm,-1cm)}]
    \node(MiAB) at (0,0) {$M_i (A.B)$};
    \node(GABi) at (-1,2) {$(\Gamma.A.B)_i$};
    \arr (GABi) to (MiAB);
    \node(MiABA) at (0,1) {$M_i (A.B).A_i$};
    \arr[fib] (MiABA) to (MiAB);
    \node(MiABAB) at (0,2) {$M_i (A.B).A_i.B_i$};
    \arr[fib] (MiABAB) to (MiABA);
    \arr (GABi) to node {$\iso$} (MiABAB);
    \node(MiSig) at (2.4,0) {$M_i \Sigma_A B$};
    \node(MiSigA) at (2.4,1) {$M_i \Sigma_A B.\Abar_i$};
    \arr[fib] (MiSigA) to (MiSig);
    \node(MiSigAB) at (2.4,2) {$M_i \Sigma_A B . \Abar_i.\Bbar_i$};
    \arr[fib] (MiSigAB) to (MiSigA);
    \arr (MiAB) to node {$M_i \pair_{A,B} $} (MiSig);
    \arr (MiABA) to (MiSigA);
    \draw let
      \p1 = (MiABA), 
      \p2 = (MiAB), 
      \p3 = (MiSigA), 
      \p4 = (MiSig), 
      \p5 = ($(\p1)!0.25!(\p2)$),
      \p6 = ($(\p1)!3.5ex!(\p3)$),
      \p7 = ($(\p5)+(\p6)-(\p1)$)
      in
      ($(\p7)!1ex!(\p5)$) -- (\p7) -- ($(\p7)!1ex!(\p6)$);
    \arr (MiABAB) to (MiSigAB);
    \draw let
      \p1 = (MiABAB), 
      \p2 = (MiABA), 
      \p3 = (MiSigAB), 
      \p4 = (MiSigA), 
      \p5 = ($(\p1)!0.25!(\p2)$),
      \p6 = ($(\p1)!3.5ex!(\p3)$),
      \p7 = ($(\p5)+(\p6)-(\p1)$)
      in
      ($(\p7)!1ex!(\p5)$) -- (\p7) -- ($(\p7)!1ex!(\p6)$);
    \node(MiABSig) at (0,2,1) {$M_i (A.B).\Sigma_{A_i} {B_i}$};
    \arr[fib,crossing over] (MiABSig) to (MiAB);
    \arr (MiABAB) to node {$\pair_{A_i,B_i}$} (MiABSig);
    \node(MiSigSig) at (2.4,2,1) {$M_i \Sigma_A B.\Sigma_{\Abar_i} {\Bbar_i}$};
    \arr[fib] (MiSigSig) to (MiSig);
    \arr (MiSigAB) to node {$\pair_{\Abar_i,\Bbar_i}$} (MiSigSig);
    \arr (MiABSig) to (MiSigSig);
  \end{maxwidthtikzpicture}
  \]

  By stability of the $\Sigma$-structure of $\C$, the comparison component $m_i \pair_{A,B}$ is just the composition of the isomorphism $(\Gamma.A.B)_i \iso M_i(A.B).A_i.B_i$ with the pairing map of $\Sigma_{A_i} B_i$, so carries an elimination structure as required.
  Functoriality in $\C$ and $\I$ is once again routine.
\end{proof}

\subsection{$\Pi$-types in inverse diagrams}

\begin{proposition} \label{prop:reedy-pi-types}
  Suppose $\C$ carries $\Pi$-type (resp.\ $\Pi_\eta$-) structure.
  Then so does $\C^\I$.
  Moreover, this construction is functorial in $\C$ and $\I$.
\end{proposition}

\begin{proof}
  First, let $\Gamma.A.B \fibto \Gamma.A \fibto \Gamma$ be a pair of types over a diagram in $\C^\I$; we will construct the type $\Pi[A,B]$, along with the morphism $\ev_{A,B} \colon \Gamma.\Pi[A,B].A \allowbreak \to \Gamma.A.B$ over $\Gamma.A$.
  As usual, we work by induction on $i$, giving these at $i \in \I$ assuming they are given on $\I_{<i}$.
  
  Briefly, we will take $\Pi[A,B]_i \in \Ty(M_i(\Pi[A,B]))$ to be an iterated $\Pi$-type $\Pi[ A_{\yon{i}} , f^* B_i ]$, for a certain map $f \colon M_i \Pi[A,B] . A_{\yon{i}} \to M_i B$, and $\ev_i$ will be the evaluation map of that iterated $\Pi$-type.
  Specifically, this map $f$ is constructed as a “weakening by $A_i$ in the middle” of $M_i\ev_{A,B} \colon M_i(\Pi[A,B].A) \to M_i(A.B)$:
  \[ 
  \phantom{}
  \begin{tikzpicture}[cd-style,x={(2cm,0cm)},y={(0cm,2cm)},z={(2.5cm,-0.8cm)}]
    
    \node(MiPiA) at (0,0) {$M_i(\Pi[A,B].A)$};
    \node(MiA) at (3,0) {$M_iA$};
    \arr (MiPiA) to (MiA);
    \node(MiPiAi) at (0,1) {$M_i\Pi[A,B]. A_{\yon{i}}$};
    \arr[fib] (MiPiAi) to (MiPiA);
    \node(GAi) at (3,1) {$(\Gamma.A)_i \mathrlap{{} = M_iA.A_i}$};
    \arr[fib] (GAi) to (MiA);
    \draw let
      \p1 = (MiPiAi), 
      \p2 = (MiPiA), 
      \p3 = (GAi), 
      \p4 = (MiA), 
      \p5 = ($(\p1)!0.25!(\p2)$),
      \p6 = ($(\p1)!0.1!(\p3)$),
      \p7 = ($(\p5)+(\p6)-(\p1)$)
      in
      ($(\p7)!1ex!(\p5)$) -- (\p7) -- ($(\p7)!1ex!(\p6)$);

    \node (MiAB) at (3,0,-1) {$M_i(A.B)$};
    \arr[fibb] (MiAB) to (MiA);
    \arr (MiPiA) to node {$M_i\ev_{A,B}$} (MiAB);


    \node (MiB) at (3,1,-1) {$M_iB$};
    \arr[fibb] (MiB) to (GAi);
    \arr (MiB) to (MiAB);
    \arr (MiPiAi) to node {$f$} (MiB);
    \draw let
      \p1 = (MiB), 
      \p2 = (MiAB), 
      \p3 = (GAi), 
      \p4 = (MiA), 
      \p5 = ($(\p1)!0.2!(\p2)$),
      \p6 = ($(\p1)!0.2!(\p3)$),
      \p7 = ($(\p5)+(\p6)-(\p1)$)
      in
      ($(\p7)!1ex!(\p5)$) -- (\p7) -- ($(\p7)!1ex!(\p6)$);
    
    \arr[crossing over] (MiPiAi) to (GAi);

  \end{tikzpicture}
  \phantom{{} = M_iA.A_i} \]
  Pulling $B_i$ back along $f$ now yields the iterated context extension
 \[ M_i\Pi[A,B] . A_{\yon{i}} . f^* B_i 
 \fibto M_i\Pi[A,B] . A_{\yon{i}}
 \fibbto M_i \Pi[A,B]. \]
  whose iterated $\Pi$-type, $\Pi[A_{\yon{i}}, f^* B_i] \in \Ty (M_i\Pi[A,B])$, we take as $\Pi[A,B]_i$.
  The evaluation map of this iterated $\Pi$-type is then a map 
  \[ \ev_{A_{\yon{i},f^*B_i}} \colon M_i\Pi[A,B] . \Pi[A_{\yon{i}}, f^* B_i] . A_{\yon{i}} \to M_i\Pi[A,B].A_{\yon{i}}.f^*B_i ; \]
  but $M_i\Pi[A,B] . \Pi[A_{\yon{i}}, f^* B_i] . A_{\yon{i}} = (\Gamma.\Pi[A,B].A)_i$, so the composite \arxivorsubmission{map}{morphism} $(f.B_i)\ev_{A_{\yon{i},f^*B_i}}$ is as required for $(\ev_{A,B})_i$.
  Summarising all this diagrammatically, we have:
  \[ \begin{maxwidth}
  \phantom{}
  \begin{tikzpicture}[cd-style,x={(2cm,0cm)},y={(0cm,2cm)},z={(2.5cm,-0.8cm)}]
    
    \node(MiPi) at (0,0) {$M_i\Pi[A,B]$};
    \node(Gi) at (3,0) {$\Gamma_i$};
    \arr[fibb] (MiPi) to (Gi);
    \node(MiPiA) at (0,1) {$M_i(\Pi[A,B].A)$};
    \arr[fibb] (MiPiA) to (MiPi);
    \node(MiA) at (3,1) {$M_iA \mathrlap{{} = \Gamma_i.A_{\del i}}$};
    \arr[fibb] (MiA) to (Gi);
    \arr (MiPiA) to (MiA);
    \node(MiPiAi) at (0,2) {$M_i\Pi[A,B]. A_{\yon{i}}$};
    \arr[fib] (MiPiAi) to (MiPiA);
    \node(GAi) at (3,2) {$(\Gamma.A)_i \mathrlap{{} = M_iA.A_i}$};
    \arr[fib] (GAi) to (MiA);
    \draw let
      \p1 = (MiPiA), \p2 = (MiPi), \p3 = (MiA), \p4 = (Gi), 
      \p5 = ($(\p1)!3ex!(\p2)$), \p6 = ($(\p1)!3.5ex!(\p3)$), \p7 = ($(\p5)+(\p6)-(\p1)$)
      in ($(\p7)!1ex!(\p5)$) -- (\p7) -- ($(\p7)!1ex!(\p6)$);
    \draw let
      \p1 = (MiPiAi), \p2 = (MiPiA), \p3 = (GAi), \p4 = (MiA), 
      \p5 = ($(\p1)!3ex!(\p2)$), \p6 = ($(\p1)!3.5ex!(\p3)$), \p7 = ($(\p5)+(\p6)-(\p1)$)
      in ($(\p7)!1ex!(\p5)$) -- (\p7) -- ($(\p7)!1ex!(\p6)$);

    \node (MiAB) at (3,1,0.-1) {$M_i(A.B)$};
    \arr[fibb] (MiAB) to (MiA);
    \arr (MiPiA) to node {$M_i \ev_{A,B}$} (MiAB);


    \node (MiB) at (3,2,-1) {$M_iB$};
    \arr[fibb] (MiB) to (GAi);
    \arr (MiB) to (MiAB);
    \arr (MiPiAi) to node {$f$} (MiB);
    \draw let
      \p1 = (MiB), \p2 = (MiAB), \p3 = (GAi), \p4 = (MiA), 
      \p5 = ($(\p1)!2.5ex!(\p2)$), \p6 = ($(\p1)!3ex!(\p3)$), \p7 = ($(\p5)+(\p6)-(\p1)$)
      in ($(\p7)!1ex!(\p5)$) -- (\p7) -- ($(\p7)!1ex!(\p6)$);
    
    \arr[crossing over] (MiPiAi) to (GAi);

    \node(MiPiAiB) at (0,3) {$M_i\Pi[A,B]. A_{\yon{i}}.f^*B_i$};
    \arr[fib] (MiPiAiB) to (MiPiAi);
    \node(GABi) at (3,3,-1) {$(\Gamma.A.B)_i \mathrlap{{} = M_i(B).B_i}$};
    \arr[fib] (GABi) to (MiB);
    \arr (MiPiAiB) to node[auto=right] {$f.B_i$} (GABi);
    \draw let
      \p1 = (MiPiAiB), \p2 = (MiPiAi), \p3 = (GABi), \p4 = (MiB), 
      \p5 = ($(\p1)!3.7ex!(\p2)$), \p6 = ($(\p1)!3.6ex!(\p3)$), \p7 = ($(\p5)+(\p6)-(\p1)$)
      in ($(\p7)!1ex!(\p5)$) -- (\p7) -- ($(\p7)!1ex!(\p6)$);

    \node (GPii) at (-3.75,1,1.25) {$ M_i\Pi[A,B] . \Pi[A_{\yon{i}}, f^* B_i]$};
    \arr[fib] (GPii) to (MiPi);

    \node (GPiAi) at (-3.75,3,1.25) {$ M_i\Pi[A,B] . \Pi[A_{\yon{i}}, f^* B_i] . A_{\yon{i}}$};
    \arr[fibb] (GPiAi) to (GPii);
    \arr (GPiAi) to (MiPiAi);
    \arr (GPiAi) to node[auto=right] {$\ev_{A_{\yon{i},f^*B_i}}$} (MiPiAiB);
    \arr[bend left=20] (GPiAi) to node[anchor=center,fill=white] {$(\ev_{A,B})_i$} (GABi);
    \draw let
      \p1 = (GPiAi), \p2 = (GPii), \p3 = (MiPiAi), \p4 = (MiPiA), 
      \p5 = ($(\p1)!2.5ex!(\p2)$), \p6 = ($(\p1)!3ex!(\p3)$), \p7 = ($(\p5)+(\p6)-(\p1)$)
      in ($(\p7)!1ex!(\p5)$) -- (\p7) -- ($(\p7)!1ex!(\p6)$);
  \end{tikzpicture}
  $\phantom{{} = M_iA.A_i}$ \end{maxwidth} \]
  Here we see also that $(\ev_{A,B})_i$ lies over $(\Gamma.A)_i$, as required.

  This completes the construction of $\Pi[A,B]$ and $\ev_{A,B}$.
  It remains to construct $\lambda$.
  So, fix some $\Gamma$, $A$, $B$ as above, along with a section $b \colon \Gamma.A \to \Gamma.A.B$.
  We need to construct a section $\lambda b \colon \Gamma \to \Gamma.\Pi[A,B]$, such that $\ev_{A,B}(\lambda b.A) = b$.

As usual, we work by induction: we assume that $\lambda b$ is defined on $\I_{<i}$, satisfying the desired equation, and wish to construct $(\lambda b)_i$.
In particular, it suffices to construct a section of $(M_i (\lambda b))^* \Pi[A,B]_i$, where $M_i (\lambda b) \colon \Gamma \to M_i (\Pi[A,B])$.

But by the construction of $\Pi[A,B]_i$ as an iterated $\Pi$-type (and stability of that under pullback), it suffices to construct a section $s$ as in the following diagram; $\lambda s$ then gives the desired section of $(M_i (\lambda b))^* \Pi[A,B]_i$:

\vspace{-\baselineskip}  
\[
  \begin{maxwidthtikzpicture}[cd-style,x={(2cm,0cm)},y={(0cm,2cm)},z={(2.5cm,-0.8cm)}]
    
    \node(MiPi) at (0,0) {$M_i\Pi[A,B]$};
    \node(Gi) at (3,0) {$\Gamma_i$};
    \arr[fib] (MiPi) to (Gi);
    \node(MiPiA) at (0,1) {$M_i(\Pi[A,B].A)$};
    \arr[fibb] (MiPiA) to (MiPi);
    \node(MiA) at (3,1) {$M_iA$};
    \arr[fibb] (MiA) to (Gi);
    \arr (MiPiA) to (MiA);
    \node(MiPiAi) at (0,2) {$M_i\Pi[A,B]. A_{\yon{i}}$};
    \node(GAi) at (3,2) {$(\Gamma.A)_i$};
    \arr[fib] (GAi) to (MiA);

    \node (MiAB) at (3,1,-1) {$M_i(A.B)$};
    \arr[fibb] (MiAB) to (MiA);
    \arr (MiPiA) to node {$M_i \ev_{A,B}$} (MiAB);


    \node (MiB) at (3,2,-1) {$M_iB$};
    \arr[fibb] (MiB) to (GAi);
    \arr (MiB) to (MiAB);
    \arr (MiPiAi) to node {$f$} (MiB);
    
    \arr[crossing over] (MiPiAi) to (GAi);

    \node(MiPiAiB) at (0,3) {$M_i\Pi[A,B]. A_{\yon{i}}.f^*B_i$};
    \node(GABi) at (3,3,-1) {$(\Gamma.A.B)_i$};
    \arr[fib] (GABi) to (MiB);
    \arr (MiPiAiB) to (GABi);

    \node(GiL) at (-3,0) {$\Gamma_i$};
    \arr (GiL) to node {$M_i(\lambda b)$} (MiPi);
    \arr[bend right = 5] (GiL.-30) to node[auto=right] {$\id$} (Gi.210);
    \node(MiAL) at (-3,1) {$M_iA$};
    \arr[fibb] (MiAL) to (GiL);
    \arr (MiAL) to node {$M_i((\lambda b).A)$} (MiPiA);
    \node(GAiL) at (-3,2) {$(\Gamma.A)_i$};
    \arr[fib] (GAiL) to (MiAL);
    \arr (GAiL) to node {$M_i(\lambda b).A_{\yon{i}}$} (MiPiAi);
    \node(GAiBL) at (-3,3) {$(\Gamma.A)_i . (M_i(\lambda b).A_{\yon{i}})^* f^* B_i$};
    \arr[fib] (GAiBL) to (GAiL);
    \arr (GAiBL) to (MiPiAiB);

    \arr[dashed, bend left] (GAiL) to node {$s$} (GAiBL);
    \arr[dashed, bend left, out = 5, in = 210] (GAiL) to node {$b_i$} (GABi);
    \arr[bend left = 10] (MiAL) to node {$M_ib$} (MiAB);

    \arr[fib, crossing over] (MiPiAi) to (MiPiA);
    \arr[fib, crossing over] (MiPiAiB) to (MiPiAi);
  \end{maxwidthtikzpicture} \]

  Such a section $s$ corresponds by pullback to a map $(\Gamma.A)_i \to (\Gamma.A.B)_i$ over $M_i(B)$.
  The evident candidate is $b_i$; so we just need to show that $b_i$ commutes over $M_i B$.
  For this, it suffices to check commutativity into $(\Gamma.A)_i$ and $M_i(A.B)$.
  But the former of these is clear, and the latter follows from the equation $(M_i\ev_{A,B})(M_i (\lambda b)) = M_i b$, which we know by induction.

  This gives the desired $(\lambda b)_i$; its computation rule $(\lambda b . A)_i (\ev_{A,B})_i = b_i$ then follows directly from the computation rule of $\Pi[A,B]_i$ as an iterated $\Pi$-type.

  All these constructions are stable under reindexing, and so assemble into a $\Pi$-structure on $\C^\I$; and it is routine to check that this satisfies the $\Pi$-$\eta$-rule if the original one on $\C$ did, and is functorial in $\C$ and $\I$.
\end{proof}

In fact, this construction provides a little more than just $\Pi$-types of Reedy types.

\begin{proposition} \label{prop:levelwise-to-reedy-pi-types}
  Suppose $\C$ is equipped with $\Pi$-type structure.
  Then for any $\Gamma \in \C^\I$, any \emph{levelwise extension} $\AA$ of $\Gamma$, and any Reedy type $B$ over $\Gamma.\AA$, there is a Reedy type $\Pi[\AA, B]$ over $\Gamma$, together with a map $\ev_{\AA,B} \colon \Gamma.\Pi[\AA, B].\AA \to \Gamma.\AA.B$ over $\AA$ and operation $\lambda_{\AA,B}$ as in the definition of $\Pi$-type structure.
  
    Moreover, this construction is stable in the base $\Gamma$, and functorial in $\C$ and $\I$.
\end{proposition}

\begin{proof}
  The construction is exactly as for Proposition~\ref{prop:reedy-pi-types}.
  The only uses made of $A$ being a Reedy type were the invocations of matching objects $M_i A$, $M_i (A.B)$ in, for instance, the definition of the map $f \colon M_i \Pi[A,B].A_{\yon i} \to M_i B$ as a pullback of $M_i \ev_{A,B} \colon M_i(\Pi[A,B].A) \to M_i(A.B)$.
  For the general case of a levelwise extension $\AA$, we cannot use $M_i \ev_{A,B}$ since the matching objects it goes between do not exist.
  However, the required map $f \colon M_i \Pi[A,B].\AA_i \allowbreak \to M_i B$ can still be defined straightforwardly (albeit less concisely), by explicitly specifying its components according to the universal property of $M_i B$ as a limit; and similarly for the other steps that make use of Reediness of $A$.
\end{proof}

\begin{proposition}
  Suppose the $\Pi$-type structure of $\C$ is extensional, i.e.\ extends to $\Piext$-structure in the sense of Definition \ref{def:funext-structure}.
  Then so is the induced $\Pi$-type structure on $\C^\I$, functorially in $\C$ and $\I$.
\end{proposition}

\begin{proof}
  Recall from Lemma~\ref{lem:funext-iff-r-equiv} that a $\Pi$-type structure is extensional just if there is an operation giving, for each suitable $\Gamma$, $A$, $B$, structured equivalence data over $\Gamma$ on the map
  \[ R \coloneqq \lambda_{A,B.B.\Id_B}(\refl_B \ev_{A,B}) \colon \Gamma . \Pi[A,B] \to \Gamma . \Pi[A,B.B.\Id_B], \] 
  stably in $\Gamma$.
  So, assume $\C$ is equipped with such an operation; we will then construct it on $\C^\I$.

  Given $\Gamma$, $A$, $B$ in $\C^\I$, we want equivalence data in $\C^\I$ on the map $R$ defined above; it suffices by Lemmas~\ref{lem:levelwise-equivalence-iff-reedy} and \ref{lem:equivalence-iff-levelwise} to construct equivalence data in $\C$ on the comparison map $m_i R$:
  \[\begin{maxwidthtikzcd}
  (\Gamma . \Pi[A,B])_i \ar[r,"m_i R"] \ar[rd, fib] & M_i \Pi[A,B] . (M_i R)^* \Pi[A,B.B.\Id_B]_i  \ar[r] \ar[d,fib] \ar[drpb] & (\Gamma . \Pi[A,B.B.\Id_B])_i \ar[d,fib]   \\
   &  M_i \Pi[A,B] \ar[r, "M_i R"] & M_i \Pi[A,B.B.\Id_B].
  \end{maxwidthtikzcd}\]  
 
  Inspecting the construction of the $\Pi$-types involved, we see that we have an equality of context extensions
$\Pi[A,B.B.\Id_B]_i = \Pi[A_{\yon{i}},g^*(B.B.\Id_B)_i]$
over $M_i \Pi[A,B.B.\Id_B]$, where $g \colon M_i \Pi[A,B.B.\Id_B] . A_{\yon{i}} \to M_i (B.B.\Id_B)$ is analogous to the map $f$ in the construction of $\Pi$-types.
  From this,
$ (M_i R)^* \Pi[A,B.B.\Id_B]_i = \Pi[A_{\yon{i}},(M_i R. A_{\yon{i}})^*g^*(B.B.\Id_B)_i]$ as context extensions over $M_i \Pi[A,B]$,
since by weakening $(M_i R)^* A_{\yon{i}} = A_{\yon{i}}$.
  But by the definition of $R$ as $\lambda (\refl_B \ev_{A,B})$, for $\refl_B \colon \Gamma.A.B \to \Gamma.A.B.B.\Id_B$, and by the construction of $g$ using $\ev_{A,B.B.\Id_B}$, we find that $g (M_i R . A_{\yon{i}}) = (\refl_B)_i f$, where $f \colon M_i \Pi[A,B] . A_{\yon{i}} \to M_i B$ is as in the construction of $\Pi$-types.
  So
  $ R_i^* \Pi[A,B.B.\Id_B]_i = \Pi[A_{\yon{i}},f^*\refl_B^*(B.B.\Id_B)_i]$
  as context extensions over $M_i \Pi[A,B]$,
  and the comparison map $m_i R$ corresponds to the map $M_i \Pi[A,B] . \Pi[A_{\yon{i}},f^*B_i]$ $\to M_i \Pi[A,B] . \Pi[A_{\yon{i}},f^*\refl_B^*(B.B.\Id_B)_i]$ induced by composition with $\refl_B$.

  But $\refl_B$ canonically carries equivalence data, and by extensionality of the $\Pi$-type structure in $\C$, this gives equivalence data on the induced map of $\Pi$-types.
  So each $m_i R$ carries equivalence data, and every step was stable in the base $\Gamma$ and functorial in $\C$ and $\I$; so we are done.
\end{proof}

Summarising the results of this section, we have:

\begin{proposition} \label{prop:logical-structure-summary}
  Let $\C$ be a CwA, and $\I$ an ordered inverse category.
  If $\C$ carries $\Id$-, $\Sigma$-,  unit, $\Pi$-, $\Pi_\eta$, or $\Pi_\ext$-types, then so does $\C^\I$.
  Moreover, all this logical structure is preserved by the functorial action under ordered discrete opfibrations in $\I$, and under CwA maps preserving the given logical structure on $\C$. \qed
\end{proposition}

\begin{example}
  Note that for functoriality in $\I$, the restriction to \emph{ordered} discrete opfibrations is really unavoidable here.
  A non-ordered discrete fibration $u \colon \J \to \I$ will still induce a \emph{pseudo}-map of CwA’s $\C^\I \to \C^\J$: acting (strictly) naturally on types, and preserving context extension up to coherent isomorphism.
  However, that pseudo-map will not typically preserve the logical structure constructed in this section.

  For instance, consider the “symmetry” automorphism $\sigma \colon \SpanCat \to \SpanCat$, switching $0$ and $1$.
  For any CwA $\C$, this induces a symmetry pseudo-map $\sigma^* \colon \Span[\C] \to \Span[\C]$; but it is straightforward to check that this will not strictly preserve $\Pi$-types in general.
\end{example}


\section{Homotopical diagrams} \label{sec:homotopical-diagrams}

\begin{definition}[cf.~{\cite[\S 1.5.1]{dhks}}]
  A \defemph{homotopical category} is a category $\C$ together with a distinguished subclass $\W$ of morphisms, containing all identities and closed under 2 out of 6.
  The morphisms in $\W$ will be called \defemph{equivalences}.

  A functor $F \colon (\C,\W) \to (\C',\W')$ is called \defemph{homotopical} if it preserves equivalences.

  Given a homotopical category $\C = (\C,\W)$, write $\disc{\C}$ for $\C$ considered as an ordinary category (i.e.\ with the equivalences forgotten).
\end{definition}

In any CwA with identity types, we have the class of equivalences given by Definition~\ref{def:equivalence-in-cwa}.
Given a homotopical inverse category $\I = (\I,\W)$, we can therefore consider \emph{homotopical} diagrams on $\I$ in $\C$.

The main goal of this section is to show that these form a sub-CwA $\C^\I$ of the ordinary diagram CwA $\C^{\disc{\I}}$ already constructed, and moreover that under reasonable assumptions, $\C^\I$ is closed under much of the logical structure on $\C^{\disc{\I}}$.

\subsection{Homotopical diagrams} \label{sec:cwa-of-homotopical-diagrams}

For this section, fix a CwA $\C$ with $\Id$-types.

\begin{definition}
  Let $\I$ be an ordered homotopical inverse category.
  
  A levelwise extension $\AA$ of a diagram $\Gamma$ in $\C^{\disc{\I}}$ is \defemph{homotopical} if for each equivalence $\alpha \colon i \to j$ in $\I$, the comparison map $\Gamma_i.\AA_i \to \Gamma_i.\Gamma_\alpha^* \AA_j$ is an equivalence.
  A (Reedy) type in $\C^{\disc{\I}}$ is \defemph{homotopical} if it is so when viewed as a levelwise extension.
\end{definition}

\begin{proposition}
  For any ordered homotopical inverse category $\I$, the homotopical types and levelwise extensions in $\C^{\disc{\I}}$ are closed under reindexing. 
\end{proposition}

\begin{proof}
  Immediate by the fact that equivalences between context extensions are closed under re-indexing in the base.
\end{proof}

\begin{proposition} \label{prop:homotopical-types}
  A levelwise extension $\AA$ of a homotopical diagram $\Gamma$ is homotopical if and only if its comprehension is a homotopical diagram.
\end{proposition}

\begin{proof}
  Immediate by right properness and 2 out of 3: if $\Gamma_\alpha$ is an equivalence, then the comparison map $\Gamma_i.\AA_i \to \Gamma_i.\Gamma_\alpha^* \AA_j$ is an equivalence if and only if $(\Gamma.\AA)_\alpha$ is.
  \[ \begin{verticalhack} \begin{tikzcd}
    \Gamma_i.\AA_i \ar[rr,bend left=20,"(\Gamma.\AA)_\alpha"] \ar[r] \ar[dr,fibb] & \Gamma_i.\Gamma_\alpha^*\AA_j \ar[r,weq,"\Gamma_\alpha.\AA_j"'] \ar[d,fibb] \ar[drpb] & \Gamma_j.\AA_j \ar[d,fibb] \\
    & \Gamma_i \ar[r,weq',"\Gamma_\alpha"] & \Gamma_j
  \end{tikzcd} \end{verticalhack} \qedhere \]
\end{proof}

\begin{definition} \label{def:cwa-of-homotopical-diagrams}
  Take $\C^\I$ (resp.\ $\C^\I_{\lw}$) to be the full sub-CwA of $\C^{\disc{\I}}$ (resp.\ $\C^{\disc{\I}}_{\lw}$) on homotopical diagrams and homotopical Reedy types (resp.\ homotopical levelwise extensions).
  The preceding propositions ensure that this indeed forms a sub-CwA.
\end{definition}

In the remainder of this paper, we will consider homotopical types only over homotopical diagrams; we will therefore make use of Proposition~\ref{prop:homotopical-types} without comment.

It is easy to see that this construction is functorial in suitable maps:
 
\begin{proposition} \leavevmode
  \begin{enumerate}
    \item For any homotopical functor $u \colon \J \to \I$ between homotopical categories, and CwA $\C$, the induced map $u^* \colon \C^{\disc{\I}}_{\lw} \to \C^{\disc{\J}}_{\lw}$ restricts to $u^* \colon \C^\I_{\lw} \to \C^\J_{\lw}$.
    \item For any ordered homotopical discrete opfibration $u \colon \J \to \I$ between ordered homotopical inverse categories, and CwA $\C$, the induced map $u^* \colon \C^{\disc{\I}} \to \C^{\disc{\J}}$ restricts to $u^* \colon \C^\I \to \C^\J$.
    \item For any ordered homotopical inverse category $\I$ and \emph{homotopical} map of CwA's $F \colon \C \to \D$, the induced map $F^\I \colon \C^{\disc{\I}} \to \D^{\disc{\I}}$ restricts to $F^\I \colon \C^\I \to \D^\I$. \qed
  \end{enumerate}
\end{proposition}

\optionaltodo{Reword the following to explicitly give proposition that map between cxl cats with Id-types is homotopical.  Perhaps allso add results later that the induced maps between diagram cats are all homotopical.}
\begin{example}
  Note the extra condition in the last item: a CwA map acts on homotopical diagrams only if it is itself homotopical, i.e.~preserves equivalences.
  A map of CwAs with identity types will always preserve equivalences within fibrant slices, since these are defined in terms of identity types.
  However, the definition of general equivalences in CwA’s is not algebraic, so their preservation is not automatic.

  For instance, every category $\C$ admits a CwA structure in which $\Ty$ is the terminal presheaf $1$, and context extension acts trivially: every dependent projection is an identity map.
  Call this CwA $(\C,1)$; it admits (a unique choice of) identity, unit, $\Sigma$-, and extensional $\Pi$-type structure.
  Each fibrant slice of $(\C,1)$ is the terminal CwA, so every map in it is an equivalence.
  Meanwhile, with a little care, one can put a CwA structure on $\Set$, equivalent to the usual one, except that over each context $X$, there is only one “singleton family” (i.e.\ only one $A \in \Ty(X)$ such that each $A_x$ is a singleton), and the context extension by this family is $\id_X$.
  Call this $\Set'$; it carries all the same structure mentioned above, and its equivalences will be just the isomorphisms.
  
  Now any functor $F \colon \C \to \Set$ induces a CwA map $(\C,1) \to \Set'$ preserving the aforementioned structure; but if $F$ hits any non-isomorphism, this will not preserve equivalences.
\end{example}

\begin{proposition} \label{prop:diag-equiv-to-homot-diag}
  Any diagram levelwise equivalent to a homotopical diagram is homotopical.
  Moreover, if $w \colon \Gamma.\AA \to \Gamma.\BB$ is an equivalence of extensions over a homotopical base, and either of $\AA$, $\BB$ is homotopical, then so is the other.  
\end{proposition}

\begin{proof}
  The first statement is immediate by 2 out of 3.
  The second follows since equivalences in fibrant slices are levelwise, by Lemma~\ref{lem:equivalences-are-levelwise}.
\end{proof}

\subsection{Logical structure} \label{sec:homotopical-logical-structure}

We give logical structure on $\C^\I$ by showing that, under suitable hypotheses, it is closed under the relevant operations in $\C^{\disc{\I}}$.
It is enough to show closure under the type-forming operations; for term-forming operations, closure is automatic since $\C^\I$ is a full subcategory.

For the covariant type-formers we consider, no extra hypotheses are needed:
\begin{proposition} \label{prop:homotopical-id-etc}
  For any ordered homotopical inverse category $\I$, $\C^\I$ is closed under the $\Id$-type, $\Sigma$-type, and unit type operations on $\C^{\disc{\I}}$ constructed in Section~\ref{sec:logical-structure}.
\end{proposition}

\begin{proof}
  We claim that in each case, the comprehension of the type produced is levelwise equivalent to a diagram which we already know is homotopical, so is homotopical by Proposition~\ref{prop:diag-equiv-to-homot-diag}.

  The levelwise equivalences are given by the constructor maps:
  \begin{itemize}
  \item $\refl_A \colon \Gamma .A \to \Gamma . A . A . \Id_A$;
  \item $\pair \colon \Gamma . A . B \to \Gamma . \Sigma_A B$;
  \item $\star_{\Gamma} \colon \Gamma \to \Gamma . 1_\Gamma$.
  \end{itemize}
  These maps are certainly equivalences in $\C^{\disc{\I}}$, since they are always equivalences in any CwA with the logical structure in question.
  But they also lie within fibrant slices of $\C^{\disc{\I}}$, and so by Lemma~\ref{lem:equivalences-are-levelwise} are levelwise.
\end{proof}

The question of closure under $\Pi$-types is less straightforward.
To obtain this closure, we will assume two extra hypotheses: functional extensionality in $\C$, and the condition that all maps in $\I$ are equivalences.

The extensionality hypothesis is unsurprising, since it is needed to know that $\Pi$-types respect equivalences in $\C$.
On the other hand, the hypothesis that all maps in $\I$ are equivalences seems quite possibly stronger than necessary.
However, \emph{some} restriction on $\I$ is certainly needed for the closure of $\C^\I$ under $\Pi$-types:

\begin{example} \label{ex:pi-types-in-weak-maps}
  Let $\WkMapCat$ be the homotopical Reedy category representing weak maps, $(0 \weqfrom 01 \to 1)$.
  Consider $\Set$ as a (large) CwA with $\Ty(X) := \Set^X$, in the evident way (or, if concerned about size issues, use $\FinSet$ instead).
  Then $\Set^\WkMapCat$ is not closed under $\Pi$-types in $\Set^{\SpanCat}$.
  
  To see this, note that equivalences in $\Set$ are exactly isomorphisms, and so $\Set^\WkMapCat \equiv \Set^{\rightarrow}$.
  Failure of closure then amounts roughly to the fact that $\Set^{\rightarrow} \to \Set^\SpanCat$ does not preserve exponentials, which may be seen from the standard calculation of exponentials in presheaf categories. 
\end{example}

The proof that $\Pi$-types preserve homotopicality under these assumptions requires a couple of preparatory observations on the levelwise-to-Reedy $\Pi$-types given in Proposition~\ref{prop:levelwise-to-reedy-pi-types}.
Fix for the next couple of lemmas a CwA with $\Id$- and $\Pi_\ext$-structure $\C$, and an ordered inverse category $\I$.

\newcommand{\constdiag}[2]{C_{#1}{#2}}
\begin{lemma} \label{lem:levelwise-pi-as-iterated}
  Let $\AA$ be a context extension over $\Gamma$ in $\C$; write $\constdiag{\I}{\Gamma}$ and $\constdiag{\I}{\AA}$ for the constant diagram and levelwise extension on these in $\C^\I$.
  Let $B$ be a Reedy type over $\constdiag{\I}{\Gamma}.\constdiag{\I}{\AA}$ in $\C^\I$.
  Then $\Pi[\constdiag{\I}{\AA},B]_i$ is exactly the iterated $\Pi$-type $\Pi[\AA,B_{\yon{i}}]$ over $\Gamma$ in $\C$.
\end{lemma}

\begin{proof}
  By induction on $i \in \I$, unwinding the construction of Proposition~\ref{prop:levelwise-to-reedy-pi-types}.
\end{proof}

\begin{lemma} \label{lem:levelwise-pi-contravariant}
  Let $\AA$, $\AA'$ be levelwise extensions over $\Gamma \in \C^\I$, and $B$ a Reedy type over $\Gamma.\AA$, and suppose $w \colon \Gamma.\AA' \to \Gamma.\AA$ is a levelwise equivalence over $\Gamma$.
  Then the map $\Gamma.\Pi[\AA,B] \to \Gamma.\Pi[\AA',w^*B]$ representing precomposition with $w$ is again a levelwise equivalence.
\end{lemma}

\begin{proof}
  Again by induction on $i \in \I$, unwinding the constructions of Proposition~\ref{prop:levelwise-to-reedy-pi-types}.
\end{proof}

\begin{proposition} \label{prop:pi-types-homotopical}
  Let $\I$ be an ordered homotopical inverse category with all maps equivalences, and $\C$ a CwA with $\Id$- and $\Pi_\ext$-structure.
  Then $\C^\I$ is closed under $\Pi$-types in $\C^{\disc{\I}}$, and hence carries $\Pi_\ext$-structure.
\end{proposition}

\begin{proof}
  Take homotopical diagrams and types $\Gamma$, $A$, and $B$ as in the input of $\Pi$-formation; we need to show that $\Pi[A,B]$ is homotopical.

  Roughly, our strategy will be to show that under the given assumption on $\I$, each component $\Pi[A,B]_i$ of the dependent product is “naturally” equivalent to the corresponding \emph{levelwise} dependent product $\Pi[A_{\yon i},B_{\yon i}]$, and then to note that this latter construction “respects equivalences”.
  This is complicated firstly by the fact that the levelwise dependent product does not itself form a diagram, since it is not covariantly functorial in the domain, so the desired “natural” equivalence can at best be dinatural; and secondly, by the change of base between different values of $\Gamma$.

  Concretely, given $\alpha \colon i \to j$ in $\I$, we will construct a diagram in $\C$
  \[ \begin{tikzcd}
      (\Gamma.\Pi[A,B])_i \ar[rr,weq,"w_5"'] \ar[dd] \ar[ddd,shift right={2em},bend right={40},"{(\Gamma.\Pi[A,B])_\alpha}"']
    & & \Gamma_i.\Pi[A_{\yon i}, B_{\yon i}] \ar[dr,weq',"w_4"]
    &[-6em]
    \\[-1em] & & & \hspace{-2em} \Gamma_i.\Pi[A_{\yon i}, (\Gamma.A_\alpha)^* B_{\yon j}]
    \\[-1em] \Gamma_i. \Gamma_\alpha^* \Pi[A,B]_{\yon j} \ar[rr,weq,"w_2"'] \ar[d,weq']
    & &  \Gamma_i.\Pi[\Gamma_\alpha^*A_{\yon j}, \Gamma_\alpha^* B_{\yon j}]  \ar[ur,weq,"w_3"'] \ar[d,weq']
    \\ (\Gamma.\Pi[A,B])_j \ar[rr,weq,"w_1"']
    & & \Gamma_j.\Pi[A_{\yon j}, B_{\yon j}]
  \end{tikzcd} \]
  commuting up to homotopy, and show that the marked maps are equivalences.
  With this done, it follows by 2 out of 3 that $(\Gamma.\Pi[A,B])_\alpha$ is an equivalence as required.

  The lower two vertical maps are pullbacks of the equivalence $\Gamma_\alpha \colon \Gamma_i \to \Gamma_j$ along the context extensions $\Pi[A,B]_{\yon j}$ and $\Pi[A_{\yon j}, B_{\yon j}]$, so are equivalences by right properness.
  The upper-left vertical is then the factorisation of $\Gamma_\alpha$ through the pullback $\Gamma_i. \Gamma_\alpha^* \Pi[A,B]_{\yon j}$.
 
  We will construct the remaining equivalences $w_1$--$w_5$ one by one.
  For $w_1$ and $w_2$, first consider the following diagram in $\C^{j/\I}$:
  \[ \begin{tikzcd}
      \bullet \ar[d,fib] \ar[r] \ar[drpb]
      & \bullet \ar[d,fib] \ar[r] \ar[drpb]
      & \Gamma.A.B\restr{j/\I} \ar[d,fib]
      \\
      \constdiag{j/\I}{(\Gamma_j.A_{\yon j})} \ar[r,dashed,"f"'] \ar[d,lw]
              \ar[rr,bend left=15,"\nu_{\Gamma.A}" pos={0.35},"\sim" pos={0.65}]
      & \bullet \ar[d,fib] \ar[r] \ar[drpb]
      & \Gamma.A \restr{j/\I} \ar[d,fib]
      \\
     \constdiag{j/\I}{\Gamma_j} \ar[r,equal]
      & \constdiag{j/\I}{\Gamma_j} \ar[r,"\nu_\Gamma"',weq]
      & \Gamma \restr{j/\I}
    \end{tikzcd} \]
  Here $(-)\restr{j/\I}$ denotes restriction of diagrams, types, etc.\ along the forgetful functor $j/\I \to \I$, and $C_{j/\I}$ denotes constant diagrams.
  The map $\nu_\Gamma \colon C_{j/\I} \Gamma_j \to \Gamma\restr{j/\I}$ has components given by the action on maps of $\Gamma$, and is a levelwise equivalence since all maps in $\I$ are equivalences; and $\nu_{\Gamma.A}$ similarly.
  Other objects and maps are induced by pullback as indicated; all horizontal maps are levelwise equivalences, by right properness and 2 out of 3.

  Taking $\Pi$-types of each column (in the Reedy diagram CwA structure on $\C^{j/\I}$) and evaluating at $\id_j \in j/\I$ yields a diagram in $\C$:
  \[ \begin{tikzcd}
      \Gamma_j.\Pi[A_{\yon j}, B_{\yon j}] \ar[d,fibb]
      & (\Gamma.\Pi[A,B])_j \ar[d,fibb] \ar[r,"\id"] \ar[drpb] \ar[l,dashed,"w_1"',weq]
      & (\Gamma.\Pi[A,B])_j \ar[d,fibb]
      \\
      \Gamma_j \ar[r,equal]
      & \Gamma_j \ar[r,"\id"]
      & \Gamma_j
    \end{tikzcd} \]
  Here the $\Pi$-type of the third column gives precisely $(\Gamma.\Pi[A,B])_j$, since restriction $\C^\I \to \C^{j/\I}$ preserves $\Pi$-types; the $\Pi$-type of the second column is a (trivial) reindexing of this; and the $\Pi$-type of the first column gives the iterated $\Pi$-type shown, by Lemma~\ref{lem:levelwise-pi-as-iterated}.
  Finally, $w_1$ is induced by the contravariant action of $\Pi$ (Lemma~\ref{lem:levelwise-pi-contravariant}), applied to the map $f$ above and evaluated at $\id_j$.
  
  Next, $w_2$ is the reindexing of $w_1$ along $\Gamma_\alpha : \Gamma_i \to \Gamma_j$; this is an equivalence by stability in the base (or 2 out of 3), and the square from $w_2$ to $w_1$ commutes.

  For $w_3$, we extend the last two diagrams leftward:
  \[ \begin{tikzcd}
      \bullet \ar[d,fib] \ar[r] \ar[drpb]
      & \bullet \ar[d,fib] \ar[r] \ar[drpb]
      &[-2em] \constdiag{j/\I}{(\Gamma_j.A_{\yon j})}.\nu_{\Gamma.A}^*(B \restr{j/\I}) \ar[d,fib]
      \\ \constdiag{j/\I}{(\Gamma_i. A_{\yon i})} \ar[r,dashed,"g"'] \ar[d,lw]
              \ar[rr,bend left=12,"\constdiag{j/\I} (\Gamma.A_\alpha)" pos={0.35},"\sim" pos={0.65}]
      & \constdiag{j/\I}{(\Gamma_i. \Gamma_\alpha^* A_{\yon j})} \ar[drpb] \ar[r] \ar[d,lw]
      & \constdiag{j/\I}{(\Gamma_j.A_{\yon j})} \ar[d,lw]
      \\
      \constdiag{j/\I}{\Gamma_i} \ar[r,equal]
      & \constdiag{j/\I}{\Gamma_i} \ar[r,"\constdiag{j/\I}{\Gamma_\alpha}"',weq]
      & \constdiag{j/\I}{\Gamma_j}
      \\
      \Gamma_i \Pi[A_{\yon i}, (\Gamma.A_\alpha)^* B_{\yon j}]  \ar[d,fibb]
      & \Gamma_i.\Pi[\Gamma_\alpha^* A_{\yon j}, \Gamma_\alpha^* B_{\yon j}]  \ar[d,fibb]  \ar[r] \ar[drpb] \ar[l,dashed,"w_3"',weq]
      & \Gamma_j.\Pi[A_{\yon j}, B_{\yon j}] \ar[d,fibb]
      \\
      \Gamma_i \ar[r,equal]
      & \Gamma_i \ar[r,"\Gamma_\alpha"',weq]
      & \Gamma_j
    \end{tikzcd} \]
  As before, the first diagram is in $\C^{j/\I}$, with all horizontal arrows levelwise equivalences; the second diagram is the $\Pi$-types of the columns of the first, evaluated at $\id_j$; Lemma~\ref{lem:levelwise-pi-as-iterated} ensures that these $\Pi$-types compute to the objects shown; and $w_3$ is given by the contravariant action of $\Pi$, this time on $g$.
  
  By contrast,  $w_4$ is given by the \emph{covariant} action of iterated $\Pi$-types (Lemma \ref{lem:iterated-pi-covariant}), applied to the map $(\Gamma.A.B)_i \to (\Gamma.A)_i(\Gamma.A)_\alpha^*B_{\yon{j}}$ induced by $(\Gamma.A.B)_\alpha$, which is an equivalence by right properness and 2 out of 3.
  
  The last equivalence $w_5$ is analogous to $w_1$, but with $i$ instead of $j$.

  Finally, the pentagon in the main diagram commutes up to homotopy by extensionality for iterated $\Pi$-types (Lemma~\ref{lem:iterated-pi-extensionality}), along with the defining properties of the maps $w_i$.  
\end{proof}

Putting together Propositions~\ref{prop:homotopical-id-etc}, \ref{prop:pi-types-homotopical},  and~\ref{prop:logical-structure-summary}, we have:

\begin{proposition} \label{prop:homotopical-structure-summary}
  Let $\C$ be a CwA with $\Id$-types, and $\I$ an ordered homotopical inverse category.
  Then the homotopical diagram CwA $\C^\I$ carries $\Id$-types; if additionally $\C$ carries $\Sigma$- or unit types, then so does $\C^\I$; and if additionally all maps in $\I$ are equivalences, then if $\C$ carries extensional $\Pi$-types, so does $\C^\I$.
  
  Moreover, all this logical structure is preserved by the functorial action under ordered homotopical discrete opfibrations in $\I$, and under CwA maps preserving the given logical structure on $\C$. \qed
\end{proposition}


\section{Properties of induced functors} \label{sec:fibrations-and-shit}

We conclude by investigating conditions on functors $u \colon \J \to \I$ under which the induced maps of CwA's $u^* \colon \C^\I \to \C^\J$ enjoy desirable extra properties.
Specifically, we will want to know when this map is a local \emph{fibration}, \emph{trivial fibration}, or \emph{equivalence} in the sense of \jpaacite{kapulkin-lumsdaine:homotopy-theory-of-type-theories}, which we recall here:

\begin{definition}[{\jpaacite[Def.~3.11]{kapulkin-lumsdaine:homotopy-theory-of-type-theories}}] \label{def:fibrations-and-equivs-of-cwas}
  Let $F \colon \C \to \D$ be a map of CwA’s.  Then we say:
  
  \begin{enumerate}
  \item $F$ is a \defemph{local fibration} if it satisfies the properties
    \begin{itemize}
    \item \defemph{equivalence lifting}: given any $\Gamma \in \C$, $A \in \Ty_\C\Gamma$, $B \in \Ty_\D (F\Gamma)$, and structured equivalence $w \colon FA \equiv B$ over $F\Gamma$, there exists some lift $\bar{B} \in \Ty_\C \Gamma$ of $B$, together with a structured equivalence $\bar{w} \colon A \equiv \bar{B}$ over $\Gamma$ such that $F \bar{w} = w$;
    \item \defemph{path lifting}: given any $\Gamma \in \C$, $A \in \Ty_\C\Gamma$, section $a$ of the projection $p_A$ in $\C$, section $a'$ of $p_{FA}$ in $\D$, and section $e$ of $p_{\Id_{FA}(Fa,a')}$, there exist lifts of $a'$, $e$ to $\C$.
    \end{itemize}
  
  \item $F$ is a \defemph{local trivial fibration} if it satisfies
    \begin{itemize}
    \item \defemph{type lifting}: given any $\Gamma \in \C$ and $A \in \Ty_\D (F\Gamma)$, there exists some $\overline{A} \in \Ty_\C \Gamma$ such that $F\overline{A} = A$;
    \item \defemph{term lifting}: given any $\Gamma \in \C$, $A \in \Ty_\C\Gamma$, and section $a$ of $FA$ in $\D$, there exists some section $\overline{a}$ of $A$ such that $F\overline{a} = a$.
    \end{itemize}
  
  \item $F$ is a \defemph{local equivalence} if it satisfies
    \begin{itemize}
    \item \defemph{weak type lifting}: given any $\Gamma \in \C$ and $A \in \Ty_\D (F\Gamma)$, there exists $\overline{A} \in \Ty_\C \Gamma$ and equivalence $w \colon F\Gamma. F\overline{A} \weqto F\Gamma.A$ over $F\Gamma$;
      
    \item \defemph{weak term lifting}: given any $\Gamma \in \C$, $A \in \Ty_\C\Gamma$, and section $a$ of $FA$ in $\D$, there exist some section $\overline{a}$ of $A$ and propositional equality $p \colon \Id_{FA}(F\overline{a},a)$.
    \end{itemize}
  \end{enumerate}
\end{definition}

We give these results just for \emph{homotopical} diagram CwA’s: the non-homotop-ical version then arises as a special case, by considering any ordinary category as a homotopical category with no maps marked as equivalences.  
No generality is lost, since local fibrations and equivalences already require consideration of $\Id$-types.

For the remainder of this section, fix a CwA $\C$ with identity types.

\subsection{Fibrations between diagram CwA’s}

\begin{lemma} \label{lem:fibration-between-diagram-cats}
  Let $u :\J \to \I$ be an \emph{injective} ordered homotopical discrete opfibration of ordered homotopical inverse categories.
  (In other words, $\J$ is a downward-closed subcategory of $\I$, or equivalently a co-sieve on $\I$.)
  Then the induced map $u^* \colon \C^\I \to \C^\J$ of homotopical diagram CwA’s is a local fibration.
\end{lemma}

\begin{proof}
  For type lifting, consider $\Gamma \in \C^\I$, $A \in \Ty^\I(\Gamma)$, $B \in \Ty^\J(u^*\Gamma)$, and an equivalence $e \colon u^*\Gamma.B \to u^*\Gamma. u^*A$ over $u^*\Gamma$.
  We will define an extension $\overline{B} \in \Ty^{\disc{\I}}(\Gamma)$ of $B$, along with a levelwise equivalence $\overline{e} \colon \Gamma. \overline{B} \to \Gamma. A$ over $\Gamma$ extending $e$.
  By Proposition~\ref{prop:diag-equiv-to-homot-diag}, $\overline{B}$ must be homotopical since $A$ is, and by Lemma~\ref{lem:equivalence-iff-levelwise}, $\overline{e}$ will be an equivalence in $\C^\I$, as required.
 
  As usual, we work by induction on $j \in \I$.
  If $j = ui$ for some $i \in \J$, we take $\overline{B}_j = B_i$ and $\overline{e}_j = e_i$ (an equivalence by \ref{lem:equivalence-iff-levelwise}).
  Otherwise, if $j \notin \im u$, take $\overline{B}_j$ to be $(M_j \overline{e})^* A_j$ and $\overline{e}_j$ to be $(M_j \overline{e}). A_j$:
  \[\begin{tikzcd}[column sep=large]
    M_j\overline{B}. (M_j \overline{e})^* A_j \ar[r, "\overline{e}_j", weq'] \ar[d,fib] \ar[drpb] 
    & M_j A. A_j \ar[d,fib] \\ 
    M_j \overline{B} \ar[r, "M_j\overline{e}", weq'] & M_j A.
  \end{tikzcd}\]
  By Lemmas~\ref{lem:equivalences-and-reedy-limits} and \ref{lem:levelwise-equivalence-iff-reedy}, $M_j \overline{e}$ is an equivalence, and hence by right properness, so is $\overline{e}_j$.

  For term lifting, the argument is exactly parallel, using transport along a “matching map” propositional equality.
\end{proof}

\subsection{Equivalences between diagram CwA’s}

\begin{definition}[{\jpaacite[p.~24]{szumilo:two-models}}]
  A homotopical functor $u \colon \J \to \I$ between homotopical categories is a \defemph{homotopy equivalence} if there is some homotopical functor $v \colon \I \to \J$ such that $uv$ and $vu$ are homotopic, via zigzags of natural weak equivalences, to the respective identity functors.
\end{definition}

The main remaining goal of this section is to show that if an ordered discrete opfibration $u \colon \J \to \I$ is a homotopy equivalence, then the induced map $u^* \colon \C^\I \to \C^\J$ is an equivalence of CwA’s (Theorem~\ref{thm:homotopy-equivalence} below).
This involves several technical difficulties, arising from two main sources: firstly, from the zigzags witnessing the homotopy equivalence, and further zigzags and grids arising from these; and secondly, from the fact that the functors appearing in these zigzags may not be discrete opfibrations, so will not act on Reedy types.
We therefore first establish some tools for dealing with these issues.

\begin{lemma} \label{lem:reedy-factorisation}
  Let $\J$ be an ordered homotopical inverse category, and suppose $\C$ is equipped with $\Sigma$-types.
  Suppose $f \colon \Gamma.\AA \to \Gamma.\BB$ is a map between levelwise extensions over $\Gamma$.
  Then $f$ can be factored as $\Gamma.\AA \to[w] \Gamma.\BB.C \to \Gamma.\BB$, where $C$ is a Reedy type over $\Gamma.\BB$ and $w$ is a levelwise equivalence.
\end{lemma}

\begin{proof}
  Work by Reedy induction; suppose $C$, $w$ have been constructed in dimensions $<i$.
  Then we have a map $g_i \colon (\Gamma.\AA)_i \to M_i C$ over $\Gamma_i$, induced by $f_i$ together with $w_j$ for $j \in i/\J$.
  Since $g_i$ lies in the fibrant slice $\C(\Gamma_i)$, we can obtain $C_i$, $w_i$ by factoring it as $(\Gamma.\AA)_i \to[w_i] M_i C . C_i \fibto M_i C$ using \jpaacite[Lem.~3.2.11]{avigad-kapulkin-lumsdaine}; that is, using the factorisation of \jpaacite[Lem.~11]{gambino-garner}, then taking an iterated $\Sigma$-type to collapse the context extension factor into a single type.
  %
  %
\end{proof}

An important special case of Lemma~\ref{lem:reedy-factorisation} is when $\BB$ is the empty context extension; this gives the “Reedy replacement” of a levelwise extension $\AA$ over $\Gamma$.

\begin{lemma} \label{lem:levelwise-to-reedy-descent}
  Let $\J$ be an ordered homotopical inverse category, and suppose $\C$ is equipped with $\Sigma$-types.
  Suppose $e \colon \Gamma \weqto \Delta$ is an equivalence in $\C^\J$, and $\AA$ a levelwise extension over $\Gamma$.
  Then there is a \emph{Reedy} type $B$ over $\Delta$, and equivalence $\bar{e} \colon \Gamma.\AA \to \Delta.B$ over $e$.
\end{lemma}

\begin{proof}
   First, by Lemma~\ref{lem:reedy-factorisation}, we obtain a Reedy replacement for $\AA$, i.e.\ a Reedy type $A'$ with an equivalence $w_0 \colon \Gamma.\AA \weqto \Gamma.A'$  over $\Gamma$.
   Then, since $e$ is an equivalence, its weak type lifting property gives a Reedy type $B$ over $\Delta$ along with an equivalence $w_1 \colon \Gamma.A' \weqto \Gamma.e^*B$ over $\Gamma$.
   Taking $\bar{e} \coloneqq (e.B)w_1w_0 \colon \Gamma.\AA \to \Delta.B$, we are done.
\end{proof}

\begin{lemma} \label{lem:span-equiv-to-map}
  Fix an ordered homotopical inverse category $\J$; suppose $\C$ is equipped with $\Sigma$-types.
  Let $\Gamma \in \C^\J$ be a diagram, $A_1$, $A_2$ Reedy types over $\Gamma$, and $B$ a levelwise extension over $\Gamma$, along with levelwise equivalences $e_\varepsilon \colon \Gamma.B \to \Gamma.A_\varepsilon$ over $\Gamma$ (for $\varepsilon = 0,1$).
  Then there is an equivalence $b \colon \Gamma.A_0 \to \Gamma.A_1$ over $\Gamma$.
\end{lemma}

\begin{proof}
  First factor $(e_0,e_1) \colon \Gamma.B \to \Gamma.A_0.A_1$ as an equivalence followed by a Reedy type $\Gamma.A_0.A_1.C \fibto \Gamma.A_0.A_1$, by Lemma~\ref{lem:reedy-factorisation}.
  By 2 out of 3, the maps $p_{\varepsilon} \colon \Gamma.A_0.A_1.C \to \Gamma.A_{\varepsilon}$ are both equivalences.
  Since $p_0$ is both an equivalence and a context extension, it has some section $s_0$; the composite $p_1 s_0$ gives an equivalence $\Gamma.A_0 \to \Gamma.A_1$ as desired.
\end{proof}

\begin{lemma} \label{lem:zigzag-improvement}
  Let $\J$ be an ordered homotopical inverse category.
  Suppose given a zigzag of diagrams $\gamma \colon \Gamma_0 \zigzagto \Gamma_n$, and a zigzag $\alpha \colon \Gamma_0.A_0 \zigzagto \Gamma_n.A_n$ of levelwise extensions over this, whose endpoints are moreover Reedy types $A_0, A_n$:
  \[ \begin{tikzcd}[sep=tiny]
    & \Gamma_1.\AA_1 \ar[dl] \ar[dr] \ar[dd,lw] & & \phantom{\Gamma.\AA} \ar[dl] \ar[r,phantom,description,"\cdots"] & \phantom{\Gamma.\AA} \ar[dr] & \\
    \Gamma_0.A_0 \ar[dd,fib] & & \Gamma_2.\AA_2 \ar[dd,lw] &  &  & \Gamma_n.A_n \ar[dd,fib] \\
    & \Gamma_1 \ar[dl] \ar[dr] & & \phantom{\Gamma} \ar{dl} \ar[r,phantom,description,"\cdots"] & \phantom{\Gamma} \ar[dr] \\
    \Gamma_0 & & \Gamma_2 & & & \Gamma_n
  \end{tikzcd} \]
  with all maps in the zigzags equivalences.

  Then we can replace $\alpha$ with a new zigzag $\alpha'$, also over $\gamma$, and with the same endpoints $\Gamma.A_0$, $\Gamma.A_n$, but with all objects occurring in $\alpha'$ being \emph{Reedy} types over $\Gamma_i$.
\end{lemma}

\begin{proof}
  First, for each cospan in the zigzag, replace its vertex $\AA_{2i}$ with its Reedy-fibrant replacement $A_{2i}$ according to Lemma~\ref{lem:reedy-factorisation}.
  Next, for the vertex $\AA_{2i+1}$ of each span in the zigzag, pull back its neighbours $A_{2i}$, $A_{2i+2}$ to $\Gamma_{2i+1}$.
  Lemma~\ref{lem:span-equiv-to-map} then gives us an equivalence $e_i \colon e^*A_{2i} \to g^*A_{2i+1}$; now replace each $\AA_{2i+1}$ with $e^*A_{2i}$, using $e_i$ to produce the new span.
\end{proof}

\begin{lemma} \label{lem:square-completion}
  Let $\J$ be an ordered homotopical inverse category.
  Suppose given a commutative square of base diagrams in $\C^\J$, with all maps equivalences, and Reedy types and equivalences between them over some proper connected subgraph of the boundary of the square, e.g.
  \[
  \begin{tikzcd}
    \Gamma_3 \ar[r,weq] \ar[d,weq] & \Gamma_2 \ar[d,weq] \\
    \Gamma_1 \ar[r,weq] & \Gamma_0
  \end{tikzcd}
  \qquad 
  \begin{tikzcd}
    \Gamma_3.A_3 \ar[r,weq] \ar[d,weq] & \Gamma_2.A_2 \\
    \Gamma_1.A_1 \ar[r,weq] & \Gamma_0.A_0
  \end{tikzcd}
  \]
  Then this partial square can be completed to a full square (not necessarily commutative) of Reedy types and equivalences over the boundary of the base square.
\end{lemma}

\begin{proof}
  Write $w$ for the composite equivalence $\Gamma_3 \weqto \Gamma_0$.
 
  By some combination of pulling back types along the base maps and pushing them forward according to Lemma~\ref{lem:levelwise-to-reedy-descent}, it is clear that any connected partial boundary can be completed to a diagram of one of the two following forms:
  \[
  \text{(I):} \hskip 0.5em plus 0.5em
  \begin{tikzcd}[ampersand replacement=\&]
    \begin{array}{cc} &\Gamma_3.A_3'\\ \Gamma_3.A_3 & \end{array}
          \ar[r,weq,start anchor={[yshift=1ex,xshift=-1ex]},near start]
          \ar[d,weq,start anchor={[xshift=-3.5ex,yshift=0.5ex]}]
      \& \Gamma_2.A_2 \ar[d,weq] \\
    \Gamma_1.A_1 \ar[r,weq] \& \Gamma_0.A_0
  \end{tikzcd}
  \hskip 1em plus 1em
  \text{(II):} \hskip 0.5em plus 0.5em
  \begin{tikzcd}[ampersand replacement=\&]
    \Gamma_3.A_3 \ar[d,weq] \ar[r,weq]
    \& \Gamma_2.A_2 \ar[d,weq,end anchor={[xshift=3.5ex,yshift=-0.5ex]}] \\
    \Gamma_1.A_1 \ar[r,weq,end anchor={[yshift=-1ex,xshift=1ex]}]
    \& \begin{array}{cc} &\Gamma_0.A_0'\\ \Gamma_0.A_0 & \end{array}
  \end{tikzcd}
  \]

  In case (I), pulling $A_0$ back along $w$ we get a cospan of types and equivalences $e \colon \Gamma_3.A_3 \weqto \Gamma_3.w^*A_0$, $e' \colon \Gamma_3.A_3' \to \Gamma_3.w^*A_0$ over $\Gamma$.
  Then $\Gamma_3.A_3.A_3'.(e,e')^*\Id_{w^*A_0}$ gives a span-equiva\-lence from $A_3$ to $A_3'$ over $\Gamma$.
  But now at most one of $A_3$, $A_3'$ was in the original partial boundary; without loss of generality, suppose $A_3'$ was not in it.
  Then use Lemma~\ref{lem:span-equiv-to-map} to turn the span-equivalence into an equivalence $\Gamma_3.A_3 \weqto \Gamma_3.A_3'$ over $\Gamma_3$.
  Composing this with the equivalence $\Gamma_3.A_3' \weqto \Gamma_2.A_2$,  we have the full square as required.

  In case (II), $A_3$ gives a span-equivalence from $w^*A_0$ to $w^*A_0'$ over $\Gamma_3$.
  Again, we may suppose without loss of generality that $A_0'$ was not in the original partial boundary.
  Then use Lemma~\ref{lem:span-equiv-to-map} to obtain an equivalence $\Gamma_3.w^*A_0' \weqto \Gamma_3.w^*A_0$ over $\Gamma_3$.
  Since $w$ is an equivalence, this descends to some equivalence $\Gamma_0.A_0' \weqto \Gamma_0.A_0$ over $\Gamma_0$.
  But now composing this with the equivalence $\Gamma_2.A_2 \weqto \Gamma_0.A_0'$, we once again have the required full square.
\end{proof}

\begin{lemma}  \label{lem:grid-completion}
  Let $\J$ be an ordered homotopical inverse category.
  Suppose we have a commutative grid of diagrams $\Gamma_{i,j}$ in $\C^\J$, with each row and column a zigzag of equivalences, as in the diagram below.
  Suppose moreover we have Reedy types $A_0, D_0$ over $\Gamma_{0,0}$, and a path of zigzags of levelwise extensions and equivalences between them going around the boundary of the base grid:
  \newcommand{\ph}{\phantom{\Gamma}}
  \[ \begin{tikzcd}[column sep=small,row sep=scriptsize,font=\scriptsize,baseline={(current bounding box.south)}]
    \Gamma_{0,0} 
      & \Gamma_{0,1} \ar[l,weq'] \ar[r,weq]
      & \ph \ar[r,phantom,description,"\cdots"]
      & \ph \ar[r,weq]
      & \Gamma_{0,m} \\
    \Gamma_{1,0} \ar[u,weq] \ar[d,description,weq']
      & \Gamma_{1,1} \ar[l,weq'] \ar[r,weq] \ar[u,weq] \ar[d,description,weq']
      & \ph \ar[r,phantom,description,"\cdots"]
      & \ph \ar[r,weq]
      & \Gamma_{1,m} \ar[u,weq'] \ar[d,description,weq] \\
    \ph \ar[d,description,phantom,"\vdots"] & \ph & \ph \ar[dr,description,phantom,"\ddots"] & \ph & \ph \ar[d,description,phantom,"\vdots"] \\
    \ph \ar[d,description,weq'] & \ph \ar[d,description,weq'] & \ph & \ph & \ph \ar[d,description,weq] \\
    \Gamma_{n,0}
      & \Gamma_{n,1} \ar[l,weq'] \ar[r,weq]
      & \ph \ar[r,phantom,description,"\cdots"]
      & \ph \ar[r,weq]
      & \Gamma_{n,m}
    \end{tikzcd}
    \hskip 1em plus 3em
    \begin{tikzcd}[column sep=small,row sep=scriptsize,font=\scriptsize,ampersand replacement=\&,baseline={(current bounding box.south)}]
    \begin{array}{cc} & A_0\\ D_0 & \end{array}
      \& A_0 \ar[l,weq',end anchor={[yshift=0.75ex,xshift=-1.25ex]}]  \ar[r,weq]
      \& \ph \ar[r,phantom,description,"\cdots"]
      \& \ph \ar[r,weq]
      \& A_m=B_0 \\
    D_1 \ar[u,weq,end anchor={[xshift=-2ex,yshift=0.5ex]}] \ar[d,description,weq']
      \& \ph
      \& \ph 
      \& \ph
      \& B_1 \ar[u,weq'] \ar[d,description,weq] \\
    \ph \ar[d,description,phantom,"\vdots"] \& \ph \& \ph \& \ph \& \ph \ar[d,description,phantom,"\vdots"] \\
    \ph \ar[d,description,weq'] \& \ph \& \ph \& \ph \& \ph \ar[d,description,weq] \\
    D_n = C_0
      \& C_1 \ar[l,weq']  \ar[r,weq]
      \& \ph \ar[r,phantom,description,"\cdots"]
      \& \ph \ar[r,weq]
      \& C_m = B_n
    \end{tikzcd}\]

  Then there is some equivalence $A_0 \weqto D_0$ over $\Gamma_{0,0}$.
\end{lemma}

\begin{proof}
  First, by applying Lemma~\ref{lem:zigzag-improvement}, we can assume that all the types involved are Reedy.
  By repeatedly applying Lemma~\ref{lem:square-completion}, starting from the bottom right and proceeding upwards and leftwards, we can complete the path to a (not necessarily commutative) grid of Reedy types and equivalences over the base grid, except for the top-left square, which ends up like case (I) of the proof of Lemma~\ref{lem:square-completion}.
  But now, just as in the proof of that case, we can obtain from this square an equivalence $A_0 \weqto D_0$ over $\Gamma_{0,0}$ as required.
\end{proof}

We are now prepared for the main result:
\begin{theorem} \label{thm:homotopy-equivalence}
 Let $u \colon \J \to \I$ be an ordered homotopical discrete opfibration between ordered homotopical inverse categories.
 If $u$ is a homotopy equivalence, then the CwA map $u^* \colon \C^\I \to \C^\J$ is a local weak equivalence.
\end{theorem}

\begin{proof}
  Take $v \colon \I \to \J$, along with zigzags $\eta \colon \id \zigzagto uv $, $\varepsilon \colon vu \zigzagto \id$ of natural equivalences witnessing that $u$ is a homotopy equivalence.
  
  Note we cannot assume that $v$ or the functors appearing in $\eta$, $\varepsilon$ are discrete opfibrations, so they may not act on Reedy types;
  but they do act on \emph{levelwise} extensions, and on the underlying diagram categories, this action is 2-functorial.

  First, we show weak type lifting for $u^*$.
  The argument is rather involved, but in outline is entirely analogous to showing that an equivalence of categories $F \colon \C \to \D$  (presented via unit and counit isomorphisms) induces an essentially surjective map on slices:
  \begin{itemize}
  \item given $A \to FC$, apply a quasi-inverse $G$ to get $GA \to GFC$; composing with the inverse unit $\eta_C^{-1} \colon GFC \to C$ gives a map $GA \to C$;
  \item the co-unit gives an isomorphism $\varepsilon_A \colon FGA \iso A$, and provided the equivalence was adjoint, this isomorphism will be over $FC$;
  \item if the original equivalence was not assumed adjoint, one adjointifies it beforehand, replacing $\varepsilon$ with the modified co-unit $\varepsilon \cdot G \eta^{-1}_F \cdot GF\varepsilon^{-1}$.
  \end{itemize}
  
  Returning to weak type lifting, take $\Gamma \in \C^\I$, and $A \in \Ty^\J(u^* \Gamma)$.
  Then $v^*A$ is a levelwise extension over $v^*u^*\Gamma$; so by alternately pulling back and pushing forward along the zigzag $\eta^*_\Gamma$ (using Lemma~\ref{lem:levelwise-to-reedy-descent}), we obtain a type $\bar{A}$ over $\Gamma$ connected to $v^*A$ by a zigzag $\alpha \colon \bar{A} \zigzagto v^*A$ of levelwise equivalences between levelwise extensions over $\eta^* \Gamma \colon \Gamma \zigzagto v^*u^*\Gamma$.
  (Here and in the rest of this proof, we write just $A$ for the total object $u^* \Gamma . A$, and similarly for other Reedy/levelwise extensions.)

  We now need to show that $u^* \bar{A} \equiv A$ over $u^* \Gamma$.
  For this, we start by constructing a commutative grid of functors $\J \to \I$ and zigzags of natural weak equivalences between them:
  \newcommand{\e}{\varepsilon}
  \[\begin{tikzcd}
    u \ar[r,equal] \ar[d,zigzag,"u\e^{-1}"'] & u \ar[r,equal] \ar[d,zigzag,"u\e^{-1}"'] & u \ar[r,zigzag,"\eta_u"] \ar[d,zigzag,"u\e^{-1}"'] & uvu \ar[d,zigzag,"uvu\e^{-1}"] \ar[dl,phantom,"\scriptstyle \text{(nat)}"] \\
    uvu \ar[r,equal] \ar[d,equal] & uvu \ar[r,equal] \ar[d,zigzag,"\eta^{-1}_u"'] & uvu \ar[r,zigzag,"\eta_{uvu}"] \ar[d,equal] & uvuvu \ar[d,equal] \\
    uvu \ar[r,zigzag,"\eta^{-1}_u"] \ar[d,equal] & u \ar[r,zigzag,"\eta_u"] \ar[d,zigzag,"\eta_u"'] & uvu \ar[r,zigzag,"\eta_{uvu}"] \ar[d,zigzag,"\eta_{uvu}"] & uvuvu \ar[d,equal] \\
    uvu \ar[r,equal] & uvu \ar[r,zigzag,"uv\eta_u"'] \ar[ur,phantom,"\scriptstyle \text{(nat)}"] & uvuvu \ar[r,equal] & uvuvu
  \end{tikzcd}\]
  Here each square is itself a grid, of one of three kinds (up to orientation): either a “naturality” grid, i.e.\ for some zizgags $\beta \colon b_0 \zigzagto b_m$ and $\gamma \colon c_0 \zigzagto c_n$ the grid with $(i,j)$th entry $b_ic_j$,
  \newcommand{\ph}{\phantom{b_mc_n}}
  \[ \begin{tikzpicture}[cd-style, row sep = 1em, column sep = 0.6 em, font=\scriptsize]
    \matrix (g) [matrix of math nodes]{
      b_0c_0      & b_0c_1       & b_0 c_2 & \ph &\ph & b_0c_{n-1}  & b_0c_n \\
      b_1c_0      & b_1c_1       & b_1c_2  & \ph &\ph & b_1c_{n-1}  & b_1c_n \\
      b_2c_0      & b_2c_1       & \ph       & \ph &\ph & \ph           & b_2c_n \\
      \ph & \ph & \ph & \ph &\ph & \ph & \ph \\
      \ph & \ph & \ph & \ph &\ph & \ph & \ph \\
      b_{m-1}c_0 & b_{m-1}c_1 & \ph       & \ph & \ph & \ph          & b_{m-1}c_n \\
      b_mc_0     & b_mc_1      & b_mc_2 & \ph & \ph & b_mc_{n-1} & b_mc_n \\
    };

    \arr (g-1-2) to (g-1-1);
    \arr (g-1-2) to (g-1-3);
    \arr (g-1-4) to (g-1-3);
    \arr[description,phantom] (g-1-4) to node{$\displaystyle \cdots$} (g-1-5); 
    \arr (g-1-6) to (g-1-5);
    \arr (g-1-6) to (g-1-7);
    \arr (g-2-1) to (g-1-1);
    \arr (g-2-2) to (g-1-2);
    \arr (g-2-3) to (g-1-3);
    \arr (g-2-6) to (g-1-6);
    \arr (g-2-7) to (g-1-7);
    \arr (g-2-2) to (g-2-1);
    \arr (g-2-2) to (g-2-3);
    \arr (g-2-4) to (g-2-3);
    \arr (g-2-6) to (g-2-5);
    \arr (g-2-6) to (g-2-7);
    \arr (g-2-1) to (g-3-1);
    \arr (g-2-2) to (g-3-2);
    \arr (g-2-3) to (g-3-3);
    \arr (g-2-6) to (g-3-6);
    \arr (g-2-7) to (g-3-7);
    \arr (g-3-2) to (g-3-1);
    \arr (g-3-2) to (g-3-3);
    \arr (g-3-6) to (g-3-7);
    \arr (g-4-1) to (g-3-1);
    \arr (g-4-2) to (g-3-2);
    \arr (g-4-7) to (g-3-7);
    \arr[description,phantom] (g-4-1) to node {$\vdots$} (g-5-1); 
    \arr[description,phantom] (g-4-4) to node {$\ddots$} (g-5-5); 
    \arr[description,phantom] (g-4-7) to node {$\vdots$} (g-5-7); 
    \arr (g-6-1) to (g-5-1);
    \arr (g-6-2) to (g-5-2);
    \arr (g-6-7) to (g-5-7);
    \arr (g-6-2) to (g-6-1);
    \arr (g-6-2) to (g-6-3);
    \arr (g-6-6) to (g-6-7);
    \arr (g-6-1) to (g-7-1);
    \arr (g-6-2) to (g-7-2);
    \arr (g-6-3) to (g-7-3);
    \arr (g-6-6) to (g-7-6);
    \arr (g-6-7) to (g-7-7);
    \arr (g-7-2) to (g-7-1);
    \arr (g-7-2) to (g-7-3);
    \arr (g-7-4) to (g-7-3);
    \arr[description,phantom] (g-7-4) to node {$\displaystyle \cdots$} (g-7-5); 
    \arr (g-7-6) to (g-7-5);
    \arr (g-7-6) to (g-7-7);

    \node (top-label-vert) at ($ (g-1-1.north) + (0,0.3) $) {}; 
    \arr[zigzag,description]
      (g-1-1 |- top-label-vert)
      to node [preaction={fill,white}] {$b_0\gamma$}
      (g-1-7 |- top-label-vert);

    \node (left-label-horiz) at ($ (g-6-1.west) - (0.35,0) $) {};
    \arr[zigzag,description]
      (g-1-1 -| left-label-horiz)
      to node [preaction={fill,white}] {$\beta_{c_0}$}
      (g-7-1 -| left-label-horiz);

    \node (bottom-label-vert) at ($ (g-7-1.south) - (0,0.3) $) {};
    \arr[zigzag,description]
      (g-7-1 |- bottom-label-vert)
      to node [preaction={fill,white}] {$b_m\gamma$}
      (g-7-7 |- bottom-label-vert);

    \node (right-label-horiz) at ($ (g-6-7.east) + (0.35,0) $) {};
    \arr[zigzag,description]
      (g-1-7 -| right-label-horiz)
      to node [preaction={fill,white}] {$\beta_{c_n}$}
      (g-7-7 -| right-label-horiz);
  \end{tikzpicture} \]
  or a “connection” grid, i.e.\ for some zigzag $\beta \colon b_0 \zigzagto b_m$ the grid with $(i,j)$th object $b_{\max(i,j)}$,
  \[ \begin{tikzpicture}[cd-style, row sep = 1em, column sep = 0.6 em, font=\scriptsize]
    \matrix (g) [matrix of math nodes]{
      b_0      & b_1       & b_2 & \ph &\ph & b_{m-1}  & b_m \\
      b_1      & b_1       & b_2 & \ph &\ph & b_{m-1}  & b_m \\
      b_2      & b_2       & b_2 & \ph &\ph & \ph       & b_m \\
      \ph & \ph & \ph & \ph &\ph & \ph & \ph \\
      \ph & \ph & \ph & \ph &\ph & \ph & \ph \\
      b_{m-1} & b_{m-1} & \ph & \ph & \ph & b_{m-1} & b_m \\
      b_m     & b_m      & b_m & \ph & \ph & b_m     & b_m \\
    };

    \arr (g-1-2) to (g-1-1);
    \arr (g-1-2) to (g-1-3);
    \arr (g-1-4) to (g-1-3);
    \arr[description,phantom] (g-1-4) to node{$\displaystyle \cdots$} (g-1-5); 
    \arr (g-1-6) to (g-1-5);
    \arr (g-1-6) to (g-1-7);
    \arr (g-2-1) to (g-1-1);
    \arr (g-2-2) to (g-1-2);
    \arr (g-2-3) to (g-1-3);
    \arr (g-2-6) to (g-1-6);
    \arr (g-2-7) to (g-1-7);
    \arr (g-2-2) to (g-2-1);
    \arr (g-2-2) to (g-2-3);
    \arr (g-2-4) to (g-2-3);
    \arr (g-2-6) to (g-2-5);
    \arr (g-2-6) to (g-2-7);
    \arr (g-2-1) to (g-3-1);
    \arr (g-2-2) to (g-3-2);
    \arr (g-2-3) to (g-3-3);
    \arr (g-2-6) to (g-3-6);
    \arr (g-2-7) to (g-3-7);
    \arr (g-3-2) to (g-3-1);
    \arr (g-3-2) to (g-3-3);
    \arr (g-3-4) to (g-3-3);
    \arr (g-3-6) to (g-3-7);
    \arr (g-4-1) to (g-3-1);
    \arr (g-4-2) to (g-3-2);
    \arr (g-4-3) to (g-3-3);
    \arr (g-4-7) to (g-3-7);
    \arr[description,phantom] (g-4-1) to node {$\vdots$} (g-5-1); 
    \arr[description,phantom] (g-4-4) to node {$\ddots$} (g-5-5); 
    \arr[description,phantom] (g-4-7) to node {$\vdots$} (g-5-7); 
    \arr (g-6-1) to (g-5-1);
    \arr (g-6-2) to (g-5-2);
    \arr (g-6-6) to (g-5-6);
    \arr (g-6-7) to (g-5-7);
    \arr (g-6-2) to (g-6-1);
    \arr (g-6-2) to (g-6-3);
    \arr (g-6-6) to (g-6-5);
    \arr (g-6-6) to (g-6-7);
    \arr (g-6-1) to (g-7-1);
    \arr (g-6-2) to (g-7-2);
    \arr (g-6-3) to (g-7-3);
    \arr (g-6-6) to (g-7-6);
    \arr (g-6-7) to (g-7-7);
    \arr (g-7-2) to (g-7-1);
    \arr (g-7-2) to (g-7-3);
    \arr (g-7-4) to (g-7-3);
    \arr[description,phantom] (g-7-4) to node {$\displaystyle \cdots$} (g-7-5); 
    \arr (g-7-6) to (g-7-5);
    \arr (g-7-6) to (g-7-7);
 
    \node (top-label-vert) at ($ (g-1-1.north) + (0,0.3) $) {}; 
    \arr[zigzag,description]
      (g-1-1 |- top-label-vert)
      to node [preaction={fill,white}] {$\beta$}
      (g-1-7 |- top-label-vert);

    \node (left-label-horiz) at ($ (g-6-1.west) - (0.35,0) $) {};
    \arr[zigzag,description]
      (g-1-1 -| left-label-horiz)
      to node [preaction={fill,white}] {$\beta $}
      (g-7-1 -| left-label-horiz);

    \node (bottom-label-vert) at ($ (g-7-1.south) - (0,0.3) $) {};
    \arr[equal,description]
      (g-7-1 |- bottom-label-vert)
      to node [preaction={fill,white}] {$\id_{b_m}$}
      (g-7-7 |- bottom-label-vert);

    \node (right-label-horiz) at ($ (g-6-7.east) + (0.35,0) $) {};
    \arr[equal,description]
      (g-1-7 -| right-label-horiz)
      to node [preaction={fill,white}] {$\id_{b_m}$}
      (g-7-7 -| right-label-horiz);
  \end{tikzpicture} \]
  or else an evident “degeneracy” grid.
  Overall, comparing the top edge with the path around the sides and bottom, the complete grid is analogous to the triangle equality for the adjointification of an equivalence, $\eta_u \homot u(\varepsilon \cdot v \eta^{-1}_u \cdot vu\varepsilon^{-1})^{-1}$.
  
  \optionaltodo{Probably the big diagrams of the “naturality” and “connection” grids could be cut down a bit while still showing enough information, maybe even fit side-by-side.}

  Postcomposing the grid with $\Gamma$ gives us a corresponding commutative grid of base diagrams and equivalences in $\C^\J$:
  \[\begin{maxwidthtikzcd}[font=\scriptsize,column sep=large,row sep=normal]
    u^*\Gamma \ar[r,equal] \ar[d,zigzag,"(\e^{-1})^*u^*\Gamma"']
      & u^*\Gamma \ar[r,equal] \ar[d,zigzag,"(\e^{-1})^*u^*\Gamma"']
      & u^*\Gamma \ar[r,zigzag,"u^*\eta^*\Gamma"] \ar[d,zigzag,"u^*(\e^{-1})^*\Gamma"']
      & u^*v^*u^*\Gamma \ar[d,zigzag,"(\e^{-1})^* u^*v^*u^*\Gamma"] \ar[dl,phantom,"\scriptstyle \text{(nat)}"] \\
    u^*v^*u^*\Gamma \ar[r,equal] \ar[d,equal]
      & u^*v^*u^*\Gamma \ar[r,equal] \ar[d,zigzag,"u^*(\eta^{-1})^*\Gamma"']
      & u^*v^*u^*\Gamma \ar[r,zigzag,"u^*v^*u^*\eta^*\Gamma"] \ar[d,equal]
      & u^*v^*u^*v^*u^*\Gamma \ar[d,equal] \\
    u^*v^*u^*\Gamma \ar[r,zigzag,"u^*(\eta^{-1})^*\Gamma"] \ar[d,equal]
      & u^*\Gamma \ar[r,zigzag,"u^*\eta^*\Gamma"] \ar[d,zigzag,"u^*\eta^*\Gamma"']
      & u^*v^*u^*\Gamma \ar[r,zigzag,"u^*v^*u^*\eta^*\Gamma"] \ar[d,zigzag,"u^*v^*u^*\eta^*\Gamma"]
      & u^*v^*u^*v^*u^*\Gamma \ar[d,equal] \\
    u^*v^*u^*\Gamma \ar[r,equal]
      & u^*v^*u^*\Gamma \ar[r,zigzag,"u^*\eta^*v^*u^*\Gamma"'] \ar[ur,phantom,"\scriptstyle \text{(nat)}"]
      & u^*v^*u^*v^*u^*\Gamma \ar[r,equal]
      & u^*v^*u^*v^*u^*\Gamma
  \end{maxwidthtikzcd}\]

  We also have a path of zigzags of equivalences over the boundary of the grid above, going from $u^*\bar{A}$ around to $A$ (each over $u^*\Gamma$).
  Here the top edge is $u^*$ applied to the zigzag $\alpha \colon \bar{A} \zigzagto v^*A$ obtained in the construction of $\bar{A}$, appropriately padded with identity zigzags, while the other edges consist of $A$ postcomposed with a similar padding of the modified counit $\varepsilon\, v\eta^{-1}_u\, vu\varepsilon^{-1}$:
  \[\begin{tikzcd}[ampersand replacement=\&,font=\scriptsize,row sep=normal]
    \begin{array}{cc} & u^*\bar{A} \\ A & \end{array}
          \ar[r,equal,start anchor={[yshift=1ex,xshift=-1ex]}]
          \ar[d,zigzag,start anchor={[xshift=-2.4ex,yshift=0.5ex]},"(\e^{-1})^*A"',near start]
      \& u^*\bar{A} \ar[r,equal] 
      \&[4ex] u^*\bar{A} \ar[r,zigzag,"u^*\alpha"]
      \& u^*v^*A \ar[d,zigzag,"(\e^{-1})^* u^*v^*A"] \\
    u^*v^*A \ar[d,equal]
      \&  
      \& 
      \& u^*v^*u^*v^*A \ar[d,equal] \\
    u^*v^*A \ar[d,equal]
      \&  
      \& 
      \& u^*v^*u^*v^*A \ar[d,equal] \\
    u^*v^*A \ar[r,equal]
      \& u^*v^*A \ar[r,zigzag,"u^*\eta^*v^*A"']
      \& u^*v^*u^*v^*A \ar[r,equal]
      \& u^*v^*u^*v^*A
  \end{tikzcd}\]
  But now Lemma~\ref{lem:grid-completion} provides an equivalence $u^*\bar{A} \to A$ over $u^*\Gamma$, as required.

  The proof of weak term lifting is an analogous adaptation of the argument that an equivalence of categories, presented via a quasi-inverse, induces full functors on slices.
\end{proof}

\begin{corollary}
  Let $u \colon \J \to \I$ be an ordered discrete opfibration between ordered homotopical inverse categories that is both injective and a homotopy equivalence.
  Then $u^* \colon \C^\I \to \C^\J$ is a local trivial fibration.
\end{corollary}

\begin{proof}
  $u^*$ is a local fibration by Lemma~\ref{lem:fibration-between-diagram-cats} and a local equivalence by Theorem~\ref{thm:homotopy-equivalence}; so by \jpaacite[Prop.~4.20]{kapulkin-lumsdaine:homotopy-theory-of-type-theories}, it is a local trivial fibration.
\end{proof}


\section{Summary and outlook} \label{sec:conclusion} 

For quick reference, we summarise the main constructions of the paper as follows:
 
\begin{theorem} \leavevmode
  \begin{enumerate}
    \item %
      For any CwA $\C$ and ordered inverse category $\I$, there is a CwA $\C^\I$ of $\I$-diagrams and Reedy $\I$-types in $\C$, functorial with respect to CwA maps in $\C$ and ordered discrete opfibrations in $\I$.
      If $\C$ carries $\Id$-, $\Sigma$-, $\Pi$-, $\Pi_\eta$-, $\Pi_\ext$-, or unit type structure, then so does $\C^\I$, and the functorial action respects these.
      (Note however the requirement of \emph{ordered} discrete opfibrations, for the functoriality in $\I$.) 

    \item %
      For any CwA $\C$ with $\Id$-types and ordered homotopical inverse category $\I$, there is a sub-CwA $\C^\I \subseteq \C^{\disc{\I}}$ of homotopical diagrams and Reedy types, functorial with respect to homotopical CwA maps and ordered homotopical discrete opfibrations.
      $\C^\I$ is closed under the $\Id$-type structure on $\C^{\disc{\I}}$, along with $\Sigma$- and unit type structure when $\C$ is equipped with these, and $\Pi_\ext$-structure if $\C$ is equipped with this and additionally all maps in $\I$ are equivalences.
      (Note the requirement of functional extensionality in $\C$, for $\Pi$-structure on $\C^\I$.)

    \item %
      Let $\C$ be a CwA with $\Id$-types, and $u \colon\J \to \I$ an ordered homotopical discrete opfibration.
      If $u$ is injective, then the induced map $u^* \colon\C^\I \to \C^\J$ is a local fibration of CwA’s.
      If $u$ is a homotopy equivalence, then $u^*$ is a local equivalence.
  \end{enumerate}
\end{theorem}

\begin{proof}
  The construction of $\C^\I$ for non-homotopical $\I$ constitutes most of Section~\ref{sec:inverse-diagrams}, completed in Definition~\ref{def:cwa-of-strict-reedy-types}, and its logical structure is given in Section~\ref{sec:logical-structure}, summed up in Proposition~\ref{prop:logical-structure-summary}.
  The sub-CwA of homotopical diagrams is presented in Section~\ref{sec:homotopical-diagrams}, and summarised in Proposition~\ref{prop:homotopical-structure-summary}.
  The results on the precomposition maps $u^* \colon\C^\I \to \C^\J$ constitute Section~\ref{sec:fibrations-and-shit}.
\end{proof}

\begin{remark}[Applications]
  The present work was primarily motivated by the cases required in the companion paper \jpaacite[\textsection 5]{kapulkin-lumsdaine:homotopy-theory-of-type-theories}: specifically, the CwA $\C^\EqvCat$ of \emph{span-equivalences}, along with auxiliary constructions to witness reflexivity, transitivity, and symmetry of span-equivalences, all also presented as homotopical inverse diagram models.

  \optionaltodo{Should we say anything here about further applications?  They seem to fall into two classes:}
  
  \optionaltodo{Things we plan to follow in future work: conservativity kind of shit)}
  
  \optionaltodo{Things we hope other people will follow up because we can’t be bothered to but might mention here if we want to take partial credit for them anyway \ldots E.g.\ homotopy coherent diagrams by replacing non-inverse categories by homotopical inverse categories; cf.\ Szumi\l o (where exactly??) for doing this in fib cats.  E.g.\ also, semi-simplicial diagrams, cf.\ that Kraus paper; but also make sure to note the difference between this and the open problem of internal semi-simplicial types.}

  \optionaltodo{I [PLL] guess probably best not to mention the latter; and we should mention the former iff we can sketch out enough of a paper to confidently say it's “in prep”, which is very unlikely to happen before the first archive release of this.}
\end{remark}

\begin{remark}[Related work] \label{rmk:related-work}
  Special cases of the present construction, and closely related constructions, have previously been considered by several authors.

  \begin{enumerate}
  \item A special case of inverse diagrams was studied by Tonelli as the \emph{basic pairs} model \jpaacite{tonelli}.
    Precisely, the model studied there amounts to the diagram CwA $\C_\T^{\SpanCat}$, where $\C_\T$ is the syntactic CwA of the type theory considered there, and $\SpanCat$ is the “walking span.” \optionaltodo{Add: of Example~\ref{ex:span-inverse-category}.}

  \item The \emph{spreads} model of Martin-Löf \cite{martin-lof:spreads-seminars} can likewise be seen as the diagram CwA $\C_\T^{\N^\op}$, where $\N$ is the posetal inverse category $(\N,\leq)$.

  \item \label{item:shulman}
    A closely analogous general inverse diagrams construction is given by Shulman in \jpaacite{shulman:inverse-diagrams}.
    Key differences between that work and ours include:
    \begin{enumerate}
    \item Shulman works with \emph{type-theoretic fibration categories} rather than CwA’s --- a categorically somewhat cleaner setting than ours, and expected to be equivalent in some $\infty$-categorical sense, but corresponding less directly to traditional type theory;
    \item Shulman does not study the homotopical case;
    \item Shulman considers not just diagrams valued in a single CwA, but more general \emph{oplax diagrams} valued in a suitable diagram of CwA’s (cf.~Remark~\ref{rmk:generalisations}(\ref{item:oplax-diagrams}) below);
    \item besides the logical structure considered in the present work, Shulman additionally considers univalent universes (cf.~Remark~\ref{rmk:generalisations}(\ref{item:univalent-universes})).
    \end{enumerate}
  \item
    Kraus \cite[Ch.~8]{kraus:thesis} makes fruitful use of inverse diagrams, primarily as an auxiliary tool for studying a given base CwA (or type-theoretic fibration category) rather than as a model in their own right.
    \optionaltodo{Actually, we should read that chapter more carefully—its constructions overlap a bit with ours!}
    This direction is further developed by Kraus and Sattler in \jpaacite{kraus-sattler:space-valued-diagrams}, with a focus on types of diagrams \emph{internal} to the type theory under consideration.

  \item
    Inverse categories (or closely related structures) have also been used to specify dependently sorted signatures, in for instance Makkai’s FOLDS \cite{makkai:folds}, developed in the framework of CwA’s by Palmgren \jpaacite{palmgren:cwfs-and-folds}, and from another type-theoretic angle by Tsementzis \jpaacite{tsementzis:first-order-logic-with-isomorphism}.

  \item
    Inverse diagram models valued in $\Set$ are special cases of \emph{presheaf} models, which (over arbitrary index categories) are the subject of a very extensive literature.
  \end{enumerate}
\end{remark}

\begin{remark}[Generalisations] \label{rmk:generalisations}

The constructions of the present paper naturally suggest several further generalisations which we have not pursued
--- primarily for the sake of keeping this paper to a focused scope and manageable length,
but also since all our intended applications fit within the current framework, and so without further motivating examples to steer by, it was unclear what would be good natural hypotheses for these generalisations.

Nonetheless, we set down a few preliminary notes which may be helpful for readers wanting applications not covered by the present results,
since we have been frequently asked about such possibilities when speaking about this work. 

\begin{enumerate}
\item
  The definitions of Reedy types extend without difficulty from CwA’s to more general comprehension categories,
  and it seems likely that the constructions of logical structure on $\C^\I$ should also extend to that setting, using something like the pseudo-stable logical structure of Lumsdaine and Warren \jpaacite{lumsdaine-warren:local-universes}.
  
  However, the basic theory of such logical structure on comprehension categories is as yet very little developed; so extending the present constructions to that setting would require a significant amount of preliminary development.

\item \label{item:oplax-diagrams}
  Similarly, it seems likely that many of the present constructions should work without much modification not just for diagrams valued in a single CwA $\C$, but for some analogue of the oplax diagrams considered by Shulman in \jpaacite{shulman:inverse-diagrams}.

\item
  More concretely, our assumption in Proposition~\ref{prop:pi-types-homotopical} that all maps in $\I$ are equivalences is certainly stronger than necessary for constructing $\Pi$-types in $\C^\I$.
  
  On the one hand, there are at least some other cases where $\C^\I$ is closed under $\Pi$-types in $\C^{\disc{\I}}$: for instance, the degenerate case where \emph{no} maps in $\I$ are equivalences.
  On the other hand, one may hope that $\C^\I$ may sometimes have $\Pi$-types even if they do not agree with those of $\C^{\disc{\I}}$, as for instance in Example~\ref{ex:pi-types-in-weak-maps}.
  It seems unlikely to us that $\C^\I$ as defined here can admit $\Pi$-types in such cases for general $\C$; but it seems hopeful that by modifying $\C^\I$ to include chosen equivalence data, one might be able to construct $\Pi$-types for such examples.

\item \label{item:univalent-universes}
  Shulman \jpaacite{shulman:inverse-diagrams} shows that in the setting of type-theoretic fibration categories, univalent universes lift from the input categories to the categories of oplax inverse diagrams.
  It seems hopeful that an analogous construction should give univalent universes in our CwA’s of homotopical inverse diagrams, but we have not pursued this since it was not needed in our applications.

\item
  The results of Section~\ref{sec:fibrations-and-shit} were stated merely in terms of existence of liftings.
  However, their proofs gave rather more: \emph{pullback-stable} type- and term-lifting operations, which moreover are functorial in various senses.
  These can therefore be read as constructing \emph{structured} analogues of the fibrations and equivalences of the companion paper \jpaacite[\textsection 3]{kapulkin-lumsdaine:homotopy-theory-of-type-theories}.
\end{enumerate}
\end{remark}


\subsection*{Acknowledgements}

This work benefited greatly from the authors’ time together at the Isaac Newton Institute in Cambridge, during the Big Proof programme in summer 2017;
we would therefore like to thank the programme organisers --- especially Jeremy Avigad --- and the INI for their support and hospitality during this programme (supported by EPSRC grant EP/K032208/1).
The second-named author would also like to thank the members of the Stockholm Logic Seminar --- especially Erik Palmgren, Per Martin-Löf, Chaitanya Leena Subramaniam, and Håkon Gylterud --- for many useful comments and discussions.
We also thank Mike Shulman for catching a serious typo in the first public version, and an anonymous referee for very thorough and valuable feedback, in particular pushing us to improve the proof of Proposition~\ref{prop:pi-types-homotopical}.



\ifarxiv
  \bibliographystyle{amsalphaurlmod}
\else
  \bibliographystyle{elsarticle-num}
\fi

\bibliography{general-bibliography}

\newcommand{\etalchar}[1]{$^{#1}$}
\providecommand{\bysame}{\leavevmode\hbox to3em{\hrulefill}\thinspace}
\providecommand{\MR}{\relax\ifhmode\unskip\space\fi MR }
\providecommand{\MRhref}[2]{%
  \href{http://www.ams.org/mathscinet-getitem?mr=#1}{#2}
}
\providecommand{\href}[2]{#2}
\begin{thebibliography}{DHKS04}

\bibitem[AKL15]{avigad-kapulkin-lumsdaine}
Jeremy Avigad, Krzysztof Kapulkin, and Peter~LeFanu Lumsdaine, \emph{Homotopy
  limits in type theory}, Math. Structures Comput. Sci. \textbf{25} (2015),
  no.~5, 1040--1070, \href {http://arxiv.org/abs/1304.0680}
  {\path{arXiv:1304.0680}}, \href {http://dx.doi.org/10.1017/S0960129514000498}
  {\path{doi:10.1017/S0960129514000498}}.

\bibitem[Bro73]{brown:abstract-homotopy-theory}
Kenneth~S. Brown, \emph{Abstract homotopy theory and generalized sheaf
  cohomology}, Trans. Amer. Math. Soc. \textbf{186} (1973), 419--458, \href
  {http://dx.doi.org/10.2307/1996573} {\path{doi:10.2307/1996573}}.

\bibitem[DHKS04]{dhks}
William~G. Dwyer, Philip~S. Hirschhorn, Daniel~M. Kan, and Jeffrey~H. Smith,
  \emph{Homotopy limit functors on model categories and homotopical
  categories}, Mathematical Surveys and Monographs, vol. 113, American
  Mathematical Society, Providence, RI, 2004, \href
  {http://dx.doi.org/10.1090/surv/113} {\path{doi:10.1090/surv/113}}.

\bibitem[Gar09a]{garner:on-the-strength}
Richard Garner, \emph{On the strength of dependent products in the type theory
  of {M}artin-{L}\"of}, Ann. Pure Appl. Logic \textbf{160} (2009), no.~1,
  1--12, \href {http://arxiv.org/abs/0803.4466} {\path{arXiv:0803.4466}}, \href
  {http://dx.doi.org/10.1016/j.apal.2008.12.003}
  {\path{doi:10.1016/j.apal.2008.12.003}}.

\bibitem[Gar09b]{garner:2-d-models}
\bysame, \emph{Two-dimensional models of type theory}, Math. Structures Comput.
  Sci. \textbf{19} (2009), no.~4, 687--736, \href
  {http://arxiv.org/abs/0808.2122} {\path{arXiv:0808.2122}}, \href
  {http://dx.doi.org/10.1017/S0960129509007646}
  {\path{doi:10.1017/S0960129509007646}}.

\bibitem[GG08]{gambino-garner}
Nicola Gambino and Richard Garner, \emph{The identity type weak factorisation
  system}, Theoret. Comput. Sci. \textbf{409} (2008), no.~1, 94--109, \href
  {http://arxiv.org/abs/0803.4349} {\path{arXiv:0803.4349}}, \href
  {http://dx.doi.org/10.1016/j.tcs.2008.08.030}
  {\path{doi:10.1016/j.tcs.2008.08.030}}.

\bibitem[Kel82]{kelly:basic-concepts}
G.~Maxwell Kelly, \emph{Basic concepts of enriched category theory}, London
  Mathematical Society Lecture Note Series, vol.~64, Cambridge University
  Press, 1982, Also available in Reprints in Theory and Applications of
  Categories, No. 10 (2005) pp. 1--136,
  \url{http://www.tac.mta.ca/tac/reprints/articles/10/tr10abs.html}.

\bibitem[KL12]{kapulkin-lumsdaine:simplicial-model}
Krzysztof Kapulkin and Peter~LeFanu Lumsdaine, \emph{The simplicial model of
  {Univalent} {Foundations} (after {Voevodsky})}, {Journal} of the European
  Mathematical Society (to appear), preprint 2012, \href
  {http://arxiv.org/abs/1211.2851} {\path{arXiv:1211.2851}}.

\bibitem[KL18]{kapulkin-lumsdaine:homotopy-theory-of-type-theories}
\bysame, \emph{The homotopy theory of type theories}, Adv. Math. \textbf{337}
  (2018), 1--38, \href {http://arxiv.org/abs/1610.00037}
  {\path{arXiv:1610.00037}}, \href
  {http://dx.doi.org/10.1016/j.aim.2018.08.003}
  {\path{doi:10.1016/j.aim.2018.08.003}}.

\bibitem[Kra15]{kraus:thesis}
Nicolai Kraus, \emph{Truncation levels in homotopy type theory}, Ph.D. thesis,
  University of Nottingham, 2015,
  \url{http://eprints.nottingham.ac.uk/28986/1/thesis.pdf}.

\bibitem[KS17]{kraus-sattler:space-valued-diagrams}
Nicolai Kraus and Christian Sattler, \emph{Space-valued diagrams,
  type-theoretically}, extended abstract, 2017, \href
  {http://arxiv.org/abs/1704.04543} {\path{arXiv:1704.04543}}.

\bibitem[Lum10]{lumsdaine:thesis}
Peter~LeFanu Lumsdaine, \emph{Higher categories from type theories}, Ph.D.
  thesis, Carnegie Mellon University, 2010,
  \url{http://peterlefanulumsdaine.com/research/Lumsdaine-2010-Thesis.pdf}.

\bibitem[Lum11]{lumsdaine:funext-blog}
\bysame, \emph{Strong functional extensionality from weak}, blog post, December
  2011,
  \url{https://homotopytypetheory.org/2011/12/19/strong-funext-from-weak/}.

\bibitem[LW15]{lumsdaine-warren:local-universes}
Peter~LeFanu Lumsdaine and Michael~A. Warren, \emph{The local universes model:
  an overlooked coherence construction for dependent type theories}, ACM Trans.
  Comput. Log. \textbf{16} (2015), no.~3, Art. 23, 31, \href
  {http://arxiv.org/abs/1411.1736} {\path{arXiv:1411.1736}}, \href
  {http://dx.doi.org/10.1145/2754931} {\path{doi:10.1145/2754931}}.

\bibitem[Mak95]{makkai:folds}
Michael Makkai, \emph{First order logic with dependent sorts, with applications
  to category theory}, \url{http://www.math.mcgill.ca/makkai/folds/}, 1995.

\bibitem[ML84]{martin-lof:bibliopolis}
Per Martin-L{\"o}f, \emph{Intuitionistic type theory}, Studies in Proof Theory.
  Lecture Notes, vol.~1, Bibliopolis, Naples, 1984.

\bibitem[ML10]{martin-lof:spreads-seminars}
\bysame, \emph{Spreads and choice sequences in type theory}, March--April 2010,
  Seminars given at Stockholm University.

\bibitem[Pal19]{palmgren:cwfs-and-folds}
Erik Palmgren, \emph{Categories with families and first-order logic with
  dependent sorts}, Ann. Pure Appl. Logic \textbf{170} (2019), no.~12, 102715,
  75, \href {http://arxiv.org/abs/1605.01586} {\path{arXiv:1605.01586}}, \href
  {http://dx.doi.org/10.1016/j.apal.2019.102715}
  {\path{doi:10.1016/j.apal.2019.102715}}.

\bibitem[RB09]{radulescu-banu}
Andrei R{\u{a}}dulescu-Banu, \emph{Cofibrations in homotopy theory}, preprint,
  2009, \href {http://arxiv.org/abs/math/0610009} {\path{arXiv:math/0610009}}.

\bibitem[Shu15]{shulman:inverse-diagrams}
Michael Shulman, \emph{Univalence for inverse diagrams and homotopy
  canonicity}, Math. Structures Comput. Sci. \textbf{25} (2015), no.~5,
  1203--1277, \href {http://arxiv.org/abs/1203.3253} {\path{arXiv:1203.3253}},
  \href {http://dx.doi.org/10.1017/S0960129514000565}
  {\path{doi:10.1017/S0960129514000565}}.

\bibitem[Szu14]{szumilo:two-models}
Karol Szumi{\l}o, \emph{Two models for the homotopy theory of cocomplete
  homotopy theories}, Ph.D. thesis, University of Bonn, 2014, \href
  {http://arxiv.org/abs/1411.0303} {\path{arXiv:1411.0303}}.

\bibitem[Ton13]{tonelli}
Simone Tonelli, \emph{Investigations into a model of type theory based on the
  concept of basic pair}, Master's thesis, Stockholm University, 2013,
  supervisors Erik Palmgren and Giovanni Sambin,
  \url{http://kurser.math.su.se/pluginfile.php/16103/mod_folder/content/0/2013/2013_08_report.pdf}.

\bibitem[Tse16]{tsementzis:first-order-logic-with-isomorphism}
Dimitris Tsementzis, \emph{First--order logic with isomorphism}, preprint,
  2016, \href {http://arxiv.org/abs/1603.03092} {\path{arXiv:1603.03092}}.

\bibitem[{Uni}13]{hott:book}
The {Univalent Foundations Program}, \emph{Homotopy type theory: Univalent
  foundations of mathematics}, \url{http://homotopytypetheory.org/book},
  Institute for Advanced Study, 2013.

\bibitem[VAG{\etalchar{+}}]{UniMath}
Vladimir Voevodsky, Benedikt Ahrens, Daniel Grayson, et~al., \emph{{\em
  UniMath}: {Univalent} {Mathematics}}, Available at
  \url{https://github.com/UniMath}.

\bibitem[vdBG11]{garner-berg:types-are-weak}
Benno van~den Berg and Richard Garner, \emph{Types are weak
  {$\omega$}-groupoids}, Proc. Lond. Math. Soc. (3) \textbf{102} (2011), no.~2,
  370--394, \href {http://arxiv.org/abs/0812.0298} {\path{arXiv:0812.0298}},
  \href {http://dx.doi.org/10.1112/plms/pdq026}
  {\path{doi:10.1112/plms/pdq026}}.

\end{thebibliography}

\end{document}